\documentclass[11pt]{article}

\usepackage{enumerate}
\usepackage{xcolor}

\usepackage{mathdots}
\usepackage{amsmath,amssymb,amsthm}
\usepackage{url}
\usepackage{mathtools}
\usepackage{epsfig,amsfonts,graphicx,color}
\usepackage{a4wide}
\newtheorem{Theorem}{Theorem}
\newtheorem{Lemma}{Lemma}[section]
\newtheorem{Remark}{Remark}[section]
\newtheorem{Proposition}{Proposition}[section]
\newtheorem{Corollary}{Corollary}[section]
\newtheorem{Fact}{Fact}
\newtheorem{Example}{Example}[section]

\newtheorem{Definition}{Definition}

\newcommand{\be}{\begin{equation}}
\newcommand{\ee}{\end{equation}}

\newcommand{\tr}{\mathrm{tr}\,}

\newcommand{\R}{\mathbb{R}}
\newcommand{\Id}{\operatorname{Id}}

\newcommand{\ddd}{\mathrm{d}\,}

\newcommand{\LC}{\mbox{\tiny{\sf LC}}}

\newcommand{\Pthree}{{\mathcal B}}
\newcommand{\Pone}{{\mathcal A}}

\newcommand{\pd}[2]{\frac{\partial#1}{\partial#2}}

\newcommand{\const}{\operatorname{const}}

\newcommand{\weg}[1]{}

\title{Applications of Nijenhuis geometry III:    Frobenius pencils and compatible non-homogeneous Poisson structures}
\author{Alexey V. Bolsinov\footnote{ School of Mathematics,
 Loughborough University,
 LE11 3TU, UK \ \ 
 \quad {\tt A.Bolsinov@lboro.ac.uk} } \quad
\& \quad  Andrey Yu. Konyaev\footnote{Faculty of Mechanics and Mathematics, Moscow State University and Moscow Center for Fundamental and Applied Mathematics, 119992, Moscow Russia
 \ \ \quad {\tt  maodzund@yandex.ru}} \quad \& \quad Vladimir S. Matveev\footnote{
Institut f\"ur Mathematik, Friedrich Schiller Universit\"at Jena,
07737 Jena Germany  \ \ \quad {\tt  vladimir.matveev@uni-jena.de}} 
}  
\date{}
\begin{document}
\maketitle

%

\begin{abstract} 
We consider multicomponent local Poisson structures of the form $\mathcal P_3 + \mathcal P_1$, 
 under the assumption that  the third order term 
  $\mathcal P_3$ is Darboux-Poisson and non-degenerate,  and study the Poisson compatibility of two such structures.
We give an algebraic interpretation of this problem  in terms of  Frobenius algebras and reduce it  to classification of Frobenius pencils, i.e.,  linear families of Frobenius algebras.  Then, we completely describe and  classify Frobenius pencils under minor genericity conditions.  In particular we  show that  each  
 Frobenuis pencil is a subpencil of a certain {\it maximal} pencil.  These maximal pencils are uniquely determined by some combinatorial object, a
directed rooted in-forest with edges and vertices labeled by  numerical marks.
 They are also naturally related to certain 
  pencils  of Nijenhuis operators. We show that common Frobenius coordinate systems admit an elegant invariant description in terms of the corresponding Nijenhuis pencils.

{\bf MSC classes:}  37K05, 37K06, 37K10,  37K25,   37K50,  53B10,  53A20,  53B20,  53B30, 53B50, 53B99,  53D17	
\end{abstract} 

\newpage


\section{Introduction}\label{sect:1}

\subsection{Foreword}\label{subsect:1.1} 

Nijenhuis operator is a (1,1)-tensor field $L= L^i_j $ on an $n$-dimensional  manifold $M$ such that its Nijenhuis torsion vanishes.   Nijenhuis geometry,  as initiated in  \cite{nij1} (where  all necessary definitions can also be found) and further developed in \cite{NijenhuisIII,NijenhuisAppl1,konyaev},  studies Nijenhuis operators and their applications. There are many topics in mathematics and mathematical physics in which Nijenhuis operators appear naturally; this paper is devoted to the study of $\infty$-dimensional compatible Poisson brackets of type $\mathcal P_3 + \mathcal P_1$, where the lower index $i$ indicates the order of the homogeneous bracket $\mathcal P_i$ (the necessary definitions will be given in Section \ref{subsect:1.2}).  Nijenhuis geometry allows us to reformulate the initial problem, originated from 
 mathematical physics, first  into the language of algebra and then into  the language of  differential geometry and finally solve it using the machinery  of differential geometry in combination with that of algebra. Translating back the results gives a full description of (nondegenerate) compatible Poisson brackets of type $\mathcal P_3 + \mathcal P_1$ such that  $\mathcal P_3$   is Darboux-Poisson.

{\bf Acknowledgements.} This research was supported through the programme `Research in Pairs'  by the Mathematisches Forschungsinstitut Oberwolfach in 2021. The authors are grateful to R.~Vitolo, P.~Lorenzoni, J.~Draisma, F.~Cl\'ery and especially to E.~Ferapontov  for their valuable comments, suggestions and explanations.  The research of V.M.  was supported by  DFG, grant number MA 2565/7.  The authors thank the reviewer  for careful reading of the manuscript and very useful remarks.

\subsection{Mathematical setup} \label{subsect:1.2}

The construction below is a special case of the general approach suggested in   \cite{gd}. For   $n=1$, the construction can be found in \cite{gdi}, see also \cite{cas, olver, vit}.

We work in an open  disc $U\subset \mathbb{R}^n$  with coordinates $u^1,...,u^n$. Our constructions are invariant with respect to  coordinate changes so one may equally think of $(u^1,...,u^n)$ as a coordinate chart on a smooth  manifold $M$.

Consider the jet bundles (of curves) over  $U$. Recall that for a point $\mathsf p\in{U}$, the  $k^{\textrm{th}}$ jet  space $J^k_{\mathsf p}U$ at this point is an equivalence class of smooth curves  $c:(-\varepsilon,\varepsilon) \to U$  such that $c(0)=\mathsf p$. The parameter of the curves $c$ will always be denoted by $x$.  The equivalence relation is as follows: two  curves are equivalent if they  coincide at $c(0)$ up to terms of order $k+1$.

For example, for $k=0$  the space $J^0_{\mathsf p}U$ contains only one element  and the definition of  $J_{\mathsf p}^1U$ coincides with one of the  standard definitions of the tangent space $T_{\mathsf p}U$.

It is known that  $J^k_{\mathsf p}U$ is naturally equipped with the structure of a vector space of dimension $n\times k$ with coordinates  denoted  by 
\begin{equation} 
\label{eq:intro:2}
(u^1_x,...,u^n_x,u^1_{x^2},...,u^n_{x^2},...,u^1_{x^k},...,u^n_{x^k}).
\end{equation}
  
Namely,  a curve $c(x)= (u^1(x),...,u^n(x))$ with $c(0)={\mathsf p}$ viewed as an element of $J_{\mathsf p}^kU$ has coordinates 
\begin{equation} 
\label{eq:intr:1} 
\begin{array}{ll} 
&\left(u^1_x,...,u^n_x,u^1_{x^2},..., u^n_{x^2},...,u^1_{x^k},...,u^n_{x^k}\right) \\
=&\left(\tfrac{\ddd}{\ddd x}(u^1),...,\tfrac{\ddd}{\ddd x}(u^n),\tfrac{\ddd^2}{\ddd x^2}(u^1),..., \tfrac{\ddd^2}{\ddd x^2}(u^n),..., \tfrac{\ddd^k}{\ddd x^k}(u^1),...,\tfrac{\ddd^k}{\ddd x^k}(u^n)\right)_{|x=0}.
\end{array} 
\end{equation}

We denote by $J^kU$ the union $  \bigcup_{{\mathsf p}\in U} J^k_{\mathsf p}U$. It   has a natural structure of a $k\times n$-dimensional vector bundle over $U$. The coordinates  $(u^1,...,u^n)$ on $U$  and \eqref{eq:intro:2} on $J^k_{\mathsf p}U$ generate  a coordinate system 
$$
(u^1,...,u^n, u^1_x,...,u^n_x,u^1_{x^2},...,u^n_{x^2},...,u^1_{x^k},...,u^n_{x^k} ) 
$$
on $J^kU$ adapted to the bundle structure. Any $C^\infty$ curve  $c:[a,b] \to U, \ x\mapsto (u^1(x),...,u^n(x)) $  naturally lifts  to a  curve $\widehat c$ on  $J^kU$ by  
\begin{equation}
\label{eq:intro:3}
\widehat c:[a,b]\to J^kU\ , \  \     x_0\mapsto \left(u^1,...,u^n,  \tfrac{\ddd}{\ddd x}(u^1),...,\tfrac{\ddd}{\ddd x}(u^n),..., \tfrac{\ddd^k}{\ddd x^k}(u^1),...,\tfrac{\ddd^k}{\ddd x^k}(u^n)\right)_{|{x=x_0}}. 
\end{equation}

Next, for every ${\mathsf p}\in U$ denote by ${\Pi}[J^k_{\mathsf p}U]$ the algebra  of polynomials in variables  \eqref{eq:intro:2} on $J^k_{\mathsf p}U$. It has a natural structure of an infinite-dimensional vector bundle over $U$. Let $\mathfrak{A}_k$ denote the algebra of $C^\infty$-smooth sections of the  bundle ${\Pi}[J^k_{\mathsf p}U]$.  Notice that we have natural inclusion $\mathfrak{A}_k \subset \mathfrak{A}_{k+1}$  and set $\mathfrak{A}=\bigcup_{k=0}^\infty \mathfrak{A}_k$.   In simple terms, the elements of $\mathfrak{A}$ are finite sums  of  finite products of coordinates 
\begin{equation} 
\label{eq:intro:2bis}
(u^1_x,...,u^n_x,u^1_{x^2},...,u^n_{x^2},...,u^1_{x^k},...,u^n_{x^k},... ) 
\end{equation}
with coefficients being $C^\infty$-functions on $U$. The summands in this sum, i.e., terms of the form 
$a_{i_1...i_n}^{j_1...j_n}(u) (u^1_{x^{i_1}})^{j_1}... (u^n_{x^{i_n}})^{j_n}  $ with $a_{i_1...i_n}^{j_1...j_n}(u)\not\equiv 0$
will be called {\it differential monomials}. The   {\it differential degree} of such a differential monomial is the number $i_1j_1+ i_2j_2+...+i_nj_n$. For example, $f(u) u^1_{x^2} (u^2_{x})^2$ has differential degree $1\times 2+ 2\times 1=4$. 
 Differential degree of an element of $\mathfrak{A}$ is the maximum of the differential degrees of its  differential monomials, it is a nonnegative integer number.  Elements of $\mathfrak{A}$ will be called {\it differential polynomials}.

Generators of this algebra are coordinates  $u^i_{x^j}$ and functions on $U$. Every element of  $\mathfrak{A}$ can be obtained from finitely many  generators using finitely  many summation and multiplication operations.

The following two linear mappings  will be important for us. The first one,  called the {\it total $x$-derivative} and denoted by $D$ (another standard notation used in literature is $\tfrac{d}{dx}$) is  defined as follows. One requires  that $D$  satisfies the Leibnitz rule and then defines it on the generators of $\mathfrak{A}$, i.e., on functions $f(u)$   and coordinates \eqref{eq:intro:2bis}, by setting
$$
D(f)= \sum_{i=1}^n \frac{\partial f}{\partial u^i} u^i_x \  , \   \ D ( u^i_{x^j})= u^i_{x^{j+1}}.
$$ 
Clearly, the operation $D$ increases the differential degree by one at most.

Next, denote by $\tilde{\mathfrak{A}}$ the quotient algebra $\mathfrak{A}/{D(\mathfrak{A})}$.  The tautological projection $\mathfrak A \to  \tilde{\mathfrak{A}}$ is traditionally denoted by $\mathcal H \mapsto \int \mathcal H dx \in \tilde{\mathfrak{A}}$.
In simple terms it means that we think that two differential polynomials $\mathcal{H}, \bar {\mathcal{H}}$  are  equal, if their difference is a total derivative of a differential polynomial.

Note that by construction, the operation $D$ has the following remarkable property, which  explains its name and  also 
the  	notation  $\tfrac{d}{dx}$ used for $D$ sometimes in  literature. For any curve $c:[a,b]\to U$ whose lift \eqref{eq:intro:3} will be denoted by $\widehat c$ and for any  element  $\mathcal{H} \in  {\mathfrak{A}}$ we have: 
\begin{equation}
\label{eq:intro:4}
\tfrac{\ddd}{\ddd x}\left(	\mathcal{H}(\widehat c)\right)= \left(D\mathcal{H}\right)(\widehat c).
\end{equation}

The second mapping is the mapping from $\mathfrak{A}$ to an $n$-tuple  of elements of $\mathfrak{A}$. 
The mapping will be denoted by $\delta$ and will be called {\it the  variational derivative}. 
Its $i^{\textrm{th}}$ component will be denoted by $\tfrac{\delta}{\delta u^i}$ and for an element $\mathcal{H}\in \mathcal{A}$  it is given by the Euler-Lagrange formula:
$$
 \frac{\delta \mathcal{H}}{\delta u^i}= \sum_{k =0}^\infty (-1)^k D^k\left(\frac{\partial \mathcal{H}}{\partial u_{x^k}^i}\right)
$$
 (only finitely many elements in the sum are different from zero so the result is again a differential polynomial).  It is known, see e.g. \cite{gdi}, that for an element $\mathcal{H}\in \mathfrak{A}$ we have $\delta \mathcal{H} =0$ if and only if $\mathcal H$ is a total $x$-derivative. Then, we again see that the variational derivative does not depend on the choice of  differential polynomial  in the equivalence class  $\mathcal{H} \subset \mathfrak{A}$.  Then, the mapping $\delta$ induces a well-defined mapping on $\tilde{\mathfrak{A}}, $ which will be denoted  by the same letter $\delta$.  One can think of $\delta \mathcal{H}$ as a covector  with entries from $\tilde {\mathfrak{A}}$, because the  transformation rule of its entries  under the change of $u$-coordinates is a natural generalisation of  the transformation rule for (0,1)-tensors.

Following \cite{dn, dn2}, let us define a (homogeneous, nondegenerate) Poisson bracket of order $1$. We choose  a  contravariant flat  
metric $g=g^{ij}$ of any signature whose   Levi-Civita connection will be denoted by 
$\nabla= (\Gamma_{jk}^i)$. Next,  
 consider the following operation 
$\Pone_g: \tilde{\mathfrak{A}} \times \tilde{\mathfrak{A}} \to \tilde{\mathfrak{A}}$: for two elements $\mathcal{H}, {\bar {\mathcal{H}}} \in \tilde{\mathfrak{A}} $ we set
\begin{equation} 
\label{eq:intro:5}
\Pone_g( \mathcal{H}, \bar{\mathcal{H}}) =   \int \tfrac{\delta \bar{\mathcal{ H}}}{\delta u^\alpha}  
\left(  g^{ \alpha\beta} D \left(\tfrac{\delta \mathcal{ H}}{\delta u^\beta}\right)-\Gamma^{\alpha\beta}_\gamma  \tfrac{\delta \mathcal{ H}}{\delta u^\beta}
 u^\gamma_x\right) \ddd x. 
\end{equation}
In the formula above and later in the text,  we sum over repeating indices and assume 
 $\Gamma^{is}_{j} = \Gamma^s_{pj} g^{pi}$. The components $\Gamma^{is}_{j}$ will be called {\it contravariant Christoffel symbols}, when we speak about different metrics we always  raise the index {\it by the own metric}.   A common way to write the operation $\Pone_g$ which we also will use in our paper assumes applying it to  $\tfrac{\delta \mathcal{ H}}{\delta u^\beta}$ and multiplication with 
$\tfrac{\delta \bar{\mathcal{ H}}}{\delta u^\alpha}$ (and  of course summation and projection to $\tilde{\mathfrak A}$):   
\begin{equation}
\label{eq:intro:6}
\Pone_g =  g^{\alpha\beta} D  - \Gamma^{\alpha\beta}_\gamma  u^\gamma_x.  
\end{equation}

It is known, see e.g. \cite{dn,dn2,dub1}, that the operation $\Pone_g$ given by  \eqref{eq:intro:5} defines a Poisson bracket on $\tilde {\mathfrak{A}}$, that is, it is skew-symmetric and satisfies the Jacobi identity.  Moreover, one can show that the operation constructed by $g$  and $\Gamma$  via  \eqref{eq:intro:5} defines a Poisson bracket if and only if $g$ is flat, that is, its curvature is zero, and $\Gamma_{jk}^i$ is the Levi-Civita connection of $g$.  
It is also known that  the construction \eqref{eq:intro:5} does not depend on the coordinate system on $U$.

Next, let us define a (nondegenerate, homogeneous)  Darboux-Poisson structure of  order $3$. We choose a nondegenerate contravariant flat metric $h= h^{ij}$ of arbitrary signature and define the operation $\Pthree_h:\tilde{\mathfrak{A}} \times \tilde{\mathfrak{A}} \to \tilde{\mathfrak{A}}$ by the formula:  
\begin{equation}
\label{intro:third}
    \begin{aligned}
    \Pthree_h & =  
		h^{\alpha q} \left(\delta^p_q D - \Gamma_{qm}^p u^m_x\right) \left(\delta^r_p D - \Gamma_{pk}^r u^k_x\right) \left(\delta^\beta_r D - \Gamma_{rs}^\beta u^s_x\right)=
  \\  &=  h^{\alpha \beta} D^3 - 3 h^{\alpha q} \Gamma^{\beta}_{q s} u^s_{x}D^2 + \\
    & + 3 \Bigg(  h^{\alpha q}\Big(\Gamma^p_{qs} \Gamma^{\beta}_{pr} - \pd{\Gamma^{\beta}_{qs}}{u^r}\Big)u^s_{x} u^r_{x} - h^{\alpha q}\Gamma^{\beta}_{qs} u^s_{x^2} \Bigg) D + \\
    & + \Bigg( h^{\alpha q} \Big( 2 \Gamma^a_{qs} \pd{\Gamma^{\beta}_{ar}}{u^p} + \pd{\Gamma^a_{qs}}{u^r}\Gamma^{\beta}_{ap} - \Gamma^a_{qs} \Gamma^b_{ar} \Gamma^{\beta}_{bp} - \frac{\partial^2 \Gamma^{\beta}_{qs}}{\partial u^r \partial u^p}\Big) u^s_{x}u^r_{x}u^p_{x} + \\
    & + h^{\alpha q} \Big( 2 \Gamma^a_{qs} \Gamma_{ar}^{\beta} + \Gamma^a_{qr} \Gamma^{\beta}_{as} - 2 \pd{\Gamma^{\beta}_{qr}}{u^s} - \pd{\Gamma^{\beta}_{qs}}{u^r}\Big) u^s_{x} u^r_{x^2} - h^{\alpha q} \Gamma^{\beta}_{qs} u^s_{x^3} \Bigg)
    \end{aligned}
\end{equation}
In the formula we have used the same conventions as above, i.e., assume summation over repeating indices. Moreover, similar to formula \eqref{eq:intro:6}, we  did not write $\mathcal{H}, {\bar {\mathcal{H}}}$ in the formula.   They are assumed there as follows: the differential operator  \eqref{intro:third}  is applied to $\tfrac{\delta \mathcal{ H}}{\delta u^\alpha}$,  the result is 
 multiplied by $\tfrac{\delta \bar{\mathcal{ H}}}{\delta u^\beta}$, and then we perform summation  with respect to  the repeating indices  $\alpha, \beta$.

As in the case of order 1, the operation   $\Pthree_h$ given by  \eqref{intro:third}
 defines a Poisson bracket on $\tilde {\mathfrak{A}}$. The construction of this Poisson bracket   is independent on the choice of coordinate system on $U$. 
However, in contrast to the case of order 1, the form \eqref{intro:third} is not the most general form for a local Poisson bracket on $\tilde {\mathfrak{A}}$ of order  $3$. 
In fact, the word {\it Darboux} indicates that in a certain coordinate system (flat coordinate system for $h$ in our case) 
 the coefficients  of the Poisson structure are constants\footnote{The terminology `Darboux-Poisson'
is motivated by \cite{doyle}.}. In this {\it Darboux} coordinate system the Christoffel symbols $\Gamma^i_{jk}$ are all zero and  formula \eqref{intro:third} reduces to\footnote{In fact, \eqref{intro:third} is just the formula \eqref{eq:intro:7} rewritten in an arbitrary (not necessarily `flat') coordinate system.} 
\begin{equation}
\label{eq:intro:7}
\Pthree_h(\mathcal{H}, \bar{\mathcal{H}}) =   \tfrac{\delta \bar{\mathcal{ H}}}{\delta u^\beta}  h^{\alpha \beta} D^3\left(\tfrac{\delta {\mathcal{ H}}}{\delta u^\alpha}\right),
\end{equation}
with the components $h^{\alpha\beta}$ of the metric $h$ being constants. 

Poisson structures $\mathcal P_1$ of order $1$ are always Darboux-Poisson, but there are examples, see e.g. \cite{fpv,fpv1, pot},  of Poisson structures $\mathcal P_3$ of order $3$ which are not Darboux-Poisson.

Similar to the  finite-dimensional case,  a Poisson structure $\mathcal{P}$  and  choice of a `Hamiltonian' $\mathcal{H}\in \tilde{\mathfrak{A}}$  allows one  to define the Hamiltonian flow, which in our setup  is a system of $n$ PDEs on $n$ functions $u^i(t,x)$ 
of two variables, $t$ and $x$. It is given by:
\begin{equation}
\label{eq:intro:8}
\tfrac{\partial u^\beta }{\partial t} = \mathcal{P}^{\alpha\beta}\left(\tfrac{\delta {\mathcal{ H}}}{\delta u^\alpha}\right).
\end{equation} 
For example, in the case of the Poisson structure  \eqref{eq:intro:5}  for a Hamiltonian of degree $0$ (i.e., for a function $H$  on $U$) the Hamiltonian flow is  given by 
\begin{equation}
\label{eq:intro:9}
 \tfrac{\partial u^\beta }{\partial t} = g^{ \beta\alpha} \tfrac{ \partial^2 H}{\partial u^\alpha \partial u^\gamma} u^\gamma_x -\Gamma^{\beta\alpha}_\gamma \tfrac{ \partial H}{\partial u^\alpha }   u^\gamma_x=
\tfrac{\partial u^\gamma}{\partial x}\nabla^\beta\nabla_\gamma H.
\end{equation}
Such systems of PDEs are called {\it  Hamiltonian systems of hydrodynamic type}.

In our paper we study compatibility of nonhomogeneous Poisson structures of type $\mathcal P_3+\mathcal P_1$ such that the part of order $3$ is Darboux-Poisson. That is, we have 4 nondegenerate  Poisson structures: $\Pone_g$ and $\Pone_{\bar g}$ constructed by flat metrics $g$ and $\bar g$ by \eqref{eq:intro:5}, and $\Pthree_h$ and $\Pthree_{\bar h}$ constructed by flat metrics $h$ and $\bar h$ by 
\eqref{intro:third}. We assume that $\Pone_g + \Pthree_h$  and  $\Pone_{\bar g} + \Pthree_{\bar h}$ are (nonhomogeneous) Poisson structures and  ask the question when these structures   are compatible in the sense that any of their linear combinations is a Poisson structure \cite{magri}. Since it is automatically skew-symmetric, the compatibility  is equivalent to the Jacobi identity for each linear combination of $\Pone_g + \Pthree_h$ and $\Pone_{\bar g} + \Pthree_{\bar h}$.

The meaning of  the word `nondegenerate'  relative to the Poisson structures under discussion is as follows: the metrics $g,\bar g, h, \bar h$ which we used to construct them are nondegenerate, i.e., they are given by matrices with nonzero determinant. Additional nondegeneracy condition, natural from the viewpoint of mathematical physics, is as follows: the operators $R_h= \bar h h^{-1}$ and $R_g= \bar g g^{-1}$ have $n$ different eigenvalues.  Under these conditions, we solve the problem completely: we find explictly all pairs of such Poisson structures.

Let us comment on the assumption that  $\mathcal P_3$ is {\it Darboux-Poisson}. The compatibility of two geometric Poisson brackets ${\mathcal P}_3 +\mathcal P_1$ and $\bar{\mathcal P}_3 + \bar{\mathcal P}_1$ amounts to a highly overdetermined PDE system which is expected to  imply additional conditions on the third order parts, from which {\it Darboux-Poisson} is a natural  candidate.     
Indeed, in the literature  we have not found any example of compatible Poisson brackets  ${\mathcal P}_3 +\mathcal P_1$ and $\bar{\mathcal P}_3 + \bar{\mathcal P}_1$ such that  
  $\mathcal P_3$ and $\bar{\mathcal P}_3$ are not Darboux-Poisson and nonproportional. Even in the case when $\bar {\mathcal P}_3=0$, only  few  
 examples are known, namely the compatible brackets for WDVV system from   \cite{FGM} (this example is three-dimensional and moreover, 
$\mathcal P_1=0$) and a family of compatible  brackets ${\mathcal P}_3 +\mathcal P_1$ and $\bar{\mathcal P}_1$  constructed in \cite{lor} in dimensions 1 and 2.

Our  main motivation came  from the theory of integrable systems, in which   many famous integrable systems have been constructed  and analysed using compatible Poisson brackets. Those of the form $\mathcal P_3 + \mathcal P_1$, described and classified in the present paper,  
generalise compatible Poisson brackets related to  KdV,  Harry Dym,  Camassa-Holm and Dullin-Gottwald-Holm equations.
The applications of such new brackets will be developed  in a series of  separate papers, of which the first one has already appeared,  \cite{NijenhuisAppl4}. In this paper we have generalised the above equations  for an arbitrary number of components, and have constructed new integrable PDE systems  that have no low-component analogues. Note that in  \cite{NijenhuisAppl4} we used the simplest 
  pencil  of compatible Poisson structures of type $\mathcal P_3 + \mathcal P_1$ (so-called AFF-pencil from Section \ref{subsect:2.3}); more complicated pencils will lead to more new families of integrable multicomponent PDEs.  

 Let us also comment on a more physics-oriented approach to the  construction above, see e.g. \cite{doyle}. Physicists often 
view  $x$ as  a space coordinate, and $(u^1,...,u^n)$ as field coordinates. In the simplest situation, 
 the values $u^1,...,u^n$ at $x$ may  describe some physical values (e.g., pressure, temperature, charge, density, momenta). The total energy of the system  is the integral over the $x$ variable of some differential polynomial in $u^1,...,u^n$, and the 
Hamiltonian functions $\mathcal{H}\in \mathfrak{A}$ have  then the physical meaning of the density of the energy, i.e., of the integrand in the formula $\textrm{{\bf Energy}}(c)= \int \mathcal{H}(\widehat c)dx$.
 Further, it is assumed that the physical system is either periodic in $x$, or one is interested in fast decaying  solutions as $x\to \pm \infty$.  The integration by parts implies then that the differential polynomial is defined up to 
an addition of the total derivative in $x$  which allows one to pass to  $\tilde{\mathfrak{A}}=   \mathfrak{A}/{D(\mathfrak{A})}$. 
The natural analog of the differential of  a function in this setup is the variational derivative $\tfrac{\delta }{\delta u^\alpha}$, and actually the equation \eqref{eq:intro:8} is the natural analog of the finite-dimensional equation $\dot u= X_H$ (where $X_H$ is the Hamiltonian vector field of a function $H$; it is given by $X_H^j= P^{ij}\tfrac{\partial H}{\partial u^i}$ where $P(u)^{ij}$ is the matrix of  the Poisson structure; please note similarity with  \eqref{eq:intro:8}).  
Generally, it is  useful  to keep in mind the physical interpretation and analogy with the finite-dimensional case.



\subsection{Brief description of main results, structure of the paper and conventions} \label{sebsect:1.3}

In this paper we address the following problems:

\begin{itemize}

\item[(A)] {\it Description of compatible pairs, $\Pthree_h + \Pone_g$ and    
$\Pthree_{\bar h} + \Pone_{\bar g}$, of non-homogeneous Poisson brackets in arbitrary dimension $n$.} In    Theorems \ref{thm:frobenius0} and  
\ref{prop:bols13} we 
 give an algebraic interpretation of this problem  in terms of  Frobenius algebras and reduce it to classification of Frobenius pencils, i.e. linear families of Frobenius algebras. We do it under the following nondegeneracy assumption: the (1,1)-tensor  $R_h= \bar hh^{-1}$ (connecting $h$ and $\bar h$) has $n$ different eigenvalues.

\item[(B)] {\it Description and classification of Frobenius pencils.} 
We reduce this purely algebraic problem to a differential geometric one (explicitly formulated in Section \ref{subsect:6.1})
and completely solve it  using geometric methods.    The nondegeneracy assumption is that the (1,1)-tensor  $R_g= \bar gg^{-1}$ (connecting $g$ and $\bar g$) has $n$ different eigenvalues.  
Namely, we show that each Frobenuis pencil in question is a subpencil of a certain {\it maximal} pencil.   
We explicitly describe all maximal pencils, see Theorems  \ref{thm:frobenius1},  \ref{thm:frobenius2a} and \ref{thm:frobenius2}.

\begin{itemize}\item[(B1)] A generic in a certain sense maximal pencil corresponds to the well-known multi-Poisson structure discovered by M. Antonowitz and A. Fordy in \cite{fordy} and studied by E.~Ferapontov and M.~Pavlov \cite{Ferapontov-Pavlov}, see also 
\cite{fordy2,fordy3,NijenhuisAppl2}.   We refer to  it  as to   Antonowitz-Fordy-Frobenius pencil,   {\it AFF-pencil}. In Theorem \ref{thm:frobenius1} we show that every two-dimensional Frobenius pencil with one additional genericity assumption is contained
in the {\it AFF-pencil}.

\item[(B2)]  Our main result, Theorems \ref{thm:frobenius2a} and \ref{thm:frobenius2}, give a complete   description in the most  general case.    Theorem    \ref{thm:frobenius2a} constructs all maximal Frobenius pencils  using AFF-pencils as building blocks.  Theorem    \ref{thm:frobenius2}  states that  each Frobenuis pencil is a subpencil of a certain {\it maximal} pencil from Theorem \ref{thm:frobenius2a}. 
 These maximal pencils are uniquely determined by some combinatorial data,   directed rooted in-forest  $\mathsf F$ with 
 edges labeled by numbers $\lambda_\alpha$'s and
 vertices labeled by natural numbers whose sum is the dimension of the manifold.  The AFF-pencil corresponds to the simplest case, when $\mathsf F$ consists of a single vertex. To the best of our knowledge, the other Frobenius pencils and the corresponding bi-Poisson structures are new. 

\end{itemize} 

In addition,  we show that common Frobenius coordinate systems  admit an  elegant invariant description in terms of the Nijenhuis pencil $\mathcal L$, see  Theorem \ref{thm:frobenius2a}.

\item[(C)]  {\it Dispersive perturbations of compatible Poisson brackets of hydrodynamic type.}
The general question is as follows: 
given two compatible Poisson structures   $\Pone_g$ and $\Pone_{\bar g}$
of the first order, can one find flat metrics 
$h$ and $\bar h$ such that  
  $\Pthree_h + \Pone_g$ and  $\Pthree_{\bar h} + \Pone_{\bar g}$ are compatible Poisson structures?   
  This passage from a Poisson bracket of hydrodynamic type to a non-homogeneous  Poisson bracket of higher  order is called {\it dispersive perturbation} in literature. 
  We study dispersive perturbations of bi-Hamiltonian structures assuming that the third order terms $\Pthree_h$ and  $\Pthree_{\bar h}$ are Darboux-Poisson.

We describe all such perturbations under the  assumption that both $R_h=\bar h h^{-1}$ and $R_g=\bar g g^{-1}$  have $n$ different eigenvalues, 
and in particular, answer a question from \cite{Ferapontov-Pavlov} on dispersive perturbations of the  AFF-pencil (Remark \ref{QFP}).

\end{itemize}

The results of this paper also have the following unexpected  application.
 It turns out that the {\it diagonal} coordinates for the operator $R_g=\bar g g^{-1}$ are 
 orthogonal separating coordinates for the metric  $g$.
 In  \cite{separation} we show that every   orthogonal  separating coordinates for  a flat metric $g$ of arbitrary signature can be constructed in this way. Namely, we   reduce PDEs that define orthogonal separating coordinates to  those studied in the present paper. This leads us to  an explicit description of all   orthogonal  separating coordinates for metrics of constant curvature and thus solves  a long-standing  and actively studied  problem in mathematical physics, see \cite{separation} for details.

The structure of the paper is as follows.
In Section \ref{sect:2}, we start with basic facts and constructions related to compatibility of homogeneous Poisson structures of order 1 and  3, then give description of compatible non-homogeneous structures $\Pthree_g + \Pone_h$ and $\Pthree_{\bar h} + \Pone_{\bar g}$ in terms of Frobenius algebras (Theorems \ref{thm:frobenius0} and \ref{prop:bols13}),  leading us to the classification problem for the so-called Frobenius pencils.  We conclude this section with an example of AFF-pencil.  The AFF-pencil plays later a role of a building block in our general construction. Moreover, it provides an answer under a minor  nondegeneracy assumption, see Theorem \ref{thm:frobenius1} in Section \ref{sect:3}, where we also discuss a question of Ferapontov and Pavlov.  Theorem \ref{thm:frobenius1} will be proved in  Section \ref{sect:6}. 
 
 In Section \ref{sect:4}  we formulate the  answer to the classification problem in its full  generality. Theorem \ref{thm:frobenius2} (proved in Section \ref{sect:7}) gives a description of Frobenius pencils in the `diagonal' coordinates for $g,\bar g$, and  Theorem \ref{thm:frobenius2a} (proved in Section \ref{sect:8})  describes the  corresponding Frobenius coordinates. In  Section  \ref{subsect:4.3} we discuss the case of two blocks and give explicit formulas, see Theorem \ref{thm:frobenius3}.

All objects in our paper are assumed to be of class $C^\infty$; actually our results show that most of them are necessarily real-analytic.

Throughout the paper we use  $\Pone_g$ and $\Pthree_h$ to denote the Poisson structures of order 1 and 3  given by \eqref{eq:intro:6} and \eqref{intro:third} respectively. Unless otherwise stated,  the metrics we deal with (such as $g$, $h$, $\bar g$, $\bar h$, \dots) are contravariant. 

\section{Non-homogeneous compatible brackets and Frobenius algebras}\label{sect:2}

\subsection{Basic facts and preliminary discussion}\label{subsec:2.1}

Recall that we study  compatibility of two Poisson structures $\Pthree_h+ \Pone_g$ and $\Pthree_{\bar h}+ \Pone_{\bar g}$, constructed by flat metrics $h,\bar h, g, \bar g$; our  goal is to construct all of them.  Recall that by definition it means that for any constants $\lambda, \bar \lambda$ the linear combination  
$\lambda(\Pthree_h+ \Pone_g) + \bar \lambda(\Pthree_{\bar h}+ \Pone_{\bar g})$ is a Poisson structure. Using that $\Pthree$ and $\Pone$ have different orders,  one  obtains (see e.g. \cite{doyle})

\begin{Fact}
\label{fact:1}
Let   $h$, $\bar h$, $g$ and $\bar g$  be flat metrics.  
If $\Pthree_h + \Pone_g$ and $\Pthree_{\bar h} + \Pone_{\bar g}$ are compatible Poisson structures, then  the following holds:
\begin{itemize} 

\item[(i)] $\mathcal A_g$ and $\mathcal A_{\bar g}$ are compatible,
\item[(ii)] $\mathcal B_h$ and $\mathcal B_{\bar h}$ are compatible, 
\item[(iii)] $\mathcal A_g$ and $\mathcal B_{h}$ are compatible (as well as $\mathcal A_{\bar g}$ and $\mathcal B_{\bar h}$).

\end{itemize}
\end{Fact}

This Fact naturally leads us to considering  pencils (\,=\,linear combinations of metrics) $\lambda h+ \bar \lambda \bar h$ and $\lambda g + \bar \lambda \bar g$. We need the following definition:

\begin{Definition}[Dubrovin, {\cite[Definition 0.5]{dub1}}]\label{def:compmetrics}{\rm
Two contravariant flat metrics $g$ and $\bar g$ are said to be \emph{Poisson compatible}, if for each (nondegenerate) linear combination $\widehat g = \lambda g + \bar \lambda  \bar g$,  $\lambda, \bar \lambda \in \mathbb R$, the following two conditions hold:
\begin{enumerate}
\item $\widehat g$ is flat; 
    \item the contravariant Christoffel symbols  for $g$, $\bar g$ and $\widehat g$ are related as
    \begin{equation}
    \label{eq:bols1}
    \widehat \Gamma^{\alpha \beta}_s = \lambda \Gamma^{\alpha \beta}_s + \bar\lambda \bar \Gamma^{\alpha \beta}_s.
    \end{equation}
\end{enumerate}
In this case, the family of metrics $\{ \lambda g + \bar \lambda  \bar g\}_{\lambda,\bar\lambda\in\R}$ is said to be a {\it flat pencil} of metrics. 
}\end{Definition}

 The next fact explains relationship between Poisson compatibility of flat metrics and compatibility of the corresponding Poisson structures.

\begin{Fact}
\label{fact:2}
Let   $h$, $\bar h$, $g$ and $\bar g$  be flat metrics.  Then, the following statements are true: 
\begin{itemize} 
\item[(i)] $\Pone_g$ and $\Pone_{\bar g}$ are compatible if and only if $g$ and $\bar g$ are Poisson compatible.
\item[(ii)] If $\Pthree_h$ and $\Pthree_{\bar h}$ are compatible, then $h$ and $\bar h$ are Poisson compatible.
\item[(iii)] If $\Pthree_h$ and $\Pone_g$ are compatible, then $h$ and $g$ are Poisson compatible. 
\end{itemize}  
\end{Fact}

The (i)-part of Fact \ref{fact:2} is in \cite{dn}, see also \cite{Fer2001, mok2,mok1}.  In view of formula \eqref{eq:intro:6}, the two conditions from Definition \ref{def:compmetrics} are nothing else but a geometric reformulation of the compatibility condition for Poisson structures of order one,  which explains the name {\it Poisson compatible}.  The (ii)-part is an easy corollary of  \cite[Theorem 3.2]{doyle}, see also proof of Theorem \ref{thm:frobenius1} below.
The (iii)-part follows from \cite[Theorem 2.2]{Andrey}.  

Notice that unlike the case of Poisson structures of order 1,  not every pair of  Poisson compatible metrics $h$ and $\bar h$  (resp. $h$ and $g$) leads to compatible Poisson structures of higher order $\Pthree_h$ and $\Pthree_{\bar h}$ as in (ii)    (resp. $\Pthree_h$ and $\Pone_g$ as in (iii)).  Some extra conditions are required. These conditions will be explained below in Fact \ref{fact:4} (for $h$ and $g$ leading to compatible $\Pthree_h$ and $\Pone_g$) and Theorem \ref{thm:frobenius1}  (for $h$ and $\bar h$ leading to compatible $\Pthree_h$ and $\Pthree_{\bar h}$).  

Let us   also recall the relation of compatible metrics to Nijenhuis geometry:

\begin{Fact}[see \cite{Fer2001,mok2,mok1}] \label{fact:3} 
If $g$ and $\bar g$ are Poisson-compatible, then 
the (1,1)-tensor  $R= \bar g g^{-1}$ is a Nijenhuis operator. Moreover, if $g$ is flat, $R$ is a nondegenerate Nijenhuis operator with $n$ different eigenvalues, and  $\bar g:= Rg$ is flat, then $\bar g$ is Poisson-compatible with $g$.
 \end{Fact}

As already explained,   the condition that  $\Pthree_h+ \Pone_g$ is a Poisson structure is a nontrivial geometric condition on the flat metrics $h$ and $g$, stronger than Poisson compatibility in the sense of Definition \ref{def:compmetrics}.  This  condition  was studied in literature (see e.g. \cite{balnov}) and it was observed that the compatibility of homogeneous  Poisson structures of order $3$ and $1$ is sometimes related to certain algebraic structure.
 In our case, under the assumption that $\Pthree_h$ is Darboux-Poisson,  the algebraic structure which pops up naturally is {\it Frobenius algebra}.

\begin{Definition}{\rm 
Let $(\mathfrak a, \star)$ be an  $n$-dimensional commutative associative algebra  over $\mathbb R$ endowed with   
 a nondegenerate symmetric bilinear form  $b(\, , \,)$. The pair $\bigl( (\mathfrak a, \star), b\bigr)$ is called a \emph{Frobenius algebra}, if  
 \begin{equation}
 \label{frobenius}
b(\xi \star \eta, \zeta) = b(\xi, \eta \star \zeta), \quad \mbox{for all } \xi, \eta, \zeta \in \mathfrak a.    
\end{equation}
The form $b$ is then  called a \emph{Frobenius form}. 
}\end{Definition}

Notice that we do not assume that $\mathfrak a$ is unital which makes our version slightly more general than the one used in the theory of Frobenius manifolds (see e.g. \cite{dub1}), or in certain branches of Algebra.  The bilinear form $b$ may have any signature.

 Condition   \eqref{frobenius} is  linear  in $b$, so all Frobenius forms (if we allow some of them to be degenerate) on a given commutative associative algebra form a vector space.

Fix a basis $e^1, \dots, e^n $ in $\mathfrak a$. Below we will interpret $\mathfrak a$ as the dual $(\R^n)^*$ and for this reason we interchange lower and upper indices.  Consider the structure constants $a_k^{ij}$ defined by $e^i \star e^j = a_k^{ij} e^k$ and coefficients $b^{ij}:=b(e^i, e^j)$ of the Frobenius form $b$. The algebra $\mathfrak a$ is Frobenius if and only if $a_k^{ij}$ and $b^{ij}$ satisfy the following conditions:
\begin{equation}\label{constants}
    \begin{aligned}
    & a_k^{ij} = a_k^{ji} \quad\quad\quad\ \  \text{(commutativity),} \\
    & a_{\alpha}^{ij} a^{\alpha r}_k = a_k^{i \alpha} a_{\alpha}^{jr} \quad \text{(associativity),} \\
    & b^{\alpha r} a_{\alpha}^{ij} = b^{i \alpha} a_{\alpha}^{jr} \quad \text{(Frobenius condition)}.
    \end{aligned}
\end{equation}
The dual $\mathfrak a^*$ has a natural structure of an affine space $\R^n$ with $u^i\simeq e^i$ being coordinates on $\mathfrak a^*\simeq \R^n$. Thus, on $\mathfrak a^*$ we can introduce the contravariant metric
$g^{\alpha\beta}(u) = b^{\alpha\beta} + a_s^{\alpha\beta} u^s$
which is known to be flat (e.g. \cite[Lemma 4.1]{Andrey}; the result also follows from  \cite{balnov}).
  What is special here is not the metric $g$ itself, but the coordinate system $u^1,\dots, u^n$ which establishes a relationship between $g$ and the Frobenius algebra $\mathfrak a$. This leads us to

\begin{Definition}\label{def:frobcoord}
{\rm
Let $g$ be a flat metric.   We say that $u^1,\dots, u^n$ is a \emph{Frobenuis coordinate system} for $g$ if 
\begin{equation}
\label{eq:Frobcoord}
g^{\alpha\beta} (u) = b^{\alpha\beta} + a_s^{\alpha\beta} u^s,
\end{equation}
where $a_s^{\alpha\beta}$ are structure constants of a certain Frobenius algebra $\mathfrak a$ and $b=(b^{\alpha\beta})$ is a (perhaps degenerate)
Frobenius form for $\mathfrak a$  .
}\end{Definition}

Frobenius coordinates possess the following important property that can be easily checked.

\begin{Fact}[see \cite{balnov} and \cite{Andrey}]
\label{fact:4}
Let $g$ be a contravariant metric and  $u^1, \dots, u^n$ a coordinate system. The following two conditions are equivalent:
\begin{enumerate}
\item  In coordinates $u^1, \dots, u^n$,  the contravariant Christoffel symbols $\Gamma^{\alpha \beta}_s$ of $g$ are constant and symmetric in upper indices. 
 
\item   $u^1, \dots, u^n$ are Frobenius coordinates, i.e.,  $g$ is given by \eqref{eq:Frobcoord}.
\end{enumerate}
If either of these conditions holds, then $g$ is flat and $\Gamma^{\alpha\beta}_s = -\frac{1}{2} \frac{\partial g^{\alpha\beta}}{\partial s}$.
\end{Fact}

The relation of Frobenius coordinate systems to our problem is established by the following remarkable and fundamental  statement:

\begin{Fact}[$\mbox{\cite[Theorem 2.2]{Andrey}}$]
\label{fact:5}  
Let $g$ and $h$ be two flat metrics.
Then  $\Pthree_h+\Pone_g$ is  Poisson 
    if and only if there exists  a coordinate system $u^1,\dots, u^n$ such that the following holds:  
\begin{enumerate}

\item  
$g^{\alpha\beta} (u) = b^{\alpha\beta} + a_s^{\alpha\beta} u^s,$
where $a_s^{\alpha\beta}$ are structure constants of a certain Frobenius algebra $\mathfrak a$, and $b$ is a Frobenius form for $\mathfrak a$;

\item  the entries $h^{\alpha \beta}$ of $h$ in this coordinate system are constant;

\item $h=\bigl(h^{\alpha\beta}\bigr)$  is a Frobenius form for $\mathfrak a$, that is,   $h^{\alpha q} a^{\beta \gamma}_q = h^{\gamma q} a^{\beta \alpha}_q$.

\end{enumerate}
\end{Fact}

This fact was independently obtained by P.\,Lorenzoni and R.\,Vitolo in their unpublished paper.  The `if' part of the statement follows from   \cite{str}  by I.\,Strachan and B.\,Szablikowski, see also \cite[Theorem 5.12]{dorfman}.

The coordinates $(u^1,\dots ,u^n)$ from Fact \ref{fact:5} will be called {\it Frobenius coordinates} for the nonhomogeneous Poisson structure $\Pthree_h+\Pone_g$. 
Of course, Frobenius coordinates are not unique; indeed, they remain to be Frobenius after any affine coordinate change.   
This is the only freedom  since the components of $h$  are constant in Frobenius coordinates.

\subsection{Reduction of our problem to an algebraic one and Frobenius pencils}\label{subsect:2.2}

\begin{Definition}{\rm 
 Let $(\mathfrak a, \star)$ and $(\bar{\mathfrak a},\bar\star)$ be Frobenius algebras defined on the same vector space $V$ and  $h, \bar h: V\times V \to \R$ the corresponding Frobenius forms.  We will say that $(\mathfrak a, h)$ and $(\bar{\mathfrak a}, \bar h)$ are  \emph{
 compatible}
 if the operation 
\begin{equation}
\label{eq:bols11}
\xi, \eta  \mapsto  \xi\star \eta + \xi \, \bar\star \, \eta, \qquad  \xi,\eta\in V,
\end{equation}
defines the structure of a Frobenius algebra with the Frobenuis form $h + \bar h$. 

Similarly,   if  $\mathfrak a$ and $\bar {\mathfrak a}$ are Frobenius algebras each of which is endowed with  two Frobenius forms $b, h$ and $\bar b, \bar h$ respectively, then we say that the triples $(\mathfrak a, b, h)$ and $(\bar{\mathfrak a}, \bar b, \bar h)$ are \emph{ compatible}  if  \eqref{eq:bols11} defines a Frobenius algebra for which $b+\bar b$ and $h + \bar h$ are both Frobenuis forms.  
}\end{Definition}

Formally, the definition requires that  $b + \bar b$ and  $h + \bar h$ are  nondegenerate. It is not essential. Indeed, if the operations 
  $\star$ and $\bar \star$ are associative, and also 
	the operation $\widehat \star:=\star+ \bar \star$  given by   \eqref{eq:bols11}
	is  associative, then any linear combination $ \lambda \star+ \bar \lambda\bar \star$ is associative. Moreover, if 
	$\widehat b:= b + \bar b$, possibly degenerate, 
		satisfies the condition \eqref{frobenius} for $\widehat \star$, 
		then the linear combination  $\lambda b + \bar \lambda \bar b$ also satisfies the condition \eqref{frobenius} with respect to $ \lambda \star+ \bar \lambda\bar \star$. Thus, passing to a suitable linear combination we can  make $\widehat b$ and also $\widehat h$   nondegenerate.

 In view of  Facts \ref{fact:4} and  \ref{fact:5},  compatible Frobenius triples 
$(\mathfrak a, b, h)$ and $(\bar{\mathfrak a}, \bar b, \bar h)$ naturally define compatible Poisson structures $\Pthree_g+\Pone_h$ and  $\Pthree_{\bar g}+\Pone_{\bar h}$.   The next theorem shows that the converse is also true under the assumption that $R_h= \bar h h^{-1}$ has $n$  different eigenvalues.

\begin{Theorem}
\label{thm:frobenius0}
Consider two non-homogeneous Poisson structures $\Pthree_h + \Pone_g$ and $\Pthree_{\bar h} + \Pone_{\bar g}$ and suppose that 
$R_h= \bar h h^{-1}$ has $n$  different eigenvalues. 

Then,  they are compatible if and only if $(g, h)$ and $(\bar g, \bar h)$ admit a common Frobenius coordinate system $u^1,\dots, u^n$ in which
\begin{enumerate}
\item $h^{\alpha\beta}$ and $\bar h^{\alpha\beta}$ are constant,
\item $g^{\alpha\beta}(u) = b^{\alpha\beta} +  a^{\alpha\beta}_s u^s$ and  $\bar g^{\alpha\beta}(u) = \bar b^{\alpha\beta} +  \bar a^{\alpha\beta}_s u^s$,
\item $(\mathfrak a, b, h)$ and $(\bar{\mathfrak a}, \bar b, \bar h)$ are compatible Frobenius triples  (here $\mathfrak a$ and $\bar{\mathfrak a}$ denote the algebras with structure constants $a^{\alpha\beta}_s$ and $\bar a^{\alpha\beta}_s$ respectively).
\end{enumerate}
\end{Theorem}

\begin{Corollary}{
In more explicit terms,   compatibility of $\Pthree_h + \Pone_g$ and $\Pthree_{\bar h} + \Pone_{\bar g}$  such that 
$R_h= \bar h h^{-1}$ has $n$  different eigenvalues, is equivalent to  reducibility of these operators, in an appropriate coordinate system $u^1,\dots, u^n$,  to the following simultaneous canonical form  
$$
    \Pthree_h + \Pone_g= h^{\alpha \beta} D^3 + b^{\alpha \beta} D +  a^{\alpha \beta}_s u^s D + \tfrac{1}{2}\, a^{\alpha \beta}_s u^s_{x},
    $$
    $$
    \Pthree_{\bar h} + \Pone_{\bar g} = \bar h^{\alpha \beta} D^3 + \bar b^{\alpha \beta} D +  \bar a^{\alpha \beta}_s u^s D + \tfrac{1}{2}\, \bar a^{\alpha \beta}_s u^s_{x},
    $$
where $h^{\alpha \beta}, \bar h^{\alpha \beta}, b^{\alpha \beta}, \bar b^{\alpha \beta}, a^{\alpha \beta}_s, \bar a^{\alpha \beta}_s$ are constants  symmetric in upper indices and satisfying the conditions:
\begin{equation}
\label{eq:r12}
\begin{aligned}
    & a^{\alpha \beta}_q a^{q \gamma}_s = a^{\gamma \beta}_q a^{q \alpha}_s, \quad \bar a^{\alpha \beta}_q \bar a^{q \gamma}_s = \bar a^{\gamma \beta}_q \bar a^{q \alpha}_s, \quad \bar a^{\alpha \beta}_q a^{q \gamma}_s + a^{\alpha \beta}_q \bar a^{q \gamma}_s = \bar a^{\gamma \beta}_q a^{q \alpha}_s + a^{\gamma \beta}_q \bar a^{q \alpha}_s, \\
    & h^{\alpha q} a^{\beta \gamma}_q = h^{\gamma q} a^{\beta \alpha}_q, \quad b^{\alpha q} a^{\beta \gamma}_q = b^{\gamma q} a^{\beta \alpha}_q, \quad  \bar h^{\alpha q} \bar a^{\beta \gamma}_q = \bar h^{\gamma q} \bar a^{\beta \alpha}_q, \quad \bar b^{\alpha q} \bar a^{\beta \gamma}_q = \bar b^{\gamma q} \bar a^{\beta \alpha}_q, \\
    & \bar h^{\alpha q} a^{\beta \gamma}_q + h^{\alpha q} \bar a^{\beta \gamma}_q = \bar h^{\gamma q} a^{\beta \alpha}_q + h^{\gamma q} \bar a^{\beta \alpha}_q, \quad \bar b^{\alpha q} a^{\beta \gamma}_q + b^{\alpha q} \bar a^{\beta \gamma}_q = \bar b^{\gamma q} a^{\beta \alpha}_q + b^{\gamma q} \bar a^{\beta \alpha}_q.
\end{aligned}
\end{equation}
}\end{Corollary}

Notice that the coordinates $u^1,\dots, u^n$ from Theorem \ref{thm:frobenius0} are just flat coordinates for $h$  (or equivalently, for $\bar h$ as these metrics have common flat coordinates  by Theorem \ref{thm:frobenius0}).

We see that Theorem \ref{thm:frobenius0} reduces the problem of description and classification of pairs of compatible 
 Poisson structures   $\Pthree_h+ \Pone_g$ and $\Pthree_{\bar h}+ \Pone_{\bar g}$ such that $R_g=\bar h h^{-1}$ has $n$ different eigenvalues to a purely algebraic problem. 
As we announced above, we will reformulate it in differential geometric terms in Section \ref{subsect:6.1}, and solve it under the assumption 
that $R_g=\bar g g^{-1}$ has $n$ different eigenvalues.

We have not succeeded in solving the problem by purely algebraic means. Like many other problems in Algebra, it reduces to   a system of quadratic and linear equations (see relations \eqref{eq:r12}).  For example,  classification of Frobenius algebras is a problem of the same type. This  problem is solved under the additional assumption that the Frobenius form is positive definite in \cite{Draisma}, and in our opinion  is out of reach otherwise. Of course, for a fixed dimension one can find complete 
 or partial  answers. In particular, in \cite{poon} it is shown, that up to dimension 6 there is a finite number of isomorphism classes of commutative associative algebras and for $n > 6$ the number of classes is infinite. In \cite{kay} the classification of nilpotent commutative associative algebras up to dimension 6 is given.
See also \cite{lor,str}.

In the situation discussed in Theorem \ref{thm:frobenius0}, consider the pencil of first order Poisson structures $\Pone_{\lambda g + \mu \bar g}$, which is sometimes referred to as quasiclassical limit \cite{Ferapontov-Pavlov} of the non-homogeneous pencil $\Pthree_{\lambda h+\mu \bar h} + \Pone_{\lambda g + \mu \bar g}$. We can ask the inverse question: {\it Given a flat pencil $\{ \lambda g + \mu \bar g\}$, does the corresponding Poisson pencil $\{\Pone_{\lambda g + \mu \bar g}\}$  admit a perturbation with nondegenerate Darboux-Poisson structures of order three of general position?}

Theorem \ref{thm:frobenius0} basically shows that the main condition for the related quadruple of metrics $(h, \bar h, g, \bar g)$ is the existence of a common Frobenius coordinate system for $g$ and $\bar g$.  Indeed, if this condition holds true and this Frobenuis coordinate system is given, then the other two metrics $h$ and $\bar h$ can be `reconstructed' by solving a system of linear equations.  More precisely, we have the following

\begin{Theorem} 
\label{prop:bols13}
Let $g$ and $\bar g$ be Poisson-compatible flat metrics that admit a common Frobenius coordinate system $u^1,\dots, u^n$, that is
$$
g^{\alpha\beta}(u) =b^{\alpha\beta} +  a^{\alpha\beta}_s u^s \quad\mbox{and}\quad   \bar g^{\alpha\beta}(u) = \bar b^{\alpha\beta} +  \bar a^{\alpha\beta}_s u^s,
$$
where $(\mathfrak a, b)$ and $(\bar{\mathfrak a}, \bar b)$ are Frobenius pairs  (here $\mathfrak a$ and $\bar{\mathfrak a}$ denote the algebras with structure constants $a^{\alpha\beta}_s$ and $\bar a^{\alpha\beta}_s$ respectively).
Then
\begin{itemize}

\item[(i)] the corresponding Frobenius algebras are compatible,

\item[(ii)]  there exist nondegenerate metrics $h$ and $\bar h$  (with $h^{\alpha\beta}$ and $\bar h^{\alpha\beta}$ being constant in coordinates $u^1,\dots,u^n$),  such that $\Pthree_h + \Pone_g$ and $\Pthree_{\bar h} + \Pone_{\bar g}$ are compatible Poisson structures,

\item[(iii)]  in Frobenius coordinates $u^1,\dots, u^n$, the (constant) metrics $h$ and $\bar h$ can always be chosen in the form \begin{equation}
\label{eq:condforh}
h^{\alpha\beta} = m^0 \, b^{\alpha\beta} +  a^{\alpha\beta}_s m^s \quad\mbox{and}\quad   \bar h^{\alpha\beta}(u) = m^0 \, \bar b^{\alpha\beta} +  \bar a^{\alpha\beta}_s m^s, \quad (m^1,\dots, m^n)\in\R^n,  \  m^0\in\R.
\end{equation}
\end{itemize}
\end{Theorem}

\subsection{AFF-pencil}\label{subsect:AFF}\label{subsect:2.3}

Consider a real affine space $V\simeq\R^n$ with coordinates $u^1, \dots, u^n$ and define the (Nijenhuis) operator $L$ and contravariant metric $g_0$ on it by:
\begin{equation}\label{norm}
    L = \left(
    \begin{array}{ccccc}
    u^1 & 1 & 0 & \dots & 0 \\
    u^2 & 0 & 1 & \dots & 0 \\
    & & & \ddots & \\
    u^{n - 1} & 0 & 0 & \dots & 1 \\
    u^n & 0 & 0 & \dots & 0 \\
    \end{array}
    \right), \quad g_0 = \left(
    \begin{array}{cccccc}
         0 & 0 & \dots & 0 & 0 & 1 \\
         0 & 0 & \dots & 0 & 1 & - u^1 \\
         0 & 0 & \dots & 1 & - u^1 & -u^2 \\
         & & \iddots & & & \\
         0 & 1 & \dots & - u^{n - 4} & - u^{n - 3} & - u^{n - 2}\\
         1 & - u^1 & \dots & - u^{n - 3} & - u^{n - 2} & - u^{n- 1} \\
    \end{array}
    \right).
\end{equation}
Next, introduce $n+1$ contravariant metrics $g_{\mathrm i} = L^i g$ for $i = 0, \dots, n$.  In matrix form, we have 
\begin{equation}\label{p3}
    g_{\mathrm i} = \left(
    \begin{array}{cc}
         a_{\mathrm{n - i}} & 0  \\
         0 & b_{\mathrm i}
    \end{array}
    \right),
\end{equation}
where $a_{\mathrm{n - i}}$ is a $(n - i) \times (n - i)$ matrix 
\begin{equation*}
    a_{\mathrm{n - i}} = \left(
    \begin{array}{ccccc}
         0 & \dots & 0 & 0 & 1 \\
         0 & \dots & 0 & 1 & - u^1 \\
         0 & \dots & 1 & - u^1 & -u^2 \\
         & \dots & & & \\
         1 & \dots & - u^{n - i - 3} & - u^{n - i - 2} & - u^{n - i - 1} \\
    \end{array}
    \right)
\end{equation*}
and $b_{\mathrm i}$ is $i \times i$ matrix of the form
\begin{equation*}
    b_{\mathrm i} = \left(
\begin{array}{ccccc}
     u^{n - i + 1} &  u^{n - i + 2} & \dots & u^{n - 1} & u^n \\
     u^{n - i + 2} & u^{n - i + 3} & \dots & u^n & 0 \\
     & & \dots & & \\
     u^{n - 1} & u^n & \dots & 0 & 0 \\
     u^{n} & 0 & \dots & 0 & 0 \\
\end{array}
    \right).
\end{equation*}
In particular, $g_{\mathrm 0} = a_{\mathrm n}$ and $g_{\mathrm n} = b_{\mathrm n}$.

The metrics $g_{\mathrm 0}, g_{\mathrm 1}, \dots, g_{\mathrm n}$  are  flat and pairwise compatible, so that they generate an $n+1$-dimensional flat pencil with remarkable properties, see e.g. \cite{Ferapontov-Pavlov, NijenhuisAppl2}.  We can write this pencil as  
\begin{equation}
\label{eq:LCpencil}
\{ P(L) g_0\}, \quad\mbox{where $P(\cdot)$ is an arbitrary polynomial of degree $\le n$}
\end{equation} 
and $L$ and $g_0$ are given by \eqref{norm}. We will refer to it as an {\it AFF-pencil}. This pencil  was discovered, 
in the form   \eqref{norm} and \eqref{p3},   by M. Antonowicz and A. Fordy \cite{fordy}. As we see, the components  of  each metric $g_i$ are affine functions, moreover,  the coordinates  $(u^1,\dots, u^n)$ are common Frobenius coordinates for all of them.

The corresponding Frobenius algebras are easy to describe.  Consider two well-known examples:
\begin{itemize}

\item  the algebra $\mathfrak a_n$ of dimension $n$ with basis $e_1,e_2,\dots, e_n$ and relations
$$
e_i \star e_j = \begin{cases} 
e_{i+j}, \quad & \mbox{if $i+j\le n$},  \\ 
0 \quad & \mbox{otherwise}.    
\end{cases} 
$$ 
Notice that $\mathfrak a_n$ can be modelled as the matrix algebra $\operatorname{Span}(J,J^2,\dots, J^n)$, where $J$ is the nilpotent Jordan block of size $(n+1)\times (n+1)$.  It contains {\it no multiplicative unity} element.

\item  the algebra $\mathfrak b_n$ of dimension $n$ with basis $e_1,e_2,\dots, e_n$ and relations
$$
e_i \star e_j = \begin{cases} 
e_{i+j-1}, \quad & \mbox{if $i+j-1\le n$},  \\ 
0 \quad & \mbox{otherwise}.    
\end{cases} 
$$
This algebra can be understood as the {\it unital} matrix algebra $\operatorname{Span}(\Id, J, J^2,\dots, J^{n-1})$ where $J$ is the nilpotent Jordan block of size $n\times n$.   The difference from the previous example is that $\mathfrak b_n$, by definition, contains the identity matrix.  Equivalently, we can define  $\mathfrak b_n$ as the algebra of truncated polynomials $\R[x]/\langle x^n\rangle$  (similarly  $\mathfrak a_n\simeq  \langle x\rangle/ \langle x^{n+1}\rangle$).  
\end{itemize}

It is straightforward to see that the metric  $g_\mathrm{n}=b_\mathrm{n}$ is related to the Frobenius algebra $\mathfrak b_n$.  Similarly $g_0 = a_n$ is related to the Frobenius algebra $\mathfrak a_n$  (this becomes obvious if we reverse the order of basis vectors and mutiply each of them by $-1$).
Hence,  formula \eqref{p3} shows that the Frobenius algebra associated with $g_{\mathrm{i}}$ is isomorphic to the direct sum $\mathfrak a_{n-i}\oplus \mathfrak b_{i}$.

It is interesting that a generic metric  $g=P(L)g_0$ from the AFF-pencil \eqref{eq:LCpencil}, i.e. such that $P(L)$ has $n$ distinct roots,  corresponds to the direct sum $\R \oplus \dots \oplus \R \oplus \mathbb C \oplus \dots \oplus \mathbb C$, where each copy of $\R$ relates to a real root  and each copy of $\mathbb C$ relates to a pair of complex conjugate roots of $P(\cdot)$.

It is a remarkable fact that for each $g_{\mathrm{i}}$ we can find a partner $h_{\mathrm{i}}$ such that $\Pthree_{h_{\mathrm{i}}} + \Pone_{g_{\mathrm{i}}}$ is a Poisson structure and all these structures are pairwise compatible.  The (constant) metrics $h_{\mathrm{i}}$ take the form
\begin{equation}
\label{eq:bols_hi_2}
h_{\mathrm{i}} = (g_{\mathrm{i}})_{\bar m, m^0},  \quad  \bar m=(m^1,\dots, m^n)\in \R^n, m^0\in\R, 
\end{equation}
where  $(g_{\mathrm{i}})_{\bar m, m^0}$ is obtained from the matrix $g_{\mathrm{i}}(u)$ by replacing $u^s$ with $m^s$ and all $1$'s with $m^0$.  
In this way, we obtain an $(n+1)$-dimensional pencil of non-homogeneous Poisson structures generated by $\Pthree_{h_{\mathrm{i}}} + \Pone_{g_{\mathrm{i}}}$:
\begin{equation}
\label{eq:att3_1}
\left\{ \sum_{i=0}^{n}  c_i \bigl(\Pthree_{h_{\mathrm{i}}} + \Pone_{g_{\mathrm{i}}}\bigr)  \right\}_{c_i\in\R}
\end{equation}

Alternatively, the pencil \eqref{eq:att3_1} can be described as follows.  Fix $\bar m=(m^1,\dots, m^n)\in \R^n$, $m^0\in\R$ and  let $L(\bar m)$ denote the operator with constant entries obtained from $L=L(u)$ by replacing $u^i$ with constants $m^i\in\R$.  Similarly, $g_0(\bar m)$ denotes the metric with constant coefficients obtained from $g_0$ by replacing $u^i$ with the same constants $m^i\in\R$.

Then for $g = P(L) g_0$ we can define its partner $h$  (metric with constant entries) as
$$
h = m^0 P\left(L \left(\tfrac{1}{m^0} \,\bar m\right)\right) g_0\left(\tfrac{1}{m^0}\, \bar m\right)
$$
It can be easily checked that the correspondence $(m^0,m^1,\dots, m^n) \mapsto h$ defined by this formula is linear so that it makes sense for $m_0=0$ (the denominators cancel out).     Then the pencil \eqref{eq:att3_1} can be, equivalently, defined as
\begin{equation}
\label{eq:disppert_AFF}
\left\{
\Pthree_{m^0 P\left(L \left(\frac{1}{m^0} \,\bar m\right)\right) g_0\left(\frac{1}{m^0}\, \bar m\right)} + \Pone_{P(L) g_0}
\right\}_{\deg P(\cdot) \le n}.
\end{equation}

Notice that such a pencil is not unique, as the above construction depends on $n+1$ arbitrary parameters $m^0,m^1,\dots, m^n$.   
In other words, in \eqref{eq:disppert_AFF}, the polynomial $P(\cdot)$ serves as a parameter of the bracket within the AFF-pencil, whereas $(m_0, \bar m)$ parametrise dispersive perturbations of this pencil.

\begin{Remark}\label{Rem:2.1}
{\rm
For our purposes below it will be convenient to rewrite this pencil in another coordinate system by taking the eigenvalues of $L$ as local coordinates $x^1,\dots, x^n$. In these coordinates,  $g_0$ and $L$ from \eqref{norm} take the following diagonal form
(see e.g. \cite[p. 214]{Ferapontov-Pavlov} or \cite[\S 6.2]{nij1})\footnote{The letters $\mathsf{LC}$ in $g_{\LC}$ refer to Levi-Civita. The metric $g_{\LC}$ played the key role in his classification of geodesically equivalent metrics \cite{LeviCivita}. See also \cite{NijenhuisAppl2} for discussion on the relationship between geodesically equivalent and Poisson-compatible metrics.}:
\begin{equation}
\label{norm1}
g_{\LC}= \sum_{i=1}^n \left( \prod_{s\ne i} (x^i-x^s) \right)^{-1} \left( \tfrac{\partial}{\partial {x^i}} \right)^2, \ \   L= \textrm{diag}(x^1,...,x^n),
\end{equation} 
so that the AFF pencil  \eqref{eq:LCpencil} becomes diagonal too:
\begin{equation}
\label{eq:AFF_LC}
\{ P(L) g_{\LC} \}, \quad  \mbox{where $P(\cdot)$ is a polynomial of degree $\le n$.} 
\end{equation} 
We also notice that the transition from the diagonal coordinates $x$ to Frobenius coordinates $u$ is quite natural:  the coordinates $u^i$ are the coefficients $\sigma_i$ of the characteristic polynomial $\chi_L (t) = \det (t\cdot \Id - L)= t^n - \sigma_1 t^{n-1} - \sigma_2 t^{n-2} -\dots - \sigma_n$, so that, up to sign,  $u_i$ are elementary symmetric polynomials in $x^1,\dots, x^n$.   
}\end{Remark}

The AFF pencil provides a lot of examples of compatible flat metrics $g$ and $\bar g$ that admit a common Frobenius coordinate system: one can take any two metrics from the pencil \eqref{eq:LCpencil} or, equivalently, \eqref{eq:AFF_LC}.

\section{Compatible flat metrics with a common Frobenius coordinate system:  
generic case}\label{sect:3}

Theorems  \ref{thm:frobenius0} and  \ref{prop:bols13}  reduce the compatibility problem for two Poisson structures of the form $\Pthree_h + \Pone_g$ to a classification of all pairs of  metrics $g$ and $\bar g$ admitting a common Frobenius coordinate system.  
The next theorem solves this problem under the standard assumption that $R_g= \bar g g^{-1}$ has $n$ different eigenvalues and one minor additional condition.

\begin{Theorem}
\label{thm:frobenius1}
Let $g$ and $\bar g$ be compatible flat metrics that admit a common Frobenius coordinate system.
Assume that the eigenvalues of the operator $R_g=\bar g g^{-1}$ are all different and in
the diagonal coordinates (such that $R_g$ is diagonal) every diagonal component of $g$ depends
on all variables.  Then the flat pencil $\lambda g + \mu\bar g$ is contained in the AFF-pencil, in other words, there exists a coordinate system $(x^1,\dots, x^n)$ such that 
$$
g = P(L)  g_{\LC} \quad\mbox{and}\quad  \bar g = Q(L) g_{\LC}.
$$
for some polynomials $P(\cdot)$ and $Q(\cdot)$ of degree  $\le n$ and $g_{\LC}$ and $L$ defined by \eqref{norm1}.

Moreover,  if  $n\ge 2$  and $P(\cdot)$ and $Q(\cdot)$ are not proportional, then
the common Frobenius coordinate system for $g = P(L)  g_{\LC}$ and  $\bar g = Q(L) g_{\LC}$
is unique up to an affine coordinate change. 
\end{Theorem}

Theorem \ref{thm:frobenius1} will be proved in Section \ref{sect:6}.   The uniqueness part will be explained in Section \ref{sect:8a}, see Remark \ref{last:step:proof}.

\begin{Remark}\label{Rem:3.1bis}{\rm 
In Theorem \ref{thm:frobenius1}  we allow some of the eigenvalues $R_g$ to be complex. In this case,  we think that a 
part of the  {\it diagonal}  coordinates $(x^1,...,x^n)$ is also complex-valued. For example, the coordinates $x^1,...,x^k$ may be 
 real-valued, and the remaining coordinates  $x^{k+1}= z^1, x^{k+2}= \bar z^1$,..., $x^{n-1}= z^{\tfrac{n-k}{2} }$, $x^{n }= \bar z^{\tfrac{n-k}{2} }$, where `$\ \bar { \ } \ $'
 means  complex conjugation, are complex-valued. 
 In this case (26) gives us a well-defined (real) metric $g_{\LC}$ and a (real) Nijenhuis operator $L$.   
}\end{Remark}

The genericity condition in Theorem \ref{thm:frobenius1}  is that every diagonal component of $g$ depends
on all variables. 
In Theorems \ref{thm:frobenius2a}, \ref{thm:frobenius2} below we will solve the problem in full generality, without assuming this or any other genericity condition.

\begin{Remark}\label{QFP}{\rm
In \cite[\S 5]{Ferapontov-Pavlov} E. Ferapontov and M. Pavlov asked    whether   dispersive perturbations of the pencil \eqref{eq:LCpencil} with $g_0$ and $L$ given by 
\eqref{norm1} other than those described in Section  \ref{subsect:AFF}   are possible. Theorem \ref{thm:frobenius1} leads to a negative answer under the additional assumption that the dispersive perturbation is in the class of  nondegenerate Darboux-Poisson structures of order 3. Indeed, according to Theorem \ref{thm:frobenius0}  every dispersive perturbation $\lambda (\Pthree_h + \Pone_g) + \mu (\Pthree_{\bar h} + \Pone_{\bar g})$ of the pencil $\lambda \Pone_g + \mu\Pone_{\bar g}$  can be reduced to a simple normal form in a common Frobenius coordinate system for $g$ and $\bar g$ (assuming that $R_h=\bar h h^{-1}$ has different eigenvalues).  Moreover,  in this coordinate system $h$ and $\bar h$ are constant and represent Frobenius forms for the corresponding Frobenius algebras $\mathfrak a$ and $\bar{\mathfrak a}$.  Since by Theorem  \ref{thm:frobenius1}, such a coordinate system is unique,  it remains to solve a Linear Algebra problem of choosing suitable forms $h$ and $\bar h$, satisfying three conditions  (cf. \eqref{eq:r12}):
\begin{equation}
\begin{aligned}
h(\xi  \, \star  \, \eta,  \zeta) &= h(\xi, \eta  \, \star \,  \zeta),\\
\bar h(\xi \ \bar \star \ \eta, \zeta) &= \bar h(\xi, \eta \  \bar \star \  \zeta),\\
\bar h(\xi  \, \star \,  \eta, \zeta) + h(\xi  \ \bar \star  \ \eta, \zeta) &= \bar h(\xi, \eta  \, \star  \, \zeta) +
h(\xi, \eta  \ \bar \star \  \zeta),
\end{aligned}   
\end{equation}

It is straightforward to show for a generic pair $g$, $\bar g$ of metrics from the AFF-pencil, the forms $h$ and $\bar h$ are defined by $n+1$ parameters $m^0, m^1,\dots, m^n$ as in \eqref{eq:disppert_AFF}. No other solutions exist. In particular, formula \eqref{eq:disppert_AFF} describes all possible dispersive perturbations of the AFF-pencil by means of nondegenerate Darboux-Poisson structures of order 3. Moreover, this conclusion holds for any generic two-dimensional subpencil.  
}\end{Remark}

\section{Compatible flat metrics with a common Frobenius coordinate system:  general case}
\label{sect:4}

\subsection{General multi-block Frobenius pencils}\label{subsect:4.1}

Let us now discuss the general case without assuming that  in diagonal coordinates,  
every diagonal component of $g$ depends on all variables.   
 
Similar to Theorem \ref{thm:frobenius1},  the metrics $g$ and $\bar g$ will belong to a large Frobenius pencil built up from  several blocks each of which has a structure of an (extended) AFF pencil.  We start with constructing a series of such pencils.

We first divide our diagonal  coordinates into $B$ blocks of  positive dimensions $n_1,...,n_B$ with   $n_1+...+n_B=n$:
	\begin{equation}
	\label{eq:decomposition}
	(\underbrace{x_1^1,...,x_1^{n_1}}_{X_1},...,\underbrace{x_B^{1},...,x_B^{n_B}}_{X_B}).
	\end{equation}

Next, we consider a collection of $n_\alpha$-dimensional  Levi-Civita metrics $g^{\LC}_\alpha$ and $n_\alpha$-dimensional operators $L_\alpha$ (as in Theorem \ref{thm:frobenius1} but now for each block separately):
\begin{equation}
\label{eq:tildeg}
	g^{\LC}_\alpha= \sum_{s=1}^{n_\alpha} \left(\prod_{j\ne s} (x_\alpha^s- x_\alpha^j)\right)^{-1}  \left(\tfrac{\partial}{\partial x_\alpha^s}\right)^2\  \ , \ \  L_\alpha= \textrm{diag}(x_\alpha^1,...,x_\alpha^{n_\alpha}).
\end{equation}

Then we introduce a new block-diagonal  metric $\widehat g$ 
\begin{equation}
\label{eq:hatg}
\widehat g = \operatorname{diag}(\widehat g_1, \dots, \widehat g_B)\quad \mbox{with } \widehat g_\alpha = 
\prod_{s<\alpha} \left(\frac{1}{\det (\lambda_{s\alpha} \cdot\Id - L_s)}\right)^{c_{s\alpha}} g^{\LC}_\alpha,
\end{equation}
where $c_{s\alpha} = 0$ or $1$.  The values of the discrete parameters $c_{s\alpha}$ and numbers $\lambda_{s\alpha}$ are determined by some combinatorial data as explained below. 
 
Finally, we consider the pencil of  (contravariant) metrics of the form 	
\begin{equation}
\label{eq:genpencil}
	\{ \widehat L  \, \widehat g~|~ \widehat L  \in \mathcal L\} 
\end{equation}
where  $\mathcal L$ is a family (pencil) of block-diagonal operators of the form
$$
\widehat L =  \operatorname{diag}\bigl(P_1(L_1), P_2(L_2), \dots, P_B(L_B)\bigr).
$$
where $P_\alpha(\cdot)$ are polynomials with $\deg P_\alpha\le n_\alpha+1$ treated as parameters of this family. The coefficients of the polynomials $P_\alpha$ are not arbitrary but satisfy a collection of linear relations involving coefficients from different polynomials so that this pencil,  in general, is not a direct sum of blocks (although, direct sum is a particular example).
Notice that $\mathcal L$ is a Nijenhuis pencil whose algebraic structure is quite different from that of the pencil $\{ P(L)\}$ from Theorem \ref{thm:frobenius1}.

The numbers $c_{s\alpha}$, $\lambda_{s\alpha}$ and relations on the coefficients of $P_\alpha$'s are determined by a combinatorial object,  an oriented graph $\mathsf F$ with special properties, namely,  a  directed rooted in-forest (see \cite{wiki} for a definition) whose edges are labelled by numerical marks $\lambda_\alpha$.   This graph may consists of several connected components, each of which  is a rooted tree whose edges are oriented from its {\it leaves} to the {\it root}. An example is shown in 
Figure \ref{Fig:1}.

Each vertex of $\mathsf F$ is associated with a certain block of the above decomposition \eqref{eq:decomposition} and labelled by an integer number $\alpha \in \{1,...,B\}$. The structure of a directed graph  defines a natural strict partial order (denoted by $\prec$) on the set $\{1,...,B\}$: for two numbers $\alpha\ne \beta \in \{1,...,B\}$ we set $\alpha \prec \beta$, if there exists an oriented way from $\beta$ to $\alpha$.
For instance,  for the graph shown on Fig. \ref{Fig:1}, we have $1\prec 3$, $2\prec 4$, $5\prec 6$. Without loss of generality we can and will always assume that the vertices of $\mathsf F$ are labeled in such a way that $\alpha \prec \beta$ implies $\alpha<\beta$.

\begin{figure}
\begin{center}
 \includegraphics[width=0.7\textwidth]{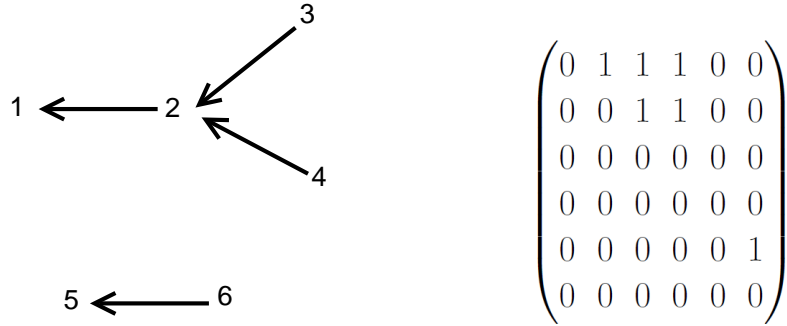} 
\end{center}
\caption{A $6\times 6$ matrix $c_{\alpha \beta}$ and the corresponding in-forest. The upper  tree corresponds to the upperleft $4\times 4$-block, the lower  tree corresponds to the downright $2\times 2$-block. 
 \label{Fig:1}}
\end{figure}

Notice that the vertices of degree one are of two types, roots and leaves:   $\alpha$ is a {\it root} if there is no $\beta$ such that $\beta  \prec \alpha$ and, conversely,  $\beta$ is a {\it leaf} if there is no $\beta$ such that $\alpha \prec \beta$.  Notice that roots of degree $\ge 2$ are also allowed, whereas all leaves have degree 1.
We say $\alpha= \operatorname{next}(\beta)$, if $\alpha\prec \beta$ and there is no $\gamma$ with $\alpha\prec  \gamma \prec \beta$. 
In the  upper tree of Fig. \ref{Fig:1} the root is $1$,   the leaves are $3$ and $4$ and we have:  $1= \textrm{next}(2)$  and $2= \textrm{next}(3)$, $2= \textrm{next}(4)$.

The numbers $c_{s \alpha}$ in \eqref{eq:hatg} are now defined from $\mathsf F$ as follows:
\begin{equation}
\label{eq:condcij}
c_{s \alpha} = \begin{cases}
1, \mbox{ if $s\prec \alpha$,}\\
0, \mbox{ otherwise}.
\end{cases}
\end{equation}
Recall that in our assumptions, $s\prec \alpha$ implies $s<\alpha$  so that the $B{\times}B$-matrix $c_{s\alpha}$ is upper triangular with zeros on the diagonal, see Figure  \ref{Fig:1}.

The parameters $\lambda_{s\alpha}$ are defined as follows.    For each vertex  $\alpha$ that is not a root there is exactly one out-going edge which we will denote by $e_\alpha$.   Notice that the correspondence $\alpha \mapsto e_\alpha$ is a bijection between the set of edges of $\mathsf F$ and the (sub)set of vertices which are not roots.  To each edge  $e_\alpha$ we now assign a number $\lambda_\alpha$ (these numbers will serve as parameters of our construction) and set
\begin{equation}
\label{eq:condlambdaij}
\lambda_{s \alpha} = \lambda_{\beta},   \quad \mbox{where  $s \prec \beta \prec \alpha$  and $s = \operatorname{next}(\beta)$   (or $\beta=\alpha$, if $s = \operatorname{next}(\alpha)$).}
\end{equation}  
Such $\beta$ exists and is unique,  if $s \prec \alpha$, i.e., $c_{s\alpha}=1$. Otherwise, $c_{s\alpha}=0$ and the value of $\lambda_{s \alpha}$ plays no role in \eqref{eq:hatg}.
\begin{Remark}{\rm
The above definitions of parameters $c_{s\alpha}$ and $\lambda_{s\alpha}$ are convenient to make our formulas shorter,  but do not quite clarify the meaning of  \eqref{eq:hatg} in terms of the graph $\mathsf F$.   Equivalently,   formula \eqref{eq:hatg}  can be rewritten as follows:
$$
\widehat g = \operatorname{diag}(\widehat g_1, \dots, \widehat g_B)\quad \mbox{with } \widehat g_\alpha = 
f_\alpha \cdot g^{\LC}_\alpha,
$$
where the function $f_\alpha$  is defined by the oriented path from the vertex $\alpha$ to a certain root $\beta$:
$$
\beta = \alpha_0  \overset{\lambda_{\alpha_1}}{\longleftarrow} \alpha_1 \overset{\lambda_{\alpha_2}}{\longleftarrow} \dots  \overset{\lambda_{\alpha_{k-1}}}{\longleftarrow} \alpha_{k-1}  \overset{\lambda_{\alpha_k}}{\longleftarrow} \alpha_k = \alpha, 
$$
each edge of which is endowed with a number $\lambda_{\alpha_i}$, $i=1,\dots,k$.  Namely, we set
$$
f_\alpha = \prod_{i=1}^k \frac{1}{\det(\lambda_{\alpha_i}\cdot\Id - L_{\alpha_{i-1}})}, 
$$
which coincides with the factor in front of  $g_\alpha^{\LC}$  in formula  \eqref{eq:hatg} written in terms of $c_{s\alpha}$ and $\lambda_{s\alpha}$.
}\end{Remark}

Finally, for a vertex  $\alpha$ we denote  the coefficients  of the  corresponding polynomial 
$P_\alpha$ by $P_\alpha(t) =   \overset{\alpha}{a}_0 + \overset{\alpha}{a}_1  t +...+  \overset{\alpha}{a}_{n_\alpha+1}t^{n_\alpha+1}$.  Then the conditions on the coefficients $P_\alpha(t)$ are

\begin{itemize}
\item[(i)] If $\alpha$ is a root, then $\overset{\alpha}{a}_{n_\alpha+1}=0$,  i.e., $\deg P_\alpha\le n_\alpha$.

\item[(ii)] If $\alpha= \operatorname{next}(\beta)$, 
then $\lambda_\beta$ is a root of $P_\alpha$ and $\overset{\beta}{a}_{n_\beta+1}= P'(\lambda_\beta)$, where $P'(t)$ denotes the derivative of $P(t)$.

\item[(iii)] If $\alpha=  \operatorname{next}(\beta)$ and $\alpha=  \operatorname{next}(\gamma)$ with $\lambda_\beta = \lambda_\gamma=\lambda$, $\beta\ne \gamma$,  then $\lambda$ is a double root of $P_\alpha$   (in view of (ii) this automatically implies 
$\overset{\beta}{a}_{n_\alpha+1} = \overset{\gamma}{a}_{n_\gamma+1} =0$).
\end{itemize}

\begin{Remark}{\rm
Each of the above conditions is linear in the coefficients of $P_\alpha$'s.  However, (i)--(iii) may imply that $P_\alpha =0$ for some $\alpha$.  This happens, for instance, if the vertex $\alpha$ has too many neighbours $\gamma_i$ such that $\alpha = \operatorname{next}(\gamma_i)$, $i=1,\dots,k$. Then all $\lambda_{\gamma_i}$ must be roots of $P_\alpha$ due to (ii). However in view of (i),   $\deg P_\alpha \le n_\alpha + 1$. If $n_\alpha + 1 < k$ and $\lambda_{\gamma_i}$ are all different, then $P_\alpha$ cannot have $k$ different roots unless $P_\alpha =0$.  Strictly speaking, such a situation should be excluded as the corresponding metrics turn out to be degenerate at every point. However,  from the algebraic point of view we still obtain an example of a good Frobenius pencil.

We also notice that the shift $L_\alpha \mapsto L_\alpha + c_\alpha \Id$ in any individual block leads to an isomorphic pencil. In particular, if at a certain vertex $\alpha$ of $\mathsf F$ we add the same number $c_\alpha$ simultaneously to all numerical parameters $\lambda_\beta$ on the incoming edges,  we get an isomorphic pencil.
}\end{Remark}

This completes the description of  the pencil \eqref{eq:genpencil}  of (contravariant) metrics and we can state our next result.

\begin{Theorem} 
\label{thm:frobenius2a}
The pencil \eqref{eq:genpencil}  {\rm(}with $c_{s\alpha}$ defined by \eqref{eq:condcij},   $\lambda_{s\alpha}$ defined by \eqref{eq:condlambdaij} and coefficients of $P_\alpha$ satisfying {\rm (i)-(iii))} is Frobenius.  In other words, 
all the metrics   
\begin{equation}
\label{eq:genpencil2}	
g = \widehat L  \, \widehat g = \operatorname{diag}\bigl(P_1(L_1)\widehat g_1, \dots, P_B(L_B)\widehat g_B\bigr)\quad \mbox{with } \widehat g_\alpha = \prod_{s<\alpha} \left(\frac{1}{\det   (\lambda_{s\alpha} \cdot \Id - L_s)}\right)^{c_{s\alpha}} g^{\LC}_\alpha,
\end{equation}
are flat, Poisson compatible and admit a common Frobenius coordinate system 
\begin{equation}
\label{eq:uFrob}
(u^1,...,u^n)= (\underbrace{u_1^1,...,u_1^{n_1}}_{U_1},...,\underbrace{u_B^1,...,u_B^{n_B}}_{U_B}  ) 
\end{equation}
which is defined as follows.
Let  $\sigma_\alpha^1,\dots,\sigma_\alpha^{n_i}$ denote the coefficients of the characteristic polynomial of $L_\alpha$
$$
\chi_{L_\alpha}(t):=\det\left(t \Id_{n_\alpha\times n_\alpha} - L_\alpha\right)= t^{n_\alpha}- \sigma_\alpha^1 t^{n_\alpha-1}-\sigma_\alpha^2 t^{n_\alpha-2}-...-\sigma_\alpha^{n_\alpha}, \quad \alpha=1,\dots,B.
$$
Then 
\begin{equation}
\label{eq:fromytou}
\begin{array}{lll}
u_1^k &= \sigma_1^k,  \quad &k=1,\dots, n_1,\\
u_2^k &= \bigl(\det (\lambda_{12}\Id - L_1)\bigr)^{c_{12}} \ \sigma_2^k, \quad & k=1,\dots, n_2, \\
u_3^k &= \bigl(\det  (\lambda_{13}\Id - L_1)\bigr)^{c_{13}}\bigl(\det  (\lambda_{23}\Id - L_2)\bigr)^{c_{23}}\ \sigma_3^k, \quad &k=1,\dots, n_3,  \\
          & \dots & \\
u_B^k &= \prod_{s< B} \bigl(\det (\lambda_{sB}\Id - L_s)\bigr)^{c_{sB}} \ \sigma_B^k, \quad &k=1,\dots, n_B.  \\
\end{array}
\end{equation} 

\end{Theorem}

Theorem \ref{thm:frobenius2a} will be proved in Section \ref{sect:8}.

The advantage of the formulas for Frobenius coordinates in Theorem \ref{thm:frobenius2a} is that they are invariant in the sense they do not depend on the  choice of  
 coordinates in blocks, but use coefficients of the characteristic polynomials of blocks $L_i$.  
 
 Let us explain how one can use this property to check algorithmically  (say, using computer algebra software) that the  coordinates in Theorem \ref{thm:frobenius2a} are indeed Frobenius for the metric $g$.

In each block (with number $\alpha$), we change from diagonal coordinates $X_\alpha= (x_\alpha^1,...,x_\alpha^{n_\alpha})$ to the coordinates 
$Y_\alpha= (y_\alpha^1,..., y_\alpha^{n_\alpha})$ given as follows:
\begin{equation}
\chi_{L_\alpha} (t)= t^{n_\alpha}-  y_\alpha^1\, t^{n_\alpha-1}- y_\alpha^2 \,t^{n_\alpha-2}-...- y_\alpha^{n_\alpha} . 
\end{equation}

Note that in the coordinates $Y_\alpha$,   the metric 
$g^{\LC}_\alpha$ and the operator $L_\alpha$ have the form \eqref{norm}  with $u^1,...,u^n$ replaced by $y_\alpha^1,..., y_\alpha^{n_\alpha}$. The iterated warped  product metric $g=(g^{ij})$ is given by the following  easy algebraic formula
$$
g = g_1 + \left(\frac{1}{
\chi_{L_1} (\lambda_{12})}
\right)^{c_{12}} g_2 +  \left(\frac{1}{
\chi_{L_1} (\lambda_{13})}\right)^{c_{13}}\left(\frac{1}{
\chi_{L_1} (\lambda_{23})}\right)^{c_{23}} g_3+... \ ,
$$
with $g_\alpha = P_\alpha(L_\alpha)  g^{\LC}_\alpha$ and $g^{\LC}_\alpha$ and $L_\alpha$ explicitly given by
\eqref{norm}.

 In order to check whether the coordinates $u$  given by \eqref{eq:fromytou}  are Frobenius, 
one needs to perform the multiplication 
$$
J g J^\top,
$$
where $J = \left( \frac{\partial u^i}{\partial y^j}\right)$ is the Jacobi matrix of the coordinate transformation\footnote{This transformation is given by \eqref{eq:fromytou}  as $y_\alpha^i = \sigma_\alpha^i$ and 
$\det(\lambda_{\alpha \beta}\Id - L_\alpha) =   \chi_{L_\alpha} (\lambda_{\alpha\beta})$.} $(y^1,..., y^n) \to (u^1,...,u^n)$  and check whether the entries of the resulting matrix $J g J^\top$ are affine
functions in $u^i$ and conditions \eqref{eq:r12} are fulfilled. All these operations can be realised by  standard computer algebra packages.

The next result gives a description of two-dimensional Frobenius pencils in the general case.

\begin{Theorem}
\label{thm:frobenius2} 
Let $g$ and $\bar g$ be compatible flat metrics that admit a common Frobenius coordinate system.  If the eigenvalues of the operator  $R_g=\bar g g^{-1}$ are all different at a point $\mathsf p$, then in a neighbourhood of this point  the pencil $\lambda g +\mu \bar g$ is isomorphic to a two-dimensional subpencil of the Frobenius pencil \eqref{eq:genpencil} with suitable parameters,  i.e.,  in a certain coordinate system 
these metrics take the form
\begin{equation}
\label{eq:thm4}
g = \operatorname{diag}\left( P_1( L_1) \widehat g_1, \dots, P_B( L_B) \widehat g_B\right) \quad \mbox{and}\quad
\bar g = \operatorname{diag}\left( Q_1( L_1) \widehat g_1, \dots,  Q_B( L_B) \widehat g_B\right)
\end{equation}
{\rm(}with parameters $c_{\alpha\beta}$ defined by \eqref{eq:condcij},  $\lambda_{s\alpha}$ defined by \eqref{eq:condlambdaij} and coefficients of $P_\alpha$ and $Q_\alpha$ satisfying {\rm (i)-(iii))}.
\end{Theorem}

Theorem \ref{thm:frobenius2}  will be proved in Section \ref{sect:7}.

\begin{Remark} {\rm 
In  Theorem   \ref{thm:frobenius2} we allow complex  eigenvalues of $R_g$. The corresponding  part of diagonal coordinates is then  complex.  Moreover, the polynomials $P_\alpha$ and $Q_\alpha$  may have complex coefficients, and also the numbers  $\lambda_{s\alpha}$  may be complex. The only condition is that the metrics given by (30) should be well-defined real metrics. It is easy to see that this condition implies in particular  that every block \  $(g^{\LC}_\alpha,  L_\alpha)$ is either real or  pure complex  (= all coordinates are complex; the coefficients of the polynomials $P_\alpha$ and $Q_\alpha$ may be complex as well), and that   a pure complex block comes together with a complex-conjugate one. 
See also  \cite[\S 3]{nij1} for  discussion on Nijenhuis operators some of whose eigenvalues are complex. 
 }\end{Remark}

In certain special cases, a common Frobenius coordinate system for $g$ and $\bar g$ is not unique (up to affine transformations). This is the case  when $n_\alpha=1$,   $c_{\alpha \beta}=0$ for all $\beta$ (i.e., this block represents a leaf of the corresponding in-forest)  and the diagonal component of $R_g=\bar g g^{-1}$ corresponding to this block is constant, in other words, the (quadratic) polynomials $P_\alpha$ and $Q_\alpha$ are proportional.  The restrictions $g_\alpha$ and $\bar g_\alpha$ onto these blocks are then as follows 
$$
g_\alpha =  f \cdot  \left(a_2 (x^\alpha)^2 + a_1 x^\alpha + a_0\right) \left(\tfrac{\partial}{\partial x^\alpha} \right)^2  \quad \mbox{and} \quad
\bar g_\alpha = c \, g_\alpha = c\,  f \cdot  \left(a_2 (x^\alpha)^2 + a_1 x^\alpha + a_0\right) \left(\tfrac{\partial}{\partial x^\alpha} \right)^2, 
$$ 
where $f$ is some function of the remaining coordinates, $c\in\R$  and $L_\alpha = (x^\alpha)$  (diagonal $1\times 1$ matrix). 
However, we can do coordinate transformation   $x^\alpha \mapsto \tilde x^\alpha =  \tilde x^\alpha (x^\alpha)$  that changes the coefficients $a_1$ and $a_0$  (the highest coefficient $a_2$ is fixed by condition (ii)). 
$$
g_\alpha =  f \cdot  \left(a_2 (\tilde x^\alpha)^2 + \tilde a_1 \tilde x^\alpha + \tilde a_0\right) \left(\tfrac{\partial}{\partial \tilde x^\alpha} \right)^2  \quad \mbox{and} \quad
\bar g_\alpha = c \, g_\alpha = c\,  f \cdot  \left(a_2 (\tilde x^\alpha)^2 + \tilde a_1 \tilde x^\alpha + \tilde a_0\right) \left(\tfrac{\partial}{\partial \tilde x^\alpha} \right)^2, 
$$ 
Hence, with a new operator  $L_\alpha^{\mathsf{new}} = (\tilde x^\alpha)$  and new polynomials $P^{\mathsf{new}} _\alpha (t) = a_2 t^2 + \tilde a_1 t + \tilde a_0$, $Q^{\mathsf{new}} _\alpha(t) = c(a_2 t^2 + \tilde a_1 t + \tilde a_0)$, we still remain in the framework of our construction and \eqref{eq:thm4} still holds. This transformation will lead to another Frobenius coordinate system. 
 In Section \ref{sect:8a} we explain that only this situation allows ambiguity in   the choice of Frobenius coordinates up to affine transformations.

\begin{Remark}{\rm 
In \cite[Theorem 2]{Ferapontov-Pavlov} it was claimed that {\it under some general assumptions} for $n>2$,
there is only one equivalence class of $(n+1)$-Hamiltonian hydrodynamic systems (in the sense of \cite{Ferapontov-Pavlov}) and $n+1$ is the best possible.  The corresponding multi-Hamiltonian structure  comes from the $(n+1)$-dimensional AFF-pencil.  
In this view,  it is interesting to notice that multi-block pencils from Theorem \ref{thm:frobenius2} also provide such a structure, which may have even higher dimension. 
}\end{Remark}

\subsection{Case of two blocks} \label{subsect:4.3} 

In the case of two blocks, i.e., $B=2$, the construction explained in the previous section gives a natural and rather simple answer. 
 We have two cases: $c_{12}=0$ and $c_{12}=1$. The first case is trivial being a direct product of two blocks (possibly complex conjugate) each of which is as in Theorem \ref{thm:frobenius1};  in \eqref{eq:genpencil2} we set  $\widehat g_i=g_i^{\LC}$  and take $P_i$ to be arbitrary polynomials of degrees $\le n_i$ ($i=1,2$). 
 
 Theorem below is a special case of Theorem 4 in the non-trivial case $c_{12}=1$.

\begin{Theorem} \label{thm:frobenius3} 
Suppose $B=2$, $c_{12}=1$ and consider the metric $g$ given by the construction from Section \ref{subsect:4.1}:
\begin{equation}
\label{eq:2blockswarp}
g = g_1 +  \frac{1}{\det(-L_1)} \, g_2, \quad  \mbox{with } g_i = P_i(L_i) g_i^{\LC}.
\end{equation}
Following this construction, assume that the polynomials $P_1$ and $P_2$ have degrees no greater than $n_1 $ and $n_2+1$ respectively: $P_1=\sum_{s=0}^{n_1} a_st^s$ and  $P_2=\sum_{s=0}^{n_2+1} b_st^s$. Then the coordinates from Theorem \ref{thm:frobenius2a} are Frobenius for $g$ if and only if    $a_0=0$    and  $a_1 = b_{n_2+1}$. 
\end{Theorem}


\begin{Example}\label{Ex:1}{\rm 
In Theorem \ref{thm:frobenius3}, take $n_1=n_2=2$.  
In diagonal coordinates $x^1, x^2, x^3, x^4$,  the metric $g=(g^{ij})$  is as follows:
$$
g = \operatorname{diag} \left(\frac{P_1(x^1)}{x^1-x^2}, \frac{P_1(x^2)}{x^2-x^1},  \frac{P_2(x^3)}{x^1x^2(x^3-x^4)},  \frac{P_2(x^4)}{x^1x^2(x^4-x^3)}\right),  
$$
where $P_1(t) = a_1 t + a_2t^2$ and $P_2(t)= b_0 + b_1 t + b_2 t^2 + b_3 t^3$ with $b_3 = a_1$. 
Recall that $L = L_1\oplus L_2$ with $L_1=\operatorname{diag} (x^1, x^2)$, $L_2=\operatorname{diag} (x^3, x^4)$,  and the relation between the diagonal coordinates $x^i$ and the Frobenius coordinates $u^i$ given by  Theorem \ref{thm:frobenius2a} are as follows:
$$
\begin{aligned}
u^1 &=    \tr L_1 = x^1 + x^2,  \\
u^2 &=  - \det L_1 =-  x^1x^2, \\
u^3 &= \det L_1 \cdot \tr L_2 = x^1x^2 (x^3 + x^4),\\
u^4 &=   -\det L_1 \cdot \det  L_2 =-x^1x^2x^3x^4=-\det L.
\end{aligned}
$$
In these Frobenius coordinates, the metric $g = (g^{ij})$ has the following form:
$$
g= 
\left( \begin {array}{cccc} 
a_{2}u^{1}+a_{1} & a_{2}u^{2} & a_{2}u^{3} & a_{2}u^{4} \\ 
\noalign{\medskip}a_{2}u^{2} & a_{1}u^{2}& a_{1}u^{3} & a_{1}u^{4} \\ 
\noalign{\medskip}a_{2}u^{3}&a_{1}u^{3} & -a_{1}u^{4}-b_{1}u^{2}-b_{2}u^{3}&-b_{0}u^{2}-b_{2}u^{4} \\ 
\noalign{\medskip}a_{2}u^{4}&a_{1}u^{4} & -b_{0}u^{2}-b_{2}u^{4} & b_{0}u^{3}-b_{1}u^{4}
\end {array} 
\right). 
$$
This formula defines a 5-dimensional pencils of metrics (with parameters $a_1,a_2, b_0,b_1, b_2$).
For any choice of the parameters such that $g$ is nondegenerate, the coordinates  $u^i$ are Frobenius for it in the sense of Definition \ref{def:frobcoord}.  

From the algebraic viewpoint,  we may equivalently think of this formula as 5-parametric family  (pencil) of Frobenius algebras $(\mathfrak a, b)$.
The entries of $g$ define the structure constants of $\mathfrak a$. For instance,  $g^{11}=a_{2}u^{1}+a_{1}$ and 
  $g^{34} = -b_0 u^2 - b_2 u^4$ imply 
$$
e^1 \star e^2= a_2 e^1 \  \ \textrm{and} \   e^3 \star e^4 = -b_0 e^2 - b_2 e^4
$$
for a basis $e^1, e^2, e^3, e^4$ of $\mathfrak a$.  The matrix $(b^{ij})$ of the corresponding Frobenius form  $b$ is obtained from that of  $g$ by assigning to $u^i$  any constant values  $u^i = m^i\in \R$  (such that $b$ is non-degenerate for generic choice of 
$a_1,a_2, b_0,b_1, b_2$). To get a Frobenius pencil, the constants $m^i$ should be the same for all parameters  $a_1,a_2, b_0,b_1, b_2$.

In the coordinates $(u^1,...,u^4)$  the operators $L_1$ and  $L_2$  are 
 given by the matrices  
 $$
 L_1=\left( \begin {array}{cccc} u^{{1}}&1&0&0\\ \noalign{\medskip}u^{{2}}
&0&0&0\\ \noalign{\medskip}u^{{3}}&0&0&0\\ \noalign{\medskip}u^{{4}}&0
&0&0\end {array} \right)\ , \ \ 
L_2= \left( \begin {array}{cccc} 
0&0&0&0\\ 
\noalign{\medskip}0&0&0&0\\ 
\noalign{\medskip}0&{\frac {-u^{{4}}u^{{2}}-(u^{3})^{2}}{(u^{2})^{2}}}&{\frac {u^{{3}}}{u^{{2}}}}&1\\ 
\noalign{\medskip}0&-{\frac {u^
{{4}}u^{{3}}}{(u^{2})^{2}}}&{\frac {u^{{4}}}{u^{{2}}}}&0\end {array}
 \right).
 $$
The matrices of $g_1^{\LC}$ and $g_2^{\LC}$ are  
$$
g_1^{\LC}= \left( \begin {array}{cccc} 0&1&{\frac {u^{{3}}}{u^{{2}}}}&{\frac {u^
{{4}}}{u^{2}}}\\ 
\noalign{\medskip}1&-u^{{1}}&-{\frac {u^{{3}}u^{{1}
}}{u^{{2}}}}&-{\frac {u^{{1}}u^{{4}}}{u^{{2}}}}\\ 
\noalign{\medskip}{
\frac {u^{3}}{u^{2}}}&-{\frac {u^{3}u^{1}}{u^{2}}}&-{\frac {
(u^{3})^{2}u^{{1}}}{(u^{2})^{2}}}&-{\frac {u^{1}u^{4}u^{3}}{
(u^{2})^{2}}}\\ 
\noalign{\medskip}{\frac {u^{{4}}}{u^{{2}}}}&-{
\frac {u^{{1}}u^{{4}}}{u^{{2}}}}&-{\frac {u^{{1}}u^{{4}}u^{{3}}}
{(u^{2})^{2}}}&-{\frac {u^{1}(u^{4})^2}{(u^{2})^{2}}}\end {array}
 \right)\ , \ \ 
 g_2^{\LC}= \left( \begin {array}{cccc} 0&0&0&0\\ 
 \noalign{\medskip}0&0&0&0
\\ \noalign{\medskip}0&0&0&-u^{{2}}\\ \noalign{\medskip}0&0&-u^{{2}}&u^{{3}}\end {array} \right) .
$$

}\end{Example}


\section{ Proof of Theorems \ref{thm:frobenius0} and  \ref{prop:bols13}}\label{sect:5}

\begin{proof} [Proof of Theorem \ref{thm:frobenius0}]
We assume that $\mathcal B_h + \mathcal A_g$ and $\mathcal B_{\bar h} + \mathcal A_{\bar g}$ are compatible with the additional condition that the eigenvalues of $R_h = \bar h h^{-1}$ are pairwise different. We also assume that $h + \bar h$ is nondegenerate.

Recall that Theorem 7.1 from \cite{doyle} implies that $\mathcal B_h$ and $\mathcal B_{\bar h}$ are compatible Poisson structures (item (i) in Fact \ref{fact:1}). Let $\Gamma^{\alpha \beta}_s$ and $\bar \Gamma^{\alpha \beta}_s$ denote the contravariant Levi-Civita connections of $h$ and $\bar h$. From Theorem 3.2 in \cite{doyle} applied to $\mathcal B_h + \mathcal B_{\bar h}$, it follows that the connection $\widehat \Gamma^{\beta}_{qs}$  defined from
$$
\Gamma^{\alpha \beta}_s + \bar \Gamma^{\alpha \beta}_s = (h + \bar h)^{\alpha q} \widehat \Gamma^{\beta}_{qs}
$$
is symmetric and flat.
 
 By direct computation $\widehat \nabla (h + \bar h) = \nabla h + \bar \nabla \bar h = 0$, so that  $\widehat \Gamma$ is the Levi-Civita connection for $h + \bar h$ and moreover, $h + \bar h$ is flat.
According to Theorem 6.2 in \cite{doyle}, this implies that $\mathcal B_h + \mathcal B_{\bar h}$  is  Darboux-Poisson (i.e., is given by \eqref{intro:third}).  Hence, in our notations,  we obtain the formula 
\begin{equation}\label{s1}
\mathcal B_h + \mathcal B_{\bar h} = \mathcal B_{h + \bar h}.    
\end{equation}

Setting  $\widehat \Gamma^{\alpha \beta}_s = (h + \bar h)^{\alpha q} \widehat \Gamma^{\beta}_{qs}$ to be the contravariant Levi-Civita connection of $h + \bar h$,   we get 
\begin{equation}\label{s2}
    \Gamma^{\alpha \beta}_s + \bar \Gamma^{\alpha \beta}_s = \widehat \Gamma^{\alpha \beta}_s,
\end{equation}
and conclude that $h$ and $\bar h$ are Poisson compatible in the sense of Definition \ref{def:compmetrics} (in particular, this proves the (ii)-part of Fact \ref{fact:2}).  Hence, $R_h= \bar h h^{-1}$ is a Nijenhuis operator (Fact \ref{fact:3}).

For a pair of flat metrics $h$ and $\bar h$, introduce the so-called obstruction tensor 
$$
S^{\beta}_{rq} = \Gamma^\beta_{rq} - \bar \Gamma^\beta_{rq}.
$$
It vanishes if and only if $h, \bar h$ can be brought to constant form simultaneously (thus, the name). It is obviously symmetric in lower indices. Condition \eqref{s2} can be written in equivalent form (\cite{mok2}, Lemma 3.1 and Theorem 3.2) in terms of only $\Gamma^{\alpha \beta}_s, \bar \Gamma^{\alpha \beta}_s, h, \bar h$
\begin{equation}\label{sr}
    \bar{\Gamma}^{\alpha \beta}_q h^{q \gamma} - \bar{\Gamma}^{\gamma \beta}_q h^{q \alpha} + \Gamma^{\alpha \beta}_q \bar{h}^{q \gamma} - \Gamma^{\gamma \beta}_q \bar{h}^{q \alpha} = 0.
\end{equation}
After lowering both indices with $h$ and rearranging the terms we get
\begin{equation}\label{s3}
    S^{\beta}_{pq} R^q_s - R^q_p S^{\beta}_{qs} = 0.
\end{equation}
For a given metric $h$ and its Levi-Civita connection, define
$$
c^{\alpha \beta}_{rs} = h^{\alpha q} \Bigg(\Gamma_{qr}^a \Gamma_{as}^{\beta} - \pd{\Gamma^{\beta}_{qs}}{u^r} \Bigg).
$$
The $\bar c^{\alpha \beta}_{rs}, \widehat c^{\alpha \beta}_{rs}$ for $\bar h$ and $h + \bar h$ are defined in a similar way. This formula is one `half' of the formula for Riemann curvature tensor and the flatness of the metrics implies that $c^{\alpha \beta}_{rs} = c^{\alpha \beta}_{sr}$ (and similarly for metrics $\bar h, h + \bar h$).  Using this symmetry in lower indices, we apply the general formula  \eqref{intro:third}  to the Poisson structures in   \eqref{s1} and collect coefficients in front of $D^2$ to get
$$
3 c^{\alpha \beta}_{rs} u^r_x u^s_x - 3 \Gamma^{\alpha \beta}_s u^s_{xx} + 3 \bar c^{\alpha \beta}_{rs} u^r_x u^s_x - 3 \bar \Gamma^{\alpha \beta}_s u^s_{xx} = 3 \widehat c^{\alpha \beta}_{rs} u^r_x u^s_x - 3 \widehat \Gamma^{\alpha \beta}_s u^s_{xx}.
$$
Collecting all the terms  with $u^r_x u^s_x$ in this differential polynomial, in turn, implies 
$$
\widehat c^{\alpha \beta}_{rs} - c^{\alpha \beta}_{rs} - \bar c^{\alpha \beta}_{rs} = 0.
$$
Using the characteristic property of the Levi-Civita connection
$$
\pd{h^{\alpha \beta}}{u^s} + h^{\alpha q} \Gamma^{\beta}_{qs} + \Gamma^{\alpha}_{qs} h^{q \beta} = 0
$$
we rewrite $c^{\alpha \beta}_{rs}$ as
\begin{equation}\label{s4}
    c^{\alpha \beta}_{rs} = h^{\alpha q} \Bigg(\Gamma_{qr}^a \Gamma_{as}^{\beta} - \pd{\Gamma^{\beta}_{qs}}{u^r} \Bigg) = - \pd{}{u^r} \Big[ h^{\alpha q} \Gamma^{\beta}_{qs}\Big] - h^{pq} \Gamma^{\alpha}_{qr} \Gamma^{\beta}_{ps}
\end{equation}
Applying \eqref{s2}, \eqref{sr} and \eqref{s4} yields
\begin{equation*}
\begin{aligned}
0 & = \big( \widehat{c}^{\alpha \beta}_{rs} - c^{\alpha \beta}_{rs} - \bar{c}^{\alpha \beta}_{rs} \big) (h + \bar h)^{s \gamma} = \Big(h^{pq} \Gamma^{\alpha}_{qr} \Gamma^{\beta}_{ps} + \bar{h}^{pq} \bar{\Gamma}^{\alpha}_{qr} \bar{\Gamma}^{\beta}_{ps} - (h + \bar h)^{pq} \widehat{\Gamma}^{\alpha}_{qr} \widehat{\Gamma}^{\beta}_{ps} \Big)(h + \bar h)^{s \gamma} = \\
& = h^{pq} \Gamma^{\alpha}_{qr} \Gamma^{\beta}_{ps} h^{s\gamma} + h^{pq} \Gamma^{\alpha}_{qr} \Gamma^{\beta}_{ps} \bar{h}^{s \gamma} + \bar{h}^{pq} \bar{\Gamma}^{\alpha}_{qr} \bar{\Gamma}^{\beta}_{ps}h^{s\gamma} + \bar{h}^{pq} \bar{\Gamma}^{\alpha}_{qr} \bar{\Gamma}^{\beta}_{ps}\bar{h}^{s\gamma} - \big( h^{pq}\Gamma_{qr}^{\alpha} + \bar{h}^{pq}\bar{\Gamma}_{qr}^{\alpha}\big) \big( \Gamma^{\beta}_{ps} h^{s\gamma} + \bar{\Gamma}^{\beta}_{ps} \bar{h}^{s \gamma}\big) = \\
& = \Gamma^{\alpha}_{qr} \big( h^{pq} \Gamma^{\beta}_{ps} \bar{h}^{s \gamma} -  h^{pq} \bar{\Gamma}^{\beta}_{ps} \bar{h}^{s \gamma}\big) - \bar{\Gamma}^{\alpha}_{qr} \big( \bar{h}^{pq} \Gamma^{\beta}_{ps} h^{s \gamma} -  \bar{h}^{pq} \bar{\Gamma}^{\beta}_{ps} h^{s \gamma}\big) = \\
& = \big( S^{\alpha}_{qr} + \bar{\Gamma}^{\alpha}_{qr}\big)\big( h^{pq} \Gamma^{\beta}_{ps} \bar{h}^{s \gamma} -  h^{pq} \bar{\Gamma}^{\beta}_{ps} \bar{h}^{s \gamma}\big) - \bar{\Gamma}^{\alpha}_{qr} \big( \bar{h}^{pq} \Gamma^{\beta}_{ps} h^{s \gamma} -  \bar{h}^{pq} \bar{\Gamma}^{\beta}_{ps} h^{s \gamma}\big) = \\
& = S^{\alpha}_{qr} h^{pq} S^{\beta}_{ps} \bar{h}^{s \gamma} + \bar{\Gamma}^{\alpha}_{qr} \big( h^{pq} \Gamma^{\beta}_{ps} \bar{h}^{s \gamma} -  h^{pq} \bar{\Gamma}^{\beta}_{ps} \bar{h}^{s \gamma} - \bar{h}^{pq} \Gamma^{\beta}_{ps} g^{s \gamma} +  \bar{h}^{pq} \bar{\Gamma}^{\beta}_{ps} h^{s \gamma} \big) = \\
& = S^{\alpha}_{qr} h^{pq} S^{\beta}_{ps} \bar{h}^{s \gamma}.
\end{aligned}
\end{equation*}
Now consider the coordinate system in which the Nijenhuis operator $R_h$ is diagonal.  As $R_h$ by definition is self-adjoint with respect to both $h$ and $\bar h$, we get that both $h$ and $\bar h$ are also diagonal. Condition \eqref{s3} implies that for given $\beta$ the only non-zero elements of $S^{\beta}_{pq}$ are the ones that stand on the diagonal. The previous calculation yields 
$$
S^{\alpha}_{qr} h^{pq} S^{\beta}_{ps} \bar{h}^{s \gamma} = 0
$$
which, for  fixed $\alpha$ and $\beta$, is just the product of four diagonal matrices, two of which are nondegenerate. Taking $\alpha = \beta$ we see that the matrix $S^{\alpha}_{qr}$ must be zero. As $\alpha$ is arbitrary, this implies that the obstruction tensor vanishes and $h, \bar h$ have common Darboux coordinates.

Fix the coordinates in which both $h$ and $\bar h$ are flat.  Applying Fact \ref{fact:5}, we see  that these coordinates are  Frobenius for both $g$ and $\bar g$. Using \eqref{s1} and applying Fact \ref{fact:5} to the sum of our Poisson structures, we get that $a^{\alpha \beta}_s + \bar a^{\alpha \beta}_s$ define a commutative associative algebra, while $b^{\alpha \beta}_s + \bar b^{\alpha \beta}_s$ and $h^{\alpha \beta}_s + \bar h^{\alpha \beta}_s$ are Frobenius forms for this algebra, as required.

The inverse statement immediately follows from Facts \ref{fact:4} and \ref{fact:5}. \end{proof}

\begin{proof} [Proof of Theorem \ref{prop:bols13}]

Consider a pair of compatible flat metrics $g, \bar g$ in common Frobenius coordinates $u^1,\dots, u^n$
$$
g^{\alpha\beta}(u) =b^{\alpha\beta} +  a^{\alpha\beta}_s u^s \quad\mbox{and}\quad   \bar g^{\alpha\beta}(u) = \bar b^{\alpha\beta} +  \bar a^{\alpha\beta}_s u^s,
$$

Fact \ref{fact:4} implies that $- \frac{1}{2}a^{\alpha \beta}_s$ and $-  \frac{1}{2}\bar a^{\alpha \beta}_s$ are the contravariant Christoffel  symbols for $g$ and $\bar g$ respectively. Compatibility of $g$ and $\bar g$ means that the contravariant Levi-Civita symbols for the flat metric $g + \bar g$ are the sum of the corresponding symbols for $g$ and $\bar g$, that is, $- \frac{1}{2} a^{\alpha \beta}_s-\frac{1}{2} \bar a^{\alpha \beta}_s$. At the same time these symbols are constant and symmetric in upper indices. Hence the coordinates $u^1,\dots, u^n$ are Frobenius for $g + \bar g$ (Fact \ref{fact:4}).

This, in turn, implies that $a^{\alpha \beta}_s + \bar a^{\alpha \beta}_s$ are the structure constants of a commutative associative algebra and $b^{\alpha \beta}_s + \bar b^{\alpha \beta}_s$ is one of its Frobenius form. Thus, the corresponding Frobenius algebras are compatible.

As $g$ and $\bar g$ are both nondegenerate metrics, this implies that for a generic collection of constants $m^0, m^1, \dots, m^n$, the bilinear forms
$$
h^{\alpha \beta} =m^0 b^{\alpha \beta} + a^{\alpha \beta}_s m^s \text{    and    } \bar h^{\alpha \beta} = m^0 \bar b^{\alpha \beta} + \bar a^{\alpha \beta}_s m^s
$$
are both nondegenerate too. At the same time, each of them is the sum of a Frobenius form ($m^0 b^{\alpha \beta}$ and resp. $m^0 \bar b^{\alpha \beta}$) and trivial form ($a^{\alpha \beta}_s m^s$ and  resp. $\bar a^{\alpha \beta}_s m^s$), which corresponds to $m \in \mathfrak a^*$ with coordinates $m^1, \dots, m^n$ and, thus, is also Frobenius\footnote{Here we use a well known fact for any $m\in \mathfrak a^*$, the form $\xi,\eta \mapsto \langle \xi \star \eta, m\rangle$ is Frobenius, perhaps degenerate. If $\mathfrak a$ has a unity element, then every Frobenius form is of this kind. Otherwise, there might exist other (nontrivial) Frobenius forms.}. As a result, $h$ and $\bar h$ lead us to Frobenius triples $(h, b, \mathfrak a)$ and $(\bar h, \bar b, \bar{\mathfrak a}.)$.

By construction, $h + \bar h$  defines (if nondegenerate) a Frobenius form for the `sum' of the algebras. Thus, we get  compatible Frobenius triples, which yield compatible non-homogeneous Poisson structures $\Pthree_h + \Pone_g$ and $\Pthree_{\bar h}+\Pone_{\bar g}$.  \end{proof}

\section{Proof of Theorem \ref{thm:frobenius1}}  \label{sect:6}

\subsection{Rewriting the existence of Frobenius coordinates    in  a differential-geometric form} \label{subsect:6.1}

We start with the following observation related to  Frobenius coordinate systems (Fact \ref{fact:4}):  $(u^1,\dots, u^n)$  is a Frobenius coordinate system for a metric $g$  if and only if the contravariant Christoffel symbols $\Gamma^{ij}_k = \sum_s g^{si} \Gamma^j_{sk}$ in this coordinate system are constant and symmetric in upper indices.

We denote by $\Gamma, \bar \Gamma$ the Levi-Civita connections of $g$ and $\bar g$.  Assuming that a common Frobenius coordinate system $u^1, \dots, u^n$ exists,  we let $\widehat \Gamma$ be the flat connection whose Christoffel symbols identically vanish in this coordinate system. 
Let $R^i_{\ jk\ell}$, $\bar R^i_{\ jk\ell}$, and $\widehat R^i_{\ jk\ell}$ denote the corresponding curvature tensors. We assume $n\ge 2$, the case $n=1$ is trivial.

Consider the tensors 
\begin{eqnarray*}
S^{ij}_{\ \  k} &:= & \sum_s g^{si} \left(\Gamma^j_{sk} - \widehat \Gamma^j_{sk}\right)   \\
\bar S^{ij}_{\ \ k} &:= & \sum_s \bar g^{si} \left(\bar \Gamma^j_{sk} - \widehat \Gamma^j_{sk}\right)   . 
\end{eqnarray*}

In terms of these tensors, the necessary and sufficient conditions   that the   connection $\widehat \Gamma$ determines Frobenius coordinates are:
\begin{eqnarray}
\label{eq:matv46} 0&=& \widehat R^i_{\ jk\ell} = R^i_{\ jk\ell}= \bar R^i_{\ jk\ell}\\\
\label{eq:matv47} S^{ij}_{\ \ k} &=& S^{ji}_{\ \ k} \label{eq:S} \\ 
\label{eq:matv48} \bar S^{ij}_{\ \ k} &=& \bar S^{ji}_{\ \ k}  \label{eq:barS} \\
\label{eq:matv49} 0&=& \widehat \nabla_m  S^{ij}_{\ \  k}=\widehat \nabla_m  \bar S^{ij}_{\ \  k}. 
\end{eqnarray}

Indeed,  if $(u^1,\dots, u^n)$ is a common Frobenius coordinate system for $g$ and $\bar g$, then in these coordinates $\widehat \Gamma^j_{sk}=0$, 
and  $\Gamma^{ij}_k= g^{is} \Gamma^j_{sk}$  and  $\bar\Gamma^{ij}_k= \bar g^{is} \bar\Gamma^j_{sk}$   are both constant and symmetric in upper indices  by 
Fact \ref{fact:4}. Hence,  \eqref{eq:matv46}-\eqref{eq:matv49} obviously follow.

Conversely,  if \eqref{eq:matv46}-\eqref{eq:matv49} hold, then in the flat coordinates for $\widehat \Gamma^j_{sk}$ we see that $\Gamma^{ij}_k = S^{ij}_{\ \ k}$ and $\bar \Gamma^{ij}_k =\bar S^{ij}_{\ \  k}$ are both symmetric in upper indices due to \eqref{eq:matv47} and \eqref{eq:matv48} and are also constant due to \eqref{eq:matv49}.  Therefore, by Fact \ref{fact:4},   $(u^1,\dots, u^n)$ are Frobenius coordinates for both $g$ and $\bar g$.

\subsection{ General form of the metric in diagonal coordinates} \label{subsect:6.2} 

We work in the coordinates $(x^1,...,x^n)$ such that   \begin{equation}\label{eq:data} 
R_g= \bar g g^{-1} = \textrm{diag}(\ell_1(x^1),...,\ell_n(x^n)) \ , \ \  g_{ij}= \textrm{diag}(\varepsilon_1 e^{g_1},...,\varepsilon_n e^{g_n}),  \end{equation}
  where $g_i$ are local functions on our manifold and $\varepsilon_i\in \{-1,1\}$. Local existence of such coordinates follows from Facts \ref{fact:2} and \ref{fact:3} which imply that  $R_g$ is a Nijenhuis operator and therefore,  according to Haantjes theorem, is diagonalisable and $\ell_i$ depends on $x^i$ only (see also various versions of diagonalisability theorems in \cite{nij1} which, in particular, allows us to include the case of complex eigenvalues too).  	We assume  that all $\ell_i(x^i)$ are different and never vanish.

\begin{Remark}{\rm 
We allow some of the {\it diagonal}  variables  $x^i$ to be  complex. Note that if a variable $x^i$ is complex then by 
 \cite[\S 3]{nij1}  we may assume that 
the corresponding eigenvalue $\ell_i$ is a holomorphic function of $x^i$.   In the first read, we recommend to think of all the eigenvalues as real and then to carefully check that our  proofs  are based on algebraic manipulations and differentiations, which are perfectly defined over complex coordinates, so that  generalisation of the proofs  to complex eigenvalues requires  no change in formulas. See also a discussion at the end of \cite[\S 7]{NijenhuisAppl2}.

Note that  the results we use (e.g. \cite{Prvanovic, Solodovnikov}) are also based on  algebraic manipulations (essentially, on a careful  calculation of the curvature tensor and connection coefficients)  and are applicable if a part of eigenvalues is complex.

Note also that when we work over complex numbers, we may think that  the numbers $\varepsilon_i$ are all equal to $1$. If all the eigenvalues of $R_g$ are real,   objects we will introduce in the proof will automatically be real as well. 
}\end{Remark}

Let us first consider the conditions (\ref{eq:S}, \ref{eq:barS}). We view them as linear (algebraic) system of equations with unknown $\widehat \Gamma^i_{jk}$'s (satisfying also $\widehat \Gamma^i_{jk}=\widehat \Gamma^i_{kj}$)  whose coefficient matrix is constructed from the entries of  $g $ and $L$ and the free terms are constructed from $g, \bar g , \Gamma, \bar \Gamma$.
Being  rewritten in such a way that unknowns are on the left hand side and free terms are on the right hand side, it has the following  form:

\begin{equation} 
\label{eq:oben}
\begin{array} {rcl} 
\varepsilon_i e^{-g_i} \widehat \Gamma^j_{ik}-  
\varepsilon_j e^{-g_j} \widehat \Gamma^i_{jk} & = &  
\varepsilon_i e^{-g_i}   \Gamma^j_{ik}-  
\varepsilon_j e^{-g_j}  \Gamma^i_{jk}\, , \\ 
\ell_i  
\varepsilon_i e^{-g_i} \widehat \Gamma^j_{ik}- \ell_j  
\varepsilon_j e^{-g_j} \widehat \Gamma^i_{jk} & = & \ell_i  
\varepsilon_i e^{-g_i}   \bar  \Gamma^j_{ik}- \ell_j  
\varepsilon_j e^{-g_j}  \bar \Gamma^i_{jk} \, . \end{array}
\end{equation}

We see that for fixed  $i=j=k$, the system bears no information. For fixed   $i\ne j$,  the coefficient matrix $\begin{pmatrix}
 \varepsilon_i e^{-g_i}  &   - \varepsilon_j e^{-g_j}  \\  \ell_i \varepsilon_i e^{-g_i} & - \ell_j \varepsilon_j e^{-g_j}  
\end{pmatrix}$ of the linear system   (\ref{eq:oben}) is  nondegenerate,
since the eigenvalues $\ell_i$ are all different, and therefore the system has a unique solution.

The entries of the connections $\Gamma$ and $\bar \Gamma$ of the diagonal metrics $g_{ij}$ and $\bar g_{ij}:= gL^{-1}$
were calculated many times in the literature, see e.g.
\cite[Lemma 7.1]{NijenhuisAppl2},  and are given by the following formulas: 
\begin{itemize}

    \item $\Gamma^k_{ij} =\bar \Gamma^k_{ij}= 0$ for pairwise different $i, j$ and $k$, 

    \item $\Gamma^k_{kj} = \frac{1}{2}  \frac{\partial g_k}{\partial x^j} $ for arbitrary $k, j$,

    \item $\Gamma^k_{jj} = - \frac{\varepsilon_j}{\varepsilon_i}  \frac{e^{g_j-g_k}}{2} \frac{\partial g_j}{\partial x^k} $ for arbitrary  $k \neq j$,

    \item $\bar \Gamma^k_{kj} = \Gamma^k_{kj}  $ for arbitrary $k\ne  j$,

    \item $\bar \Gamma^i_{ii} = \Gamma^i_{ii}- \frac{\ell_i'}{2\ell_i}  $

    \item $\bar \Gamma^k_{jj} = \frac{\ell_k}{\ell_j}\Gamma^k_{jj}  $ for arbitrary  $k \neq j$.

\end{itemize}

 By direct calculations using these formulas   we obtain that the solution of the  system \eqref{eq:oben} is as follows:
\begin{itemize} 

\item[(A)] 
$\widehat \Gamma^{i}_{ii}  =  u_i $ for all $i$, where $u_i$'s are (yet unknown) functions on the manifold.

\item[(B)]
$\widehat \Gamma^{i}_{ij}= \widehat \Gamma^{i}_{ji}= \frac{\partial g_i}{\partial x^j}$  for $i\ne j$ 

\item[(C)]
$\widehat \Gamma^{i}_{jk}= 0 $  for  all $i\ne j $  and  $k\ne i$ (we allow the case $k=j$).
\end{itemize}

Combining  these with the formulas for $S^{ij}_{\ \ k}$ we obtain:

\begin{itemize} 

\item $S^{ii}_{\ \ i} =  \varepsilon_i e^{-g_i}\left(u_i - \frac{1}{2}\frac{\partial g_i}{\partial x^i}\right)$ for all $i$, 

\item $S^{i i}_{\ \ j} =  \frac{\varepsilon_i}{2}  e^{-g_i}  \frac{\partial g_i}{ \partial x^j},   $  for all $i \ne j$, 

\item $S^{i j}_{\ \  i} = S^{ji }_{\ \  i} =  \frac{\varepsilon_j}{2} e^{-g_j}\frac{\partial g_i}{\partial x^j}$ for all  $i \ne j$, 

\item $S^{i j}_{\ \  k} = 0 $ for all  $i\ne j\ne k\ne i$.
\end{itemize} 

By direct calculations we see that for any $i\ne j \ne k \ne i$ we have 
$\widehat \nabla_k S^{ij}_{\ \  j}= \frac{\varepsilon_i}{2}\frac{\partial^2 g_i}{\partial x^j \partial x^k}$ implying 
\begin{equation} \label{eq:liouville} 0= \frac{\partial^2 g_i}{\partial x^j \partial x^k}. \end{equation}

Next, consider the terms of the form $\widehat \nabla_jS^{ii}_{\ \ i}$ and $\widehat \nabla_iS^{ii}_{\ \ j}$ with $i\ne j$. They are given by 
$$
\widehat \nabla_jS^{ii}_{\ \ i}= 
\varepsilon_i \frac{e^{-g_i}}{2} \left( \frac{\partial g_i}{\partial  x^j }   \frac{\partial g_j}{\partial  x^i}  + \frac{\partial^2  g_i}{\partial  x^i \partial x^j}\right),
$$
$$
\widehat \nabla_iS^{ii}_{\ \ j}=- \varepsilon_i \frac{e^{-g_i}}{2} \left( \frac{\partial g_i}{\partial  x^j }   \frac{\partial g_j}{\partial  x^i}  + \frac{\partial^2  g_i}{\partial  x^i \partial x^j}- 2 \frac{\partial u_i}{\partial x^j}\right).
$$
Since    $\widehat \nabla_jS^{ii}_{\ \ i}=\widehat \nabla_iS^{ii}_{\ \ j} =0$, the  formulas  above    imply  \begin{equation} \label{eq:u} \frac{\partial u_i}{\partial x^j}=0\end{equation} so each $u_i$ is a function of $x^i$ only.

Next, we prove the following Lemma.  Denote by $U_i= U_i(x^i)$ and $U_j= U_j(x^j)$  the primitive functions for  $ e^{\tilde u_i}$ and $  e^{\tilde u_j}$, where
$\tilde u_i$ and $\tilde u_j$ are primitive functions for $u_i$ and $u_j$. By their definition, $U'_i\ne 0$ and $U'_j\ne 0.$

\begin{Lemma} \label{Lem:ansatz}
There exist constants  $C_{ij}$  and $E_{ij}= E_{ji}$ 
such that for the constants $\alpha_{ij}\in \{0, 1\}$ given by the formula 
$$
\alpha_{ij}=\alpha_{ji}= \left\{ \begin{array}{cc} 0 & \textrm{ \ if $C_{ij}=C_{ji}=0$\ }\\
1 & \textrm{ otherwise } \end{array}\right.$$  and 
for any $i\ne j$  the function $$
g_i- \ln\left(|C_{ij} U_i + C_{ji} U_j + E_{ij} |^{\alpha_{ij}} \right)  
$$
does not depend on $x^j$ (we use the convention that $0^0=1$).
\end{Lemma}

\begin{proof}
We consider the curvature tensor $\widehat R^i_{\ j k \ell} $ of the connection $\widehat \Gamma$. To compute it, we need to substitute $\widehat \Gamma$ given by (A,B,C) above into   
 the standard formula for the curvature
\begin{equation}  \label{standard} 
   \widehat R^\ell_{\ ijk} = \tfrac{\partial  }{\partial x^j} \widehat \Gamma^\ell_{ik} - \tfrac{\partial  }{\partial x^k} \widehat \Gamma^\ell_{ij}+  \sum_{s}\left(\widehat \Gamma^\ell_{js} \widehat \Gamma^s_{ik} - \widehat \Gamma^\ell_{ks} \widehat \Gamma^s_{ij}\right).  \end{equation}
We obtain for $i\ne j$:
\begin{eqnarray}
0=\widehat R^{i}_{\ i j i}  & = & - {\frac {\partial {g_i}}{\partial x^{{j}}}} \frac {\partial {g_j}}{\partial 
x^{{i}}}   -{\frac {\partial^{2} g_i  }{\partial x^{{j}}
\partial x^{{i}}}}  
\\
0=\widehat R^{i}_{\ j j i} &= &  
-  {\frac {\partial  g_i}{\partial x^{{j}}}} u_j +         \left( \frac {\partial { g_i} }{\partial x^{{j
}}}\right)^2+{\frac {\partial^2 g_i}{\partial {{x^{{j}}}^2 }   }}\\
0=\widehat R^{j}_{\ j j i} &= &  
  \frac {\partial g_i}{\partial x^{{j}}}   \frac {\partial g_j }{\partial x
^{{i}}} +{
\frac {\partial^{2} g_j}{\partial x^{{i}}\partial x^{{j}} }}\\
0=\widehat R^j_{ \ i j i  } &= & 
- \left( {\frac {\partial g_j}{\partial x^{{i}}}}  \right) ^{2}+ {u_i}  {\frac {
\partial g_j}{\partial x^{{i}}}}   \  -{\frac {\partial ^{2} g_j }{\partial { {x^{{i}}}^2   }}}.\end{eqnarray} 
We view 4 equations above a system of PDEs on  the unknown functions 
\begin{equation} \label{eq:sub} a= \frac{\partial g_i}{\partial x^j} \  \textrm{and} \   b= \frac{\partial g_j}{\partial x^i}.\end{equation} 
The condition \eqref{eq:u} implies that  the coefficients of the system depend on $x^i$ and $x^j$ and we may temporarily 
`forget' all other variables.  The system then has the following form:
\begin{equation}\label{eq:pdeab} 
  \frac {\partial a}{\partial x^{i}}  =-a  b  \ , \ \frac {\partial a }{\partial x^{{j}}}  =- a^{2}+a
  {u_j} \ , \ 
\frac{\partial  b}{\partial x^i} =
-  b^{2}+b u_i \  , \ \frac{\partial b}{\partial x^j}  =-a  b . 
\end{equation}
This  system is of Cauchy-Frobenius type (in the sense that all  first derivatives of unknown functions  are explicit expressions of the unknown functions and variables).   By direct computation we check that its integrability conditions hold identically.  Then, 
its solution depends on an  arbitrary choice of the values of $a$ and $b$ at one  arbitrarily chosen point   $p_0$. 
 Note that if $a(p_0)= 0$ then  $a$ is identically $0$, the same is true for $b$.

 By direct substitution we see that  for any constants  $C_i,C_j,E$,  the pair of functions 
\begin{equation}
\label{eq:sola} 
a= \frac{C_j U'_j}{C_i U_i + C_i U_j +E}, \quad b= \frac{C_i U'_i}{C_i U_i + C_j U_j+E}
\end{equation}  
satisfies the equation. In the case  $C_i=C_j=0$ we think that $a$ and $b$ given by \eqref{eq:sola} vanish identically.  
 The functions $U_i= U_i(x^i)$ and $U_j= U_j(x^j)$ used in \eqref{eq:sola}  are as defined before Lemma \ref{Lem:ansatz}. 
 By varying the constants $C_i,C_j, E$, one  can get any nonzero  initial values $a(p_0), b(p_0)$  so this is indeed a general solution.

In the case $C_i\ne 0$ or $C_j\ne 0$,    using \eqref{eq:sub}, we obtain  
 \begin{equation}
 \label{eq:solg}
g_i= \ln (|C_i U_i + C_j U_j +E |)  + D_i  \ \textrm{ and }  \   g_j= \ln (|C_i U_i + C_j U_j+E|)  + D_j
\end{equation}
with $D_i$ independent of $x^j$ and  $D_j$ independent of
$x^i$. In the case $C_i=0=C_j$ we obtain that  $g_i$ is independent of $x^j$ and  $g_j$  is independent of
$x^i$ automatically.

Note that if we `remember' all the coordinates, then $C_i$,  $D_i$, $E$   may also  depend on all other  variables  
$x^k, $  $k\not\in \{ i, j\}$.

Let us study the dependence of $C_i$,  $C_j$ and $E$ on the variable $x^k$ with $k \not\in \{i,j\}$. We first consider the case when $C^i\ne 0$ and $C^j\ne 0$.  
We  observe that by \eqref{eq:liouville},  for $i\ne j\ne k\ne i$,  we have 
$$ 
0=\frac{\partial^2   g_{i}}{\partial x^j \partial x^k}\stackrel{(\ref{eq:sub}, \ref{eq:sola})}{=} 
 \frac{\partial }{\partial x^k}\frac{C_j U'_j}{C_i U_i + C_j U_j + E}= 
-\frac{U'_j\left(U_i\frac{\partial }{\partial x^k}\frac{C_i}{C_j} + \frac{\partial }{\partial x^k} \frac{E}{C_j}\right) }{(\tfrac{C_i}{C_j} U_i + U_j + \tfrac{E}{C_j})^2}
$$ 
implying that the  ratios  $C_i/C_j$  and  $ {E}/{C_j}$  are constant. Note that  $U'_j\ne 0$ since it is a primitive function for a nonvanishing function. Then, we may assume that $C_i$, $C_j$ and $E$ are constants, since in the formula \eqref{eq:solg} 
their dependence   on other variables can be hidden in $D_i$ and $D_j$.

In the  cases $C_i=0$ or $C_j=0$ (but not $C_i=0=C_j$ both)  the dependence of $C_i$,  $C_j$  and $E$  on other variables can similarly  be hidden in $D_i$ and $D_j$.  
In the remaining case, when $C_i=0=C_j$, we already know  that  $g_i$ is independent of $x^j$ and  $g_j$ is independent of
$x^i$.

Thus,  in all cases  $g_i- \ln (|C_i U_i + C_j U_j + E|^{\alpha_{ij}})$   does not depend on $x^j$ and $g_j- \ln (|C_i U_i + C_j U_j +E |^{\alpha_{ij}})$ does not depend on $x^i$.
 Since  constants   $C_i$, $C_j$ and $E$  are constructed for fixed $i$ and  $j$, in what follows we denote  them $C_{ij}$,   $C_{ji}$ and $E_{ij}=E_{ji}$ respectively. In this notation,  the function  $g_i- \ln (|C_{ij} U_i + C_{ji} U_j +E_{ij} |^{\alpha_{ij}})$ does not depend on $x^j$,  as stated.
\end{proof}

Consequently applying the Lemma, we see that 
$$
g_i -\sum_{s\ne i} \ln \left(|C_{ij} U_i + C_{is} U_s + E_{ij} |^{\alpha_{is}}\right) 
$$
depends on $x^i$ only.

Therefore, the $i$th diagonal component $g_{ii}$ of the  metric $g$ is as follows: 
\begin{equation}
\label{subg}
g_{ii}= \varepsilon_i e^{g_i} = h_i(x^i) \left(\prod_{s\ne i} \left( C_{is} U_i(x^i)+ C_{si} U_s(x^s)+ E_{is}\right)^{\alpha_{is}}  \right)
\end{equation}
for some functions $h_i$ of one variable.

 \subsection{Last step of the proof of Theorem \ref{thm:frobenius1}: making all $C_{ij}$ equal $\pm1$} \label{sec:LC} 

In the previous section we have proved that the metric $g$  is given by \eqref{subg}. Observe that  by the assumptions of Theorem \ref{thm:frobenius1}, the diagonal coordinates depend on all variables,  so  all $\alpha_{ij}=1$ and all $C_{ij}\ne 0$ for $i\ne j$. 
First note that  in the case when all $C_{ij}=1$ for $i<j$,  $C_{ij}= -1$ for $i>j$ and $E_{ij}=0$ for $i \ne j$, the diagonal
 metric $g$ is   in the so-called  {\it Levi-Civita} form: 
 \begin{equation}
 \label{eq:LCform}
g_{ii}=  \left(\prod_{j\ne i} ( U_i(x^i)- U_j(x^j)) \right)h_i(x^i).
\end{equation}

In this section we show  that   one can bring the metric to the  form \eqref{eq:LCform} by certain  `admissible' 
 operations which include only coordinate transformations and renaming of  functions. Combining this with a result of A. Solodovnikov (Fact \ref{rem:solodovnikov})  will prove Theorem \ref{thm:frobenius1}.

 We will use  the condition 
$
\widehat \nabla_i S^{ij}_{\ \ k}=0.
$ 
Assuming  $i\ne j \ne k \ne i$, this condition reads 
\begin{equation}
\label{eq:third} 
-\frac{\varepsilon_j}{2} e^{- g_j }   \left( \frac {\partial g_i }{\partial x^j
}  
\frac{\partial g_i}{\partial x^k}  -  \frac{\partial g_i}{\partial x^j}
  \frac {\partial g_j}{\partial x^k} -  \frac {\partial g_i }{\partial x^k} 
\frac {\partial g_k }{\partial x^j}  \right)=0.
\end{equation}
Substituting \eqref{subg} in it, we see that the  following relation holds: 
 \begin{equation}
 \label{eq:detijk} 
 0=C_{{ik}}C_{{ji}}C_{{kj}}-C_{{ij}}C_{{jk}}C_{{ki}}=\textrm{det} 
 \begin{pmatrix} 
 0 & C_{ij} & C_{ik}\\ 
-C_{ji} & 0 & C_{jk} \\ -C_{ki} & -C_{kj} & 0
\end{pmatrix}. 
\end{equation}

Let us now use the relation \eqref{eq:detijk} and `make'  $C_{ij}=1$ for $i<j$ and $C_{ij}= -1$ for $i>j$. 
 We will  use the following operations for it:
\begin{itemize}
\item[(a)] 
We can multiply the  factor $(C_{ij} U_i+ C_{ji} U_j+ E_{ij})$ in  the  $i$th and $j$th  diagonal components of $g$  (given by  \eqref{subg})  by a   nonzero  constant   and correspondingly change $h_i$ and 
 $h_j$ by dividing them  by the same constant.

\item[(b)] For every fixed pair $i\ne j$ if $C_{ij}\ne 0$ 
we  can rename $C_{ij} U_i$ by $U_i$ (which we will do if $i<j$)  or by $-U_i$ (if $i>j$).

\end{itemize}

By applying the operation (a), we  make $C_{1i}=1$ for all $i\ne 1$. By applying the operation (b), we  make   and $C_{i1}=-1$
 for all $i\ne 1$.  In addition, by applying operation (a) we  make $C_{2i}=1$ with $i>2$.   Then, the first two rows and the first column of the matrix $C_{ij}$ are  as we want. 
Then, the condition \eqref{eq:detijk} with $i=1, j=2$  arbitrary $k>2$  reads $C_{k2}=-1$
 so the second column is automatically   as we want.  Then, applying operation (a) we make all $C_{3k}$ with  $k>3$ equal to $1$.
 The  condition \eqref{eq:detijk} with $i=1, j=3$  arbitrary $k>3$  reads $C_{k3}=-1$
and implies that  the third  column is  as we want. 
Repeating the  procedure, we bring the metrics in  the Levi-Civita form \eqref{eq:LCform}.

Let us now   take  $U_i(x^i)$ as a local coordinate system: $ x^i_{\mathsf{new}}:=U_i(x^{i}_{\mathsf{old}})$. We can do it  because the derivative of $U_i$ is not zero.
	In the new coordinates $R_g$ is still diagonal and the $i^{\textrm{th}}$ diagonal component depend on the variable $i$ only. 
	In these coordinates, the  metric \eqref{eq:LCform}
	is diagonal with 
\begin{equation*}
g_{ii}= \left(\prod_{j\ne i}^n (  x^i-  x^j +E_{ij}) \right)\frac{1}{H_i(x^i)}, 
\end{equation*}
where $E_{ij}$ is now a skew-symmetric constant matrix. 
In these coordinates   the condition \eqref{eq:third} is equivalent to $E_{jk} + E_{ki} + E_{ij}=0$.  An easy exercise in Linear Algebra shows the  existence of 
 constants $E_1,...,E_n$ such  that $E_{ij}=E_i-E_j$. Moreover, the    constants $E_i$ are uniquely defined up to adding a common  constant $E$ to all of them.  Hence, after  the  coordinate change 
$x_{\textrm{new}}^i= x_{\textrm{old}}^i + E_i$,  the metric is diagonal with  
\begin{equation}
\label{eq:solodovnikov}
g_{ii}= \left(\prod_{j\ne i}^n (  x^i-  x^j ) \right)\frac{1}{H_i(x^i)}. 
\end{equation}

Therefore,  the  metric $\bar g= L \bar g$ is also diagonal with  
\begin{equation}
\label{eq:solodovnikovbar}
\bar g_{ii}=  \left(\prod_{j\ne i}^n (  x^i-  x^j) \right)\frac{1}{\bar H_i(x^i)} 
\end{equation}

\begin{Fact} \label{rem:solodovnikov}\label{fact6}
 The diagonal metric of form  \eqref{eq:solodovnikov} {\rm(}resp. \eqref{eq:solodovnikovbar}{\rm)} in dimension at least two 
  has  constant curvature  if and only if there exists a polynomial  $P$ {\rm(}resp. $Q${\rm)}  of degree $\le n+1$ such that 
	$H_i(x^i)=P(x^i)$  {\rm(}resp. $\bar H_i(x^i)= Q(x^i)${\rm)}. 
	Moreover, the curvature of the  metric   vanishes if and only if the polynomial $P$ {\rm(}resp. $Q${\rm)}   has  degree $\le n$.  
\end{Fact} 
Fact \ref{rem:solodovnikov} was proved in  \cite[\S 5]{Solodovnikov} and easily follows from  calculations in  \cite[\S 7]{NijenhuisAppl2}.

Taking $L=\textrm{diag}(x^1,...,x^n)$ and the (contravariant) Levi-Civita metric 
\begin{equation}
\label{eq:solodovnikovtilde}
g_{\LC}= \sum_{i} \left(\prod_{j\ne i}^n (  x^i-  x^j)\right)^{-1} \left(\tfrac{\partial}{\partial {x^i}} \right)^2,  
\end{equation}
we see that $g = P( L) g_{\LC}$ and $ \bar g = Q(L)  g_{\LC}$, which completes the proof of Theorem 
\ref{thm:frobenius1}.  (The `uniqueness' part of Theorem \ref{thm:frobenius1} will be explained in Remark \ref{last:step:proof}).

\section{  Proof of Theorem \ref{thm:frobenius2}} \label{sect:7}

\subsection{Upperblockdiagonal structure of the matrix $C_{ij}$}\label{subsect:7.1}

We assume that the metric $g$ is diagonal and its diagonal elements have the form  
\begin{equation}
\label{61}
g_{ii}=  h_i(x^i) \prod_{s\ne i} \left( C_{is} x^i+ C_{si} x^s+ E_{is}\right)^{\alpha_{is}}  , 
\end{equation}
which is the form 
\eqref{subg} in the `new' coordiante system $ x^i_{\mathsf{new}}:=U_i(x^{i}_{\mathsf{old}})$. 
  We view  $C_{ij}$ as entries of an  $n\times n$-matrix and $E_{ij}$ as entries of  an  $n\times n$ symmetric matrix.   
Since the   diagonal elements $C_{ii}, E_{ii}$  do not come into the formula for $g$, we assume they are zero.

We will first  re-arrange  the coordinates $x^1,...,x^n$  and make   the matrix
 $C$ upperblockdiagonal such that  in every diagonal block  all  nondiagonal entries are different from zero.  
We will need the following Lemma:

\begin{Lemma}  \label{Lem:2} 
If $C_{ji}=0$ for certain different $i,j\in \{1,...,n\}$,
 then  for any $k\in \{1,...,n\}$ we have $C_{jk} C_{ki}=0$. Moreover, if in addition $C_{ij}=0$, then for any   $k\in \{1,...,n\}$ we have 
$C_{ik}C_{jk}=0$.
\end{Lemma}

\begin{proof}
For $k=i$ or $k=j$ the statement follows from our convention $C_{ii}=C_{jj}=0$, further we assume $i\ne k\ne j(\ne i)$.  

 We consider  the equation  \eqref{eq:third}: under the assumption $C_{ji}=0$ the terms 
$ \frac {\partial g_i }{\partial x^j}  \frac{\partial g_i}{\partial x^k}$ and
$\frac{\partial g_i}{\partial x^j}  \frac {\partial g_j}{\partial x^k}$   vanish. Then, the equation reads 
$
\frac {\partial g_i }{\partial x^k} 
\frac {\partial g_k }{\partial x^j} =0
$
and implies that $\frac {\partial g_i }{\partial x^k} =0$ (which in turn implies $C_{ki}=0$) or $\frac {\partial g_k }{\partial x^j} =0$
(which in turn implies $C_{jk}=0$). This proves the first statement of the lemma. Next, observe that under the assumption $C_{jk}=C_{kj}=0$ the equation 
\eqref{eq:third} reads $ \frac {\partial g_i }{\partial x^j} 
\frac {\partial g_i }{\partial x^k} =0$ implying $C_{ki}C_{ji}=0$. Renaming $i\leftrightarrow k$  finishes the proof. 
 \end{proof}

Next, consider  
  $i\in \{1,...,n\}$  such that the $i^{\textrm{th}}$ column of the matrix $C_{ij}$ contains the maximal number of zero entries.  We  assume without loss of generality that $i=1$,  that the elements $C_{21},...,C_{d 1}$ are not zero and the other 
	elements of the first column are zero.   Applying Lemma \ref{Lem:2} to the element $C_{d' 1}$ with $d'>d$, we obtain that 
	$
	C_{d'k} C_{k1}=0. 
	$
	Since $C_{k1}\ne 0$ for $k\le d$, we obtain 	$C_{d'k}=0$
	for such $k.$

  Thus, all elements of the matrix $C_{ij}$ staying under the upper left $d\times d$ block are zero. 
  If $C_{ij}=0$ with $i\ne j \in \{1,...,d\}$,    we obtain a contradiction with  the assumption that 	the $i^{\textrm{th}}$ column of the matrix $C_{ij}$ contains the maximal number of zeros. Thus, all $C_{ij}$
	with $i\ne j \in \{1,...,d\}$	are not zero.  Thus, the first $d$ columns of $C_{ij}$ 
	are as in the upperblockdiagonal matrix with the first block of dimension  $d\times d$. 
	We further have that all nondiagonal components of the  first block are different from zero.

	Next, consider the index $i\in \{d+1,...,n\}$ such that the number of zero entries in the columns of lower 
	right  $(n-d)\times (n-d)$ block is maximal. We may assume without loss of generality 
	that $i={d+1}$, that the components $C_{d+1 \  d+2}, ..., C_{d+1   \  d+ d'}$ are not  zero and the components 
$C_{d+1  \ d+d'+1}, ..., C_{d+1 \    n}$ are zero. Arguing as above, using Lemma \ref{Lem:2},  
we obtain that  for any $k\in \{d+1,...,d+d' \}$ 
the components $C_{d+k \ \ \ d+d'+1}, ..., C_{d+k \  \  n}$ are zero. Thus, the first $d+d'$ 
columns of $C_{ij}$ are as in the upperblockdiagonal matrix with the first block of dimension  $d\times d$ and the second block of dimension $d'\times d'$. Moreover, by the `maximality'  condition in our choice of the first column of the second block,  all nondiagonal elements of the second block are nonzero.

We can repeat the procedure further   and further and obtain that the matrix $C_{ij}$ is as we claimed: it is 
upperblockdiagonal and   in every block  all  nondiagonal entries are different from zero. 
Let us explain now that by the operations (a,b) from  Section \ref{sec:LC}  we can make $C_{ij}$ in every block  equal $1$ for $i<j$  and $-1$ for $j>i$. 
Indeed,  by applying the operations (a) and (b) we can make the first two rows and the first column of every block to be as we claimed.
 The condition \eqref{eq:detijk}  automatically  implies that the  second column of the block is as we want. 
 Next, applying operation (a) we make  the third row as we want. Then,  \eqref{eq:detijk} 
implies that  the third  column is  as we want and so on.
 Note that  these operations with one block do not affect other   blocks.

Next, let us construct a $B\times B$ matrix $(c_{\alpha\beta})$. 
We will denote by $B_{\alpha \beta}$ the blocks  of the matrix $C$ (corresponding to the decomposition $n=n_1+...+n_B$),   the block $B_{\alpha\beta}$ has dimension  $n_\alpha\times n_\beta$. 
Above we achieved  that   if $\alpha\ne \beta$, then either  all   entries of $B_{\alpha\beta}$  are    zero or are   equal to 1.  
In the first case we put $c_{\alpha \beta}=0$, in the second case $c_{\alpha \beta}=1$.
If $\alpha>\beta$ then all entries of the block $B_{\alpha\beta}$ are zero, so such $c_{\alpha\beta}=0$.  
We put $c_{\alpha\alpha}=0$.

Next we  show  that if  such a block $B_{\alpha \beta}$ with $\alpha<\beta$ is zero and the block $B_{\alpha \alpha'}$ is not, then all the blocks $ B_{\alpha' \beta} $  with $\alpha'\ne \beta$ are also zero. In order to do it,   we take an element $C_{ij}$ of this block. By Lemma \ref{Lem:2}, if $c_{\alpha \beta}=0$ with $\alpha\ne \beta$,  then for any $s\not \in \{\alpha,\beta\}$ we have $c_{\alpha s}c_{s \beta}=0$ implying the claim. 
Analogously one shows, using the second statement of Lemma  \ref{Lem:2}, that $c_{\alpha\beta}=0$    with $\alpha< \beta$ implies $c_{\alpha s} c_{\beta s}=0$. 


Let us summarise the properties of the $B\times B$ matrix  $(c_{\alpha\beta})$:
\begin{itemize}

\item [(a)] $c_{\alpha\beta}=0$ for $\alpha\ge \beta$, i.e., the matrix is upper triangular with zeros on the diagonal. 

\item[(b)]  If $c_{\alpha\beta}=0$, then for every $s\in \{1,...,B\}$ we have $c_{s\beta}c_{\alpha s}=0$. 

\item[(c)] If $c_{\alpha\beta}=0$  for certain $\alpha<\beta$, then for every $s\in \{1,...,B\}$ we have    $c_{\beta s}c_{\alpha s}=0$. 

\end{itemize}

Let us show that any such matrix can be constructed from a directed rooted in-forest by a procedure described in Section  4.1.
To see this, we introduce  the relation  $\prec$ on the set $\{1,...,B\}$: we define  $\alpha\prec \beta$ if and only if $ c_{\alpha \beta}=1$.
Clearly, $\alpha\prec \beta$ implies that the number  $\alpha$ is smaller than the number $\beta$. The relation $\prec$  is a strict partial order. Indeed, 
$ \alpha \not \prec \alpha$ because of (a), so the relation is irreflexive.  If $\alpha \prec \beta$, then $\alpha  < \beta $ by (a) implying   $\beta \not\prec \alpha $,
 so the relation is asymmetric.   If $\alpha\prec \beta$ and $\beta  \prec \gamma$ then by (b) we have  $\alpha \prec \gamma$, so the relation is transitive.

Moreover, for every $s \in \{1,...,B\}$ the set $S_s:= \{\alpha \mid \alpha  \prec s\}$ is a {\it chain}, i.e., is totally ordered. 
Indeed,  for  $\alpha <\beta $ $\in S_s$ we have   $c_{\alpha s}=c_{\beta s}=1$ implying $c_{\alpha \beta }=1$ in view of (c).

Next, it is easy to see that every  strict partially ordered    finite set  such that every $S_s$ is a chain can be described by a  directed rooted in-forest. 
The vertices of the forest are the numbers $1,...,B$, and two vertices $\alpha ,\gamma$ 
 are connected   by the oriented edge  $\vec{\gamma \alpha}$ if $\alpha \prec \gamma$ and if there is no $\beta $ such that $\alpha\prec \beta \prec \gamma$, each connected component of this  oriented graph is a  directed rooted in-tree.  
 The forest clearly  reconstructs the order `$\prec$' and therefore the matrix $(c_{\alpha\beta})$: 
for two numbers $\alpha \ne \beta\in \{1,...,B\}$ we have $\alpha\prec \beta$ if there exists an oriented way from $\beta$ to $\alpha$, see example on Fig. \ref{Fig:1}.

The converse is also true: every directed in-forest $\mathsf F$  (with appropriately labelled vertices) defines a matrix $c_{\alpha\beta}$ with properties (a), (b) and (c)  by \eqref{eq:condcij}.


\weg{
\begin{Example}  \label{Ex:differentc} {\rm
If $B=3$, the matrix $c_{\alpha \beta }$ is one of the following: 
$$ 
\begin{pmatrix} 
0 & 0 & 0 \\
0 & 0 & 0 \\
0 &  0 &0
\end{pmatrix}\ , \ \
\begin{pmatrix} 
0 & 1 & 0 \\
0 & 0 & 0 \\
0 &  0 &0
\end{pmatrix}
\ , \ \ \begin{pmatrix} 
0 & 0 & 1 \\
0 & 0 & 0 \\
0 &  0 &0
\end{pmatrix}\ , \ \
\begin{pmatrix} 
0 & 0 & 0 \\
0 & 0 & 1 \\
0 &  0 &0
\end{pmatrix}\ , \ \
\begin{pmatrix} 
0 & 1 & 1 \\
0 & 0 & 0 \\
0 &  0 &0
\end{pmatrix}\  , \ \ \begin{pmatrix} 
0 & 1 & 1 \\
0 & 0 & 1\\
0 &  0 &0
\end{pmatrix}.
$$
}
\weg{The first 4  cases correspond to the direct product situations. For example,  in case 3   the metric $g$ 
is the direct  product of the metrics $g_1 + \frac{1}{\det (L_1) }g_3$ and $g_2$ with  $g_1=Q_1( L_1) g_1^{\LC}$, 
 $g_2=Q_2( L_2) g_2^{\LC}$  and 
 $g_3=Q_3( L_3) g_3^{\LC}$. Notice that the metric  $g_1 + \frac{1}{\det ( L_1) }g_3$ depends on the coordinates $X_1, X_3$ only and the metric  $g_2$ depends on the coordinates $g_2$ only. }
 \end{Example} }

Let us now deal  with $E_{ij}$. Consider $\alpha\ne  \beta\in 1,...,B$. We say that the index $j$ belongs to  the $\alpha$-block (of coordinates), if  $ n_1+...+n_{\alpha-1}< j \le n_1+...+n_{\alpha}.
$ Consider  a pair $(j,i)$ such that $j$ belongs to the $\alpha$-block and $i$ to the $\beta$-block.   If  $c_{\alpha \beta}=0$, then  
	the component 
$E_{ji}$ is irrelevant for our formulas  since  when we build $g$ by \eqref{61} 
 the corresponding term $(C_{ji}x^j + C_{ij}x^i + E_{ji})  $ equals $( C_{ij}x^i + E_{ji})$ and 
    can be hidden in the factor $h_i$ of  $g_{ii}$.  Assume  now     $c_{\alpha\beta}=1$ and let us show that  then the corresponding   numbers $E_{ji}$ do  not depend on $j$ from $\alpha$-block and $i$ from $\beta$-block, i.e., all entries of the  whole $(\alpha,\beta)$-block of the matrix $(E_{ij})$ are equal to each other.

We will use  formula \eqref{eq:third}. If the $\beta$-block has more than one entry, take $k\ne i$ from the $\beta$-block.  Then,  \eqref{eq:third}
$$
\frac{1}{(x^j+ E_{ji})(x^k+E_{ki})}- \frac{1}{(x^j+E_{ji})(x^j-x^k)} + \frac{1}{(x^k+E_{ki})(x^j-x^k)} = 0, 
$$
implying $E_{ji}=E_{ki}$.  Thus, elements  of every  column  of the $(\alpha, \beta)$ block of the matrix  $E_{ij}$ are equal to each other. 

Similarly, by taking $k\ne j$ from the $\alpha$-block 
one proves that the   elements  of every  raw   of the $(\alpha, \beta)$ block of the matrix $E_{ij}$ are equal to each other. Thus, provided 
 $c_{\alpha\beta}=1$, all elements of the $(\alpha, \beta)$ block of $E_{ij}$ are equal to each other, we call them $-e_{\alpha\beta}$. 
The matrix $e=(e_{\alpha\beta})$ is a $B\times B$ constant matrix;   its element $e_{\alpha \beta}$ is  relevant  for us 
(in the sense that it is included in the formulas for the metrics) if the corresponding $c_{\alpha \beta}=1$.

Finally, let us show that $e_{\alpha \beta}= e_{\alpha\gamma}$ provided $c_{\alpha\beta}= c_{\alpha\gamma} = c_{\beta\gamma}=1$. We again use \eqref{eq:third} assuming that $j$ relates to 
the $\alpha$-block, $k$ to the $\beta$-block and $i$ to the $\gamma$-block, to get
$$
\frac{1}{(x^j- e_{\alpha\gamma})(x^k- e_{\beta\gamma})}-0 -\frac{1}{(x^k- e_{\beta\gamma})(x^j- e_{\alpha\beta})}=0
$$
implying $e_{\alpha \gamma}= e_{\beta\gamma}$.

The matrices $c_{\alpha \beta}$  and $e_{\alpha\beta}$ clearly determine the components of the matrices $C_{ij}$ and  
components of the matrix $E_{ij}$ which are relevant  for \eqref{61}. 
Plugging these $C_{ij}$  and $E_{ij}$ into  \eqref{61} and taking in account that
$\prod_{s=1}^{n^\alpha} (e_{\alpha \beta} -x_\alpha^s)=\det(e_{\alpha \beta}\Id_{n_\alpha} -L_\alpha )$,  we obtain     the  form similar to that in  
Theorem \ref{thm:frobenius2}. The difference is as follows:  we  did not prove yet that  the diagonal factors  $H_j(x^j):= \tfrac{1}{h_j(x^j)}$  are polynomials
 $P_\alpha(x^j)$ of degrees $\le n_\alpha+1$ and  we did not obtained the additional conditions on these polynomials  

Let us do  these:  first we notice that the metric $g$ has the iterated warped product structure: 
$$
g= g_1 + \sigma_1(X_1) g_2 + \sigma_2(X_1,X_2)g_3 +...+ \sigma_{B-1}(X_1,...,X_{B-1}) g_B,
$$
where the metric $g_\alpha$ is as follows: take   the metric   $g_\alpha^{\LC}$     given  by   \eqref{eq:tildeg} and multiply its $j$-th 
 diagonal element by $h_j(x^j)$ for every $j$ related to the $\alpha$-block.  The functions $\sigma_\beta$ are as follows: 
 $$
 \sigma_\beta =\left( \prod_{\alpha <\beta}   \det\left(e_{ \alpha \beta}\Id_{n_\beta} - L_\beta\right)^{c_{ \alpha \beta}}\right)^{-1}. 
 $$  
Since $g$ is flat, $g_1$ must be flat and $g_2,...,g_{B}$ of constant curvature.  Applying the result of \cite[\S 5]{Solodovnikov} (see Fact \ref{rem:solodovnikov} above) shows  that blocks $g_\alpha$ of dimension greater than one  are given by  $P_\alpha(L_\alpha)g^{\LC}_\alpha $, where $P_\alpha$ is a polynomial of degree $\le n_\alpha+1$ (in Section \ref{sec:con} we will show that coefficients of these polynomials satisfy the conditions (i--iii) from Section \ref{subsect:4.1} and also consider 1-dimensional blocks).

\subsection{ Conditions on the coefficients of   $P_\alpha$} \label{sec:con} 

To complete the proof of Theorem \ref{thm:frobenius2}, it remains to explain that  in the case of one dimensional blocks the corresponding 
 function  $H_j= 1/h_j$ is a polynomial of degree at most $2$ and that   conditions  (i)-(iii) on the coefficients of $P_\alpha$ 
stated before Theorem \ref{thm:frobenius2a} are fulfilled.  
We will need some facts and preliminary work. 

Recall that a \emph{Casimir}  of a Poisson structure  
is defined by the property that Poisson structure applied to it gives  zero. For a Poisson structure $\mathcal A_g$ given by \eqref{eq:intro:6} 
and corresponding to a (flat) metric $g$ of dimension $\ge 2$, a Casimir (of the lowest order) can be understood as 
 a function $f$ satisfying $\nabla_i \nabla_j f=0$. Of course, any constant is a Casimir and $n$ functionally independent Casimirs give us  flat coordinates for $g$ in which the components $g^{ij}$ are all constants. For  a metric $g$ of constant curvature $K$ we define \emph{Casimir}   as a function 
$\widehat f$ satisfying the equation\footnote{The functions  satisfying this equation are  indeed  Casimirs  of the (nonlocal) Poisson structure corresponding to the constant curvature metric $g$, see e.g. \cite[\S 2]{NijenhuisAppl2} and references therein.} $\nabla_i \nabla_j \widehat f + K \widehat f g_{ij} =0$.   The space of Casimirs of a constant curvature metric  (on a simply-connected manifold of dimension $\ge 2$) is a vector space of dimension $n+1$.

Dimension 1 is special. In this case, we will define a Casimir as a function satisfying $\nabla_i \nabla_j \widehat f + K \widehat f g_{ij} =0$ 
for some constant $K$, so that for each $K$, the Casimirs form a two-dimensional space.

\begin{Fact}   \label{Fact:cas} 
Consider the $n\ge 2$-dimensional metric $g=  P(L)  g_{\LC}$, where $L$ and $g_{\LC}$ are as in \eqref{eq:tildeg}, and $P$ is a polynomial 
$
P(t) = a_0+ a_1 t + ... + a_{n+1} t^{n+1}
$.  Then, the metric has constant curvature $-\tfrac{1}{4} a_{n+1}$.  Moreover, in any dimension $n\ge 1$,   
 the function $\sqrt{\det(\lambda \Id - L)}$ is a Casimir of $g$ if and only if $P(\lambda)=0$ and in this case we  have: 
\begin{equation}
\label{eq:length} 
4 g\left(\ddd\sqrt{ (\lambda\Id - L)} , \ddd\sqrt{\det  (\lambda \Id - L)}\right)=  - \frac{ \ddd P(t)}{\ddd t}_{\mid t=\lambda}  +   a_{n+1} \det  (\lambda \Id - L). 
\end{equation}
\end{Fact}

\begin{proof}
The metric $g_{\LC}$ and operator $L$ are explicitly  given (and w.l.o.g. one may assume $\lambda=0$) so the proof is an exercise in the Vandermonde identities and is left to the reader.  
\end{proof}

Below, we work with warped product metrics for which the `covariant language' is more convenient. For this reason, starting from Fact \ref{fact:6,5}  and till the end of the current Section \ref{sec:con},    $g$ and $g_i$ will denote covariant metrics. For the corresponding contravariant metrics  we use $g^*$ and $g_i^*$. 
                                   
\begin{Fact}\label{fact:6,5}
Suppose a warped product metric  $g=g_1 + f(X_1)^2  g_2$  has constant curvature. Then $g_1$   and $g_2$ have constant curvatures. Moreover,  the following statements hold:
\begin{enumerate}
\item     If $g$  is flat, then $g_1$ is flat.    
 \item  If $g$ is flat and $g_2$ is $n_2\ge2$-dimensional, then $K_2= g^*_1(\ddd f, \ddd f)$,   where  $K_2$ is the curvature  of $g_2$. 
 \item $f(X_1)$ is $g_1$-Casimir. \end{enumerate}
\end{Fact} 

\begin{proof}
The first  statement is  well-known and immediately follows from geometric arguments. The second statement  follows  from 
the  second formula in the first line of   \cite[(4.2)]{Prvanovic}. The third  statement follows, under the additional assumption  that the curvature is zero, 
 from the second line   of \cite[(4.2)]{Prvanovic}.
If the curvature is not zero, we may assume that it is equal to $1$.
 Then, we employ the conification construction: we consider the $(n+1)$-dimensional metric $\widehat g= (dx^0)^2+ (x^0)^2\bigl( g_1 + f(X_1)^2  g_2\bigr)$.    The metric $\widehat g$ is flat  and can be viewed as a warped product metric with base  $(dx^0)^2+ (x^0)^2 g_1$ and warping function $\widehat f^2:=  \left(x^0 f(X_1)\right)^2$. Then,  the function $\widehat f:= x^0 f(X_1)$ is a Casimir of $\widehat g_1:= (dx^0)^2+ (x^0)^2  g_1$ implying that $f(X_1)$ is a Casimir of $g_1$. 
\end{proof}

Next, we need the following technical lemma: 
\begin{Lemma}\label{lem:Casimir1} 
Suppose the warped product metric  $g_1 + f(X_1)^2  g_2$   is flat and $f(X_1)$ is $g_1$-Casimir. Then, the following holds:
\begin{enumerate} 

\item  Every $g_1$-Casimir  $F(X_1)$ such that $g_1^*(\ddd f,\ddd F)=0$ is a Casimir of $g$.

\item  If $g_2$ has constant nonzero curvature $K_2$ or is one-dimensional, then for any function $\phi(X_2)$ satisfying $\nabla^{g_2} \nabla^{g_2}\phi + K_2 \phi g_2=0 $ the function $f \phi$ is a $g$-Casimir. 
\item If $g_2$ is flat, then for any function  $\phi(X_2)$ satisfying $\nabla^{g_2} \nabla^{g_2} \phi(X_2)= \const g_2 $ 
and for any $g_1$-Casimir $\tilde f$ such that $g_1^*(\ddd f, \ddd \tilde f)=1$ we have 
that 
$f \phi - \const  \tilde f $ is a $g$-Casimir. 

\end{enumerate} 
\end{Lemma}

Notice that the $g$-Casimirs described in Lemma \ref{lem:Casimir1}  span a $n+1$-dimensional vector space and, therefore, form a basis of 
the space of $g$-Casimirs. 

\begin{proof} 
By direct calculations  (done many times in the literature, see e.g. \cite[(4.1)]{Prvanovic})
  one sees that  the Christoffel symbols of the warped product metric $g$ are  given by the following formulas: 
	$$
	\Gamma^a_{bc}=\overset{1}{\Gamma}{}^a_{bc} \ ,  \ \Gamma^\alpha_{\beta \gamma}= \overset{2}{\Gamma}{}^\alpha_{\beta \gamma} \ ,
	\Gamma^a_{\beta \gamma} = - f 
	\sum_s f_{,s} \overset{1}{g}{}^{sa} \overset{2}{g}_{\beta \gamma} \ , \Gamma^\alpha_{a\beta}= \frac{1}{f} f_{,a} \delta^\alpha_\beta \  , \Gamma^a_{\alpha b}= \Gamma^a_{b\alpha } =\Gamma^\alpha_{ab}=0.      
	$$
Here    $\overset{1}{\Gamma}$, $\overset{2}{\Gamma}$  relate to the Christoffel symbols of the metrics $g_1= \overset{1}{g}$ and  $g_2= \overset{2}{g}$ respectively;  $a,b,c,s $ run from $1$ to $n_1$ and $\alpha,\beta,\gamma$ from $n_1+1 $
	to $n$.    The notation $f_{,s}$  means $\tfrac{\partial f }{\partial  x^s}$.

	Therefore, for any function $F(X_1)$ we have:
	\begin{equation}\label{eq:con1} 
	\nabla_a \nabla_b F = \nabla^{g_1}_a \nabla^{g_2}_b F  \ , \quad \nabla_\alpha \nabla_a F=0 \ , \quad 	\nabla_\alpha \nabla_\beta  F=  f g_1^*(dF, df) \overset{2}{g}_{\alpha \beta}.
	\end{equation}
	In particular, if $F$ is a $g_1$-Casimir, then $\nabla^{g_1}_i \nabla^{g_1}_j F= f g_1^*(dF, df) \overset{2}{g}_{ij}$ which implies
	the first statement.  Also, if $g_2$ is flat, then 
	$\textrm{grad}_{g_1}       f$ is light-like (Fact \ref{fact:6,5}, item 2), so  $f$ is   a $g$-Casimir.

	Next,  for any function  $\phi(X_2)$  we have 
	$$
	\nabla_a\nabla_b \phi=0 \ , \quad \nabla_a \nabla_\beta  \phi= -\frac{1}{f} f_{,a} \phi_{,\beta} \ , \quad \nabla_\alpha \nabla_\beta  \phi= \nabla^{g_2}_\alpha \nabla^{g_2}_\beta  \phi.$$

	In particular, if $\phi$  satisfies  $\nabla^{g_2}_{\alpha}\nabla^{g_2}_\beta \phi = \const \overset{2}{g}_{\alpha\beta}$ and $g_2$ is flat, then
	$$
	\nabla_{a} \nabla_b (\phi f)= 0 \ ,  \quad 
	\nabla_{\alpha } \nabla_b  (\phi f)=  0 , \quad     \nabla_\alpha \nabla_\beta (\phi f) = \const \overset{2}{g}_{\alpha\beta} f.
	$$	
	Combining this with \eqref{eq:con1},  we see that   $\phi f - \const \tilde f$ is a Casimir. 
If  $\phi$ satisfies $\nabla^{g_2} \nabla^{g_2}\phi + K_2 \, \phi \, g_2=0 $, 
then $\nabla_\alpha \nabla_\beta (\phi f) = -K_2\, \phi \,  \overset{2}{g}_{\alpha\beta}\, f -\phi \, f \, g_1^*(\ddd f, \ddd f)\, \overset{2}{g}_{\alpha \beta}=0$.  
	\end{proof}

Now, we are able to describe conditions on the polynomials $P_\alpha$  that are necessary for the flatness of the metric $g$  from Theorem \ref{thm:frobenius2} given by  \eqref{eq:thm4}, and also to finish the case of one-dimensional blocks: We need to show that if  an $\alpha$-block is one-dimensional, 
then  the corresponding function $H_j(x^j)= \tfrac{1}{h_j(x^j)}$   from \eqref{eq:solodovnikov}  is a polynomial of degree $\le 2$. 
 
First we consider the most important case, when  $c_{\alpha \beta}=1$ for all  $1\le\alpha<\beta\le B$. The corresponding graph in this case is just a `path' $1\ {\longleftarrow} \ 2 \ {\longleftarrow}\cdots {\longleftarrow} \ B$ from leaf $B$ to root $1$, so that  in view of \eqref{61} 
the metric $g_{ij}$ is  given by the warped product of the form 
\begin{equation} \label{eq:ttt}
g = g_1 + f_1(X_1)^2  g_2 +f_1(X_1)^2f_2(X_2)^2g_3+...+ \left(\prod_{s=1}^{B-1} f_s(X_s)^2 \right) g_B
\end{equation}
with $f_\alpha(X_\alpha)^2= \det ( \lambda_\alpha \Id_{n_\alpha}- L_\alpha)$ for   $\lambda_\alpha= e_{\alpha\ \alpha+1}= e_{\alpha\ \alpha +2}=...=e_{\alpha\ B}$.  The metrics $g_\alpha$ are $n_\alpha$-dimensional  metrics of the form  \eqref{eq:solodovnikov}.

Suppose  that a $g_\alpha$-component is one-dimensional  and denote the corresponding coordinate by $x^j$. 
 By direct calculations,  using
 the formulas for $S^{ij}_k$    and $\widehat \Gamma^i_{jk}     $
  from Section \ref{subsect:6.2} and assuming $u_j=0$,  we see that 
the condition $\widehat \nabla_j S^{jj}_j=0$ becomes 
\begin{equation}\label{eq:1dim}
\frac{\partial^2 h_j(x^j)}{\partial {x^j}^2}  +  
8 \, g^*\left(\ddd \prod_{s=1}^{\alpha-1} f_s(X_s), \ddd \prod_{s=1}^{\alpha-1} f_s(X_s)\right)=0. 
\end{equation}
Note that the function $\prod_{s=1}^{\alpha-1} f_s(X_s)$ is a Casimir of the  metric 
$g_1 + f_1(X_1)^2  g_2 +...+ \left(\prod_{s=1}^{\alpha-1} f_s(X_s)^2 \right) g_\alpha$ by Fact \ref{fact:6,5} so 
$g^*\left(\ddd\prod_{s=1}^{\alpha-1} f_s(X_s), \ddd\prod_{s=1}^{\alpha-1} f_s(X_s)\right)$ is a constant and $h_j$ is a polynomial of degree $\le 2$ whose leading coefficient is this constant with a minus sign. 
Note that the  case $\alpha=1$ is also covered by the argumentation above by assuming that this constant is zero; in this case $h_j$ is a linear polynomial.

Thus, our metric $g$ is given by \eqref{eq:ttt} 
with each $g_\alpha = g_\alpha^{\LC}\bigl(P_\alpha( L_\alpha)\bigr)^{-1}$ and  
$f_\alpha^2=  {\det  (\lambda_\alpha\Id_{n_\alpha}- L_\alpha ) }$,  where 
 $P_\alpha$ has degree  at most $n_\alpha+1$. We denote  the coefficients of the polynomials 
by 
$$
P_\alpha(t)=  \overset{\alpha}{a}_0 + \overset{\alpha}{a}_1t +...  +\overset{\alpha}{a}_{n_\alpha+1}t^{n_1+1}.
$$

\begin{Lemma} 
\label{lem:flat}
In the above notation, we have  the following relations:
\begin{equation} \label{eq:conditions} \begin{array}{lr} \overset{1}{a}_{n_1+1}=0    \, &  \\ 
P_\alpha(\lambda_\alpha)=0 & \textrm{\ for $\alpha=1,...,B-1$.}  \\
 \frac{\ddd P_\alpha     }{\ddd t}_{|t=\lambda_\alpha}    =  \overset{\alpha+1}{a}_{n_{\alpha+1}+1}    & \textrm{\ for $\alpha=1,...,B-1$ } 
\end{array}\end{equation}
\end{Lemma}

 Note  that in view of Theorem  \ref{thm:frobenius2a},  conditions \eqref{eq:conditions} are sufficient for the flatness of  $g$ and existence of Frobenius coordinates. 

\begin{proof} 
We view $g$ as a warped product metric over the $n_1+ n_2$ dimensional base equipped with the metric   $g_1 + f_1(X_1)^2  g_2 $. 
Then,   the metric $g_1 + f_1(X_1)^2   g_2 $ is flat. If $n_1>1$ (the case $n_1=1$ was already discussed above), 
combining Facts \ref{Fact:cas} and  \ref{fact:6,5}   we obtain $a_{n_1+1}=0$ implying the first line of \eqref{eq:conditions}.   Next, from Fact \ref{fact:6,5}  we know that $ f_1(X_1)$ is a $g_1$-Casimir implying  $P_1(\lambda_1)=0$  in view of Fact \ref{Fact:cas}. 

Let us now  show that
\begin{equation}
\label{eq:relation}  
\frac{\ddd P_1     }{\ddd t}_{|t=\lambda_1} =   \overset{2}{a}_{n_2+1}.
\end{equation}  
 Since   $g_1 + f_1(X_1)^2   g_2 $ is flat, 
 by  Fact \ref{fact:6,5}  we have  (we denote  by $K_2$ the curvature of $g_2$ and assume $n_2>1$) 
   \begin{equation} 
   \label{eq:combbist} 
   \tfrac{1}{4} g^*\bigl( \ddd f_1, \ddd f_1 \bigr) =  K_2.
   \end{equation}
Combining \eqref{eq:combbist}  with Fact \ref{Fact:cas}, we obtain \eqref{eq:relation}. In the case $n_2=1$,  \eqref{eq:relation} follows from   \eqref{eq:1dim}  and  Fact \ref{Fact:cas}.   
 
Next,  we  view $g$ as a warped product metric over the  $n_1+ n_2 + n_3$ dimensional  base equipped with the metric 
 $g_1 + f_1(X_1)^2  g_2 +f_1(X_1)^2  f_2(X_2)^2 g_3$. 
Then, the metric $g_1 + f_1(X_1) ^2 g_2 +f_1(X_1)^2 f_2(X_2)^2 g_3$ is flat. But it is itself a warped product metric over the $n_1+n_2$ dimensional base with  the metric  $g_1 + f_1(X_1) ^2 g_2$. 
By Fact  \ref{fact:6,5}, this  
implies  
that $ {f_1f_2}$ is a Casimir of this metric.  By  Lemma \ref{lem:Casimir1},  $f_2$ is a Casimir of $g_2$ so $P_2(\lambda_2)=0$. 
Moreover, by  Fact \ref{fact:6,5}  and using $g^*\bigl( \ddd f_1, \ddd f_2\bigr)=0$ we have  
   \begin{equation} 
   \label{eq:comb} 
   \tfrac{1}{4} g^*\bigl( \ddd (f_1f_2), \ddd (f_1 f_2)\bigr) =  K_3   \ \ (= \ -\tfrac{1}{4} \overset{3}{a}_{n_3+1} ).
   \end{equation}
On the other hand 
\begin{equation}
\label{eq:comb2}  
\begin{array}{ll}
 g^*\bigl( \ddd ({f_1f_2}), \ddd (f_1 f_2)\bigr) &= {f_2}^2  g_1^*( \ddd f_1, \ddd f_1) +     g_2^*( \ddd f_2, \ddd f_2)  \\  
 &  \stackrel{ \eqref{eq:length} }{ =  }   -\tfrac{1}{4}{f_2}^2 \frac{\ddd P_1     }{\ddd t}_{|t=\lambda_1}  -   
\tfrac{1}{4} \frac{\ddd P_2     }{\ddd t}_{|t=\lambda_2}  +    \overset{2}{a}_{n_2 +1} {f_2}^2 \\ & \stackrel{ \eqref{eq:relation} }{ =  } -\tfrac{1}{4} \frac{\ddd P_2     }{\ddd t}_{|t=\lambda_2}   .
\end{array}
\end{equation} 
Combining this with \eqref{eq:comb}      we obtain   $\overset{3}{a}_{n_3+1}=  \frac{\ddd P_2     }{\ddd t}_{|t=\lambda_2} $, as claimed. Iterating this procedure we   obtain     \eqref{eq:conditions}.   \end{proof} 

Lemma \ref{lem:flat} completes the proof of Theorem \ref{thm:frobenius2} under the additional assumption that for every $\alpha<\beta$ we have  $c_{\alpha\beta}=1$.  
  We now reduce the general case to this situation.

 We assume  without  loss of generality that the  combinatorial data are given by a directed rooted in-tree with $B$ vertices, i.e. the graph $\mathsf F$ is connected.  Otherwise, 
we have the direct product situation, i.e.,  the metric and all other relevant objects are direct products  of lower-dimensional metrics and relevant lower dimensional  objects.

We denote by $1,2,...,B$ the vertices of the in-tree $\mathsf F$ in such a way that $\alpha\prec \beta$ implies $\alpha <\beta$;
 of course, the vertex    $1$ is then the  root. Other vertices of degree one are called  \emph{leaves}. 
Recall that $\alpha= \operatorname{next}(\beta)$, if $\alpha\prec \beta$ and there is no $\gamma$ with $\alpha\prec  \gamma \prec \beta$. 

For every leaf $\beta$ we define the {\it chain} $S_\beta$ (oriented path to the root) as the sub-tree with vertices   $ \beta$, $\operatorname{next}(\beta)$, $   \operatorname{next}\left(\operatorname{next}(\beta)\right)$,...,$ 1$. For example, 
the  upper tree of Fig. \ref{Fig:1} has two chains, one with vertices $3,2,1$ and another with vertices $4,2,1$.
		
	Next, for the chain $S_\beta$ and for any fixed point $p$
	we consider the following submanifold $M_\beta$ passing through $p$: in the coordinates $(X_1,...,X_B)=(x^1,...,x^n)$ it is defined by the system of 
	equations 	$$X_\alpha=X_\alpha(p) \ \textrm{ for every $\alpha\not \in S_\beta$ }.$$    
This is a totally geodesic submanifold with respect to the connections $\Gamma$, $\bar \Gamma$ and  $\widehat \Gamma$. For $\widehat \Gamma$ this follows from formulas (A,B,C) of Section \ref{subsect:6.2}. For $\Gamma$ and $\bar \Gamma$ it follows from  \eqref{subg}.

 Therefore,  the restriction of $g$ and $\bar g$ onto $M_\beta$ satisfies the assumptions of Theorem \ref{thm:frobenius2}. Moreover, the components  $c_{\alpha \beta}$ corresponding to this restriction are equal  to $1$ for $\alpha<\beta$.

For example for  $\beta= 3$,  the metric $g$  corresponding  to  the  upper tree of Fig. \ref{Fig:1}   
 is given by 
$$
g= g_1 + \det ( e_{12} \Id_{n_1}-L_1) \cdot g_2 +  \det  (e_{12} \Id_{n_2}-L_1 ) \cdot ( ( e_{23} \Id_{n_3}- \det L_2)\cdot  g_3 + \det(e_{24} \Id_{n_4}-L_2  ) \cdot g_4)  
$$ 
and the restriction of the metric $g$ onto $M_3$ is   
$$
g_1 + \det (e_{12} \Id_{n_1}-L_1 )  \cdot g_2 +  \det (e_{12} \Id_{n_2}-L_1 )  ( e_{23} \Id_{n_3}-\det L_2 ) \cdot g_3.
$$

The case when   $c_{\alpha \beta}=1$ for all   $\alpha<\beta$ is described in Lemma \ref{lem:flat}  and it has been proved that the metric is as  in Theorem \ref{thm:frobenius2}. This implies that the metrics $g$ and $\bar g$ are constructed as in Section \ref{subsect:4.1} and the coefficients of $P_\alpha$ and 
$Q_\alpha$ satisfy conditions (i, ii) from Section   \ref{subsect:4.1}. It remains to show that they also satisfy condition (iii). In order to do this, suppose that $\alpha=  \textrm{next}(\beta)= \textrm{next}(\gamma)$ with $\beta\ne \gamma$. 
We consider the  sub-tree with vertices  $\alpha$, $ \beta$,  $\gamma:= \textrm{next}(\alpha)= \textrm{next}(\beta)$ 
and the corresponding  warped product  metric  with the base metric $g_\alpha$ and fibre 
  metric $g_\beta + g_\gamma$. For example,  for $\alpha=2$   in the case of  the  upper tree of Fig. \ref{Fig:1}, we consider 
		the  warped product metric $g_2 +   \det (\lambda_\alpha \Id_{n_2} - L_2) \cdot (g_3 + g_4)$. 
		
		We know that it must be of constant curvature which implies that the 
		direct product metric  $g_\beta + g_\gamma$ must  be of constant curvature which  in turn implies that it is flat. Then, by Fact  \ref{Fact:cas} 
		the   coefficients    
		$ \overset{\beta}{a}_{n_\beta+1}$ and   $\overset{\gamma}{a}_{n_\gamma+1}$  vanish implying   $\frac{\ddd P_{\alpha}}{\ddd t}_{| t=   \lambda_\alpha}=0$. 
Theorem  \ref{thm:frobenius2}  is proved.

\subsection{On  the uniqueness of Frobenius coordinates for a pair of metrics}\label{sect:8a}

We consider two flat metrics $g, \bar g$ possessing a common Frobenius coordinate system, and discuss the uniqueness of this coordinate system.  As before we assume that $R_g=\bar g g^{-1}$ has $n$ different eigenvalues. We know that $g, \bar g$ are as described in  Theorem \ref{thm:frobenius2}. In particular, in  the  corresponding coordinates, the connection  $\widehat \Gamma$ defining the Frobenius  coordinate system (i.e., the flat connection that vanishes in Frobenius coordinates) is 
 given by the formulas from Section  \ref{subsect:6.2}.  That is, 
\begin{itemize}

\item For every $i\ne j$ and $ k\ne i$,  $\widehat \Gamma^i_{jk}=0$. 

\item For the indices $i\ne j $ from one block, $\widehat \Gamma^i_{ij}=\frac{1}{x^j -x^i}$. 

\item For the index $i$ from the block number $\alpha$ and $j$ from the block number $\beta\ne \alpha$, we have  
 $\widehat \Gamma^j_{ji}=\widehat \Gamma^j_{ij}=\frac{c_{\alpha \beta}}{x^i}$.

\item For every $i$, the component  $\widehat\Gamma^i_{ii}:= u_i$ depends on $x^i$ only. 
\end{itemize}

However,  in view of \eqref{eq:sub}  and \eqref{eq:pdeab}, the function $u_i$ is also uniquely defined in terms of the metric $g$ (in fact, $u_i\equiv 0$ in the diagonal coordinates from Theorem \ref{thm:frobenius2}) unless  the only  component of $g$  which can  depend on $x^i$ is the component $g_{ii}$.  Such an exceptional situation is possible if and only if $g_{ii}$  represents a one-dimensional block which is a leaf in terms of the in-forest $\mathsf F$. 

Clearly,  the Frobenius coordinate system is determined up to affine change of coordinates by the flat connection $\widehat \Gamma$. Therefore, 
non-uniqueness of this coordinate  system only appears in the above exceptional situation.
It is easy to see that in this case the freedom in choosing it is the same as discussed in Section \ref{subsect:4.1} after Theorem \ref{thm:frobenius2}.

\begin{Remark} \label{last:step:proof} {\rm  Coming back to Theorem \ref{thm:frobenius1} (uniqueness part), we notice that   for the metrics $g=P(L) g_{\LC}$ and $\bar g=Q(L) g_{\LC}$,  every diagonal component  $g_{ii}$ depends on all variables $x^1,...,x^n$. Then, the connection  $\widehat \Gamma$ is unique implying 
 that Frobenius coordinates are unique  up to an affine coordinate change, as required. 
}\end{Remark}

\section{Pro-Frobenius coordinates and  multi-block Frobenius pencils. Proof of Theorem~\ref{thm:frobenius2a}}\label{sect:8}

\subsection{Extended AFF-pencils and pro-Frobenius coordinates} \label{subsect:8.1}

As seen from Section \ref{subsect:4.1},  the main ingredients in multi-block Frobenius pencils (Theorem \ref{thm:frobenius2})  are metrics of the form 
\begin{equation}
\label{eq:AFFextended}
P(L)g_{\LC}, \quad \mbox{where $P(\cdot)$ is a polynomial of degree $n+1$.}
\end{equation}   
If $\deg P\le n$,  then such metrics are flat  and form the AFF-pencil \eqref{eq:AFF_LC}.  However, if $\deg P = n+1$, i.e., $P(t) = a_{n+1}t^{n+1} + \dots$,  then $g=P(L)g_{\LC}$ has constant curvature  $K=-\frac{1}{4}a_{n+1}$ (see Fact \ref{fact:6,5}).  All together, the metrics \eqref{eq:AFFextended} form a pencil of compatible constant curvature metrics,  which can be thought of as one-dimensional extension of the AFF-pencil \eqref{eq:AFF_LC}, see details in \cite{Ferapontov-Pavlov}, \cite{NijenhuisAppl2}.  In Frobenuis coordinates $u^1, \dots, u^n$ from  Section \ref{subsect:AFF},   the coefficients $g^{\alpha\beta}$ of the metric $g=P(L)g_{\LC}$ with $\deg P = n+1$  are not affine functions anymore. In particular, in the notation from Section \ref{subsect:AFF}, for $P(t)=t^{n+1}$ we get:   
$$
g_{n+1} = L^{n+1} g_0 = 
\begin{pmatrix}
u^2 & u^3 & \dots & u^n & 0 \\
u^3 & \dots &u^n & 0 & 0 \\
\vdots & \iddots & \iddots & \vdots & \vdots \\
u^n & 0 & \dots & 0 & 0 \\
0& 0  & \dots & 0 & 0 \\
\end{pmatrix} +   \begin{pmatrix}
u^1 u^1  &   u^1 u^2 &  \dots &  u^1 u^n\\
u^2 u^1 & u^2u^2 &  \dots & u^2 u^n \\
u^3 u^1 & u^3u^2  &  \ddots &\vdots \\
\vdots & \vdots &  \ddots & \vdots \\
u^nu^1 & u^n u^2&  \dots & u^n u^n \\
\end{pmatrix}.
$$
All the other metrics from the extended AFF pencil \eqref{eq:AFFextended}, in coordinates $u^1,\dots, u^n$, take the form 
\begin{equation}
\label{eq:proFrobmetric}
g^{\alpha\beta} = b^{\alpha\beta} + a^{\alpha\beta}_s u^s - 4K u^\alpha u^\beta,
\end{equation}
which looks as a {\it quadratic perturbation} of \eqref{eq:Frobcoord}.  These metrics (more precisely,  the corresponding coordinates $u^1,\dots, u^n$) still possess remarkable properties, similar to those from Fact \ref{fact:4}.  We use one of them to introduce the following 

\begin{Definition}{\rm We say that $u^1,\dots, u^n$  is a {\it pro-Frobenius coordinate system} for a metric $g$ given by \eqref{eq:proFrobmetric}, if
the contravariant Christoffel symbols $\Gamma^{\alpha \beta}_s$ of $g$ are symmetric in upper indices. 
}\end{Definition}

The Frobenuis coordinates $u^1, \dots, u^n$ for the metrics from the AFF-pencil are
pro-Frobenius for all the metrics \eqref{eq:AFFextended} from the extended AFF-pencil  (recall that these coordinates are the coefficients of the characteristic polynomial of the Nijenhuis operator $L$).  

Notice that the symmetry of $\Gamma^{\alpha\beta}_s$  in upper indices is a strong geometric condition that naturally appears in many problems and admits several equivalent interpretations. The following statement summarises some of them  (the proof is straightforward and is left to the reader).   

\begin{Lemma}\label{lem:8.1}
Let $g_{ij}$ be a metric written in a certain coordinate system $u^1,\dots, u^n$.   The  following conditions are equivalent
\begin{itemize}
\item  The contravariant Christoffel symbols $\Gamma^{ij}_k=g^{is}\Gamma^j_{sk}$ are symmetric in upper indices,
\item   $\Gamma^{ij}_k = -\frac{1}{2} \frac{\partial g^{ij}}{\partial u^k}$,
\item   The Christoffel symbols  (of the first kind) $\Gamma_{ijk} = g_{ks}\Gamma^s_{ij}$ are totally symmetric in lower indices,
\item   $\Gamma_{ijk} = \frac{1}{2}\frac{\partial g_{ij}}{\partial u^k}$,
\item   $g_{ij}$ is a Hessian metric, i.e.,  $g_{ij} = \frac{\partial^2 F}{\partial u^i\partial u^j}$ for a certain function $F(u^1,\dots, u^n)$\footnote{Following \cite{Shima} we will refer to this function $F$ as the {\it potential}.}.
\end{itemize} 
\end{Lemma}

This lemma immediately implies the following explicit formula for the contravariant Christoffel symbols of $g$ in pro-Frobenius coordinates:
        \begin{equation}\label{cw3}
        \Gamma^{\alpha \beta}_s = - \frac{1}{2}  \frac{\partial g^{\alpha\beta}}{\partial u^s} = - \frac{1}{2}    a^{\alpha \beta}_s + 2K u^{\alpha} \delta^{\beta}_s + 2K \delta^{\alpha}_s u^{\beta}.   
        \end{equation}

\begin{Remark}{\rm
Notice that for $K = 0$,   a pro-Frobenius coordinate system is Frobenius  (see Fact \ref{fact:4}).   Also it can be checked that a pro-Frobenius metric \eqref{eq:proFrobmetric}  has constant curvature $K$.   For this reason,  pro-Frobenius metrics can naturally be understood as generalisation of Frobenius metrics \eqref{eq:Frobcoord} to the case of constant curvature metrics.
}\end{Remark}

\subsection{Warp product of pro-Frobenius metrics}

The proof of Theorem \ref{thm:frobenius2a} is based on the following key statement.

\begin{Proposition}\label{prop:keyprop}
Let $g$ and $\widehat g$ be two constant curvature metrics written in pro-Frobenius coordinates $u=(u^1,\dots, u^{n_1})$ and $v=(v^1,\dots , v^{n_2})$ respectively.  
Let $f(u)$ be a non-homogeneous linear function satisfying the following relations
\begin{itemize}
\item[\rm(b1)]
$\operatorname{grad} f(u) = f(u) (\alpha - 4K u)$, where $\alpha = \sum \alpha^i\frac{\partial}{\partial u^i}$ is a constant vector field and $u=  \sum u^i\frac{\partial}{\partial u^i}$,

\item[\rm(b2)]
$\frac{1}{f(u)} g(\ddd f, \ddd f) =  4\widehat K - 4K f(u)$, where $K$ and $\widehat K$ are the curvatures of $g$ and $\widehat g$ respectively.
\end{itemize}

Then the warp product metric\footnote{This metric was also discussed in Section \ref{sect:7}, see Fact \ref{fact:6,5} et seq. Note the difference in notation and settings used.  Unlike  
Fact \ref{fact:6,5}, now we are working with contravariant metrics,   and our {\it new} function $f$ was previously denoted by $f^2$.}
$g_{\mathsf w} = g + \frac{1}{f} \, \widehat g$ has constant curvature and the coordinate system $u^1,\dots, u^{n_1},   y^1= f v^1, \dots , y^{n_2} = f v^{n_2}$  is pro-Frobenius for $g_{\mathsf w}$.
\end{Proposition}

\begin{proof}
The matrix of  $g_{\mathsf w}$ in coordinates $u, y$ takes the form
$$
g_{\mathsf w} = \begin{pmatrix}  \Id & 0 \\ v\cdot \ddd f & f \cdot \Id \end{pmatrix}
\begin{pmatrix}  g & 0 \\ 0 & \frac{1}{f} \widehat g  \end{pmatrix}
\begin{pmatrix} \Id &  \ddd f^\top \cdot v^\top \\ 0 & f \cdot \Id \end{pmatrix} =
\begin{pmatrix}
g &  \operatorname{grad} f \cdot v^\top \\
v \cdot (\operatorname{grad} f)^\top &  g(\ddd f,\ddd f) v \cdot v^\top {+} f \, \widehat g
\end{pmatrix}
$$
Hence the off-diagonal block has the following form
$$
 \operatorname{grad} f \cdot v^\top  =  f(u) (\alpha - 4K u) \cdot v^\top = (\alpha - 4K u) \cdot y^\top.
$$ 
The lower diagonal block is 
$$
g(\ddd f,\ddd f) v \cdot v^\top + f\, \widehat g = (4\widehat Kf - 4K f^2 ) v \cdot v^\top + f \, \widehat g   
$$
Recall that  $\widehat g$ consists of three parts, constant, linear and quadratic
$$
\widehat g = \widehat b + \widehat a(v) - 4\widehat K v\cdot v^\top \quad \mbox{or, in more detail,} \quad 
\widehat g^{ij} = \widehat b^{ij} + \widehat a^{ij}_s v^s -  4\widehat K v^i v^j.
$$
Substitution gives:
$$
\begin{aligned}
(&4\widehat Kf - 4K f^2 ) v \cdot v^\top + f \, \widehat g =
(4\widehat Kf - 4K f^2 ) v \cdot v^\top + f (\widehat b + \widehat a(v) - 4\widehat K v\cdot v^\top ) = \\
& f \widehat b + \widehat a(y) - 4K y \cdot y^\top
\end{aligned}
$$
Summarising  (and denoting  $g = b + a(u) - 4K u\cdot u^\top$)  we get 
$$
g_{\mathsf w} = \begin{pmatrix} b & 0 \\ 0 & f(0) \widehat b  \end{pmatrix}
+ \begin{pmatrix}  a(u)   &     \alpha \cdot y^\top \\ y \cdot \alpha^\top &  \widehat a(y) + m(u) \widehat b  \end{pmatrix}  - 4K \begin{pmatrix}  u \\ y \end{pmatrix}
\begin{pmatrix} u^\top \  y^\top \end{pmatrix},
$$ 
where $m(u) = f(u) - f(0)$ is the (homogeneous) linear part of the function $f$.  
This shows that $g_{\mathsf w}$ takes form \eqref{eq:proFrobmetric}  in coordinates $(u,y)$. 

To complete the proof we only need to show that the (covariant) metric  $(g_{\mathsf w})_{ij}$ is Hessian (see Lemma \ref{lem:8.1}).  This fact follows from    

\begin{Lemma}
Suppose $g_{ij}(u)=\frac{\partial^2 F}{\partial u^i\partial u^j}$ is a Hessian metric in coordinates $u^1,\dots, u^{n_1}$ and $\widehat g_{\alpha\beta}(v)=\frac{\partial^2 \widehat F}{\partial v^\alpha\partial v^\beta}$ is a Hessian metric in coordinates $v^1,\dots, v^{n_2}$.   Then the (covariant) warp product metric 
\begin{equation}
\label{eq:wprod1}
(g_{\mathsf w})_{ij} =  g_{ij}(u)\ddd u^i \ddd u^j + f(u) \widehat g_{\alpha\beta}(v) \ddd v^\alpha \ddd v^\beta
\end{equation}
where $f(u)$ is an arbitrary linear function in $u$,  is Hessian in coordinates $u^1,\dots, u^{n_1}, y^1,\dots, y^{n_2}$, where $y^i = f(u) v^i$.  The corresponding potential is  $F_{\mathsf w} = F + f \cdot \widehat F = F(u^1,\dots, u^{n_1}) + 
f(u) \widehat F \left(\frac{y^1}{f(u)}, \dots , \frac{y^{n_2}}{f(u)}\right)$.
\end{Lemma}

\begin{proof}
Let us rewrite the warp product metric \eqref{eq:wprod1} in coordinates $u, y$: 
$$
g_{ij} \ddd u^i \ddd u^j + f \,\widehat g_{\alpha\beta} \, \ddd v^\alpha \ddd v^\beta =
g_{ij} \ddd u^i \ddd u^j + f \,\widehat g_{\alpha\beta} \, \ddd \left(\frac{y^\alpha}{f}\right) \ddd\left(\frac{y^\beta}{f}\right) =
$$
$$
g_{ij} \ddd u^i \ddd u^j + f \,\widehat g_{\alpha\beta} \,  \left(\frac{\ddd y^\alpha}{f} -  \frac{y^\alpha \ddd f}{f^2}\right) \left(\frac{\ddd y^\beta}{f} -  \frac{y^\beta \ddd f}{f^2}\right) =
$$
\begin{equation}
\label{eq:wprod2}
\left(g_{ij} +    \frac{1}{f^3} \,\widehat g_{\alpha\beta}y^\alpha y^\beta  \frac{\partial f}{\partial u^i}\frac{\partial f}{\partial u^j}   \right)\ddd u^i \ddd u^j   
- 2 \frac{1}{f^2}  \widehat g_{\alpha\beta}y^\alpha\frac{\partial f}{\partial u^j} \ddd y^\beta \ddd u^j 
+ \frac{1}{f}\widehat g_{\alpha\beta} \ddd y^\alpha\ddd y^\beta.
 \end{equation}
 
 Then we compute the second derivatives of the function 
$F + f \cdot \widehat F$  (here  we use  $v^\alpha = \frac{y^\alpha}{ f(u)}$ and $\frac{\partial \widehat F}{\partial u^j} = \sum_\alpha \frac{\partial \widehat F}{\partial v^\alpha} \frac{\partial v^\alpha}{\partial u^j}=
- \sum_\alpha \frac{\partial \widehat F}{\partial v^\alpha} \frac{y^\alpha}{ f^2} \frac{\partial f}{\partial u^j}$ )  :
$$
\frac{\partial^2}{\partial u^i\partial u^j}\left(F + f \cdot \widehat F\right) = 
\frac{\partial^2 F}{\partial u^i\partial u^j} + \frac{\partial }{\partial u^i}\left( \frac{\partial f}{\partial u^j} \cdot \widehat F
+  f \cdot \frac{\partial \widehat F}{\partial u^j}\right) =
$$  
$$
\frac{\partial^2 F}{\partial u^i\partial u^j} + \frac{\partial }{\partial u^i}\left( \frac{\partial f}{\partial u^j} \cdot \widehat F
- \frac{1}{f} \sum_\alpha \frac{\partial \widehat F}{\partial v^\alpha} y^\alpha \frac{\partial f}{\partial u^j} \right)=  \mbox{\small (we use $\frac{\partial f }{\partial u^j}=\mathrm{const}$ as $f$ is linear)}
$$  
$$
\frac{\partial^2 F}{\partial u^i\partial u^j} + \frac{\partial f}{\partial u^j}  \left( - \sum_\alpha \frac{\partial \widehat F}{\partial v^\alpha} \frac{y^\alpha}{ f^2} \frac{\partial f}{\partial u^j} 
+ \frac{1}{f^2} \frac{\partial f}{\partial u^i} \cdot \sum_\alpha \frac{\partial \widehat F}{\partial v^\alpha} y^\alpha -
\frac{1}{f} \sum_\alpha \frac{\partial}{\partial u^i} \left(\frac{\partial \widehat F}{\partial v^\alpha}\right) y^\alpha
\right)=
$$  
$$
\frac{\partial^2 F}{\partial u^i\partial u^j} - \frac{\partial f}{\partial u^j}  \,  
\frac{1}{f} \sum_{\alpha,\beta} \frac{\partial^2 \widehat F}{\partial v^\beta \partial v^\alpha} \left( - \frac{1}{f^2} \right)   \frac{\partial f}{\partial u^i}  y^\alpha y^\beta=g_{ij} +  \frac{1}{f^3} \, \sum_{\alpha,\beta}\widehat g_{\alpha\beta}y^\alpha y^\beta  \frac{\partial f}{\partial u^i}\frac{\partial f}{\partial u^j}.
$$ 

Next,
$$
\frac{\partial^2}{\partial y^\beta \partial u^j}\left(F + f \cdot \widehat F\right) =
\frac{\partial }{\partial y^\beta}\left( \frac{\partial f}{\partial u^j} \cdot \widehat F
- \frac{1}{f} \sum_\alpha \frac{\partial \widehat F}{\partial v^\alpha} y^\alpha \frac{\partial f}{\partial u^j} \right) =
\frac{\partial f}{\partial u^j} \cdot \frac{\partial }{\partial y^\beta}\left(  \widehat F
- \frac{1}{f} \sum_\alpha \frac{\partial \widehat F}{\partial v^\alpha} y^\alpha  \right)=
$$  
$$
\frac{\partial f}{\partial u^j}\left(  \frac{\partial \widehat F}{\partial v^\beta} \frac{1}{f} - \frac{1}{f} \frac{\partial \widehat F}{\partial v^\beta} - \frac{1}{f}\sum_{\alpha} \frac{\partial^2 \widehat F}{\partial v^\alpha\partial v^\beta}  y^\alpha \frac{1}{f} 
\right) = -\frac{\partial f}{\partial u^j} \frac{1}{f^2} \sum_\alpha \widehat g_{\alpha\beta} y^\alpha.  
$$
And finally,
$$
\frac{\partial^2}{\partial y^\alpha \partial y^\beta}\left(F + f \cdot \widehat F\right) =
f \frac{\partial^2 \widehat F}{\partial y^\alpha \partial y^\beta} = f \frac{\partial^2 \widehat F}{\partial v^\alpha \partial v^\beta } \frac{1}{f^2} = \frac{1}{f} \,\widehat g^{\alpha\beta}.
$$

Comparing $\frac{\partial^2}{\partial u^i \partial u^j}\left(F + f \cdot \widehat F\right)$, $\frac{\partial^2}{\partial y^\alpha \partial u^j}\left(F + f \cdot \widehat F\right) $ and $\frac{\partial^2}{\partial y^\alpha \partial y^\beta}\left(F + f \cdot \widehat F\right)$  with \eqref{eq:wprod2} shows that $g_{\mathsf w}$ is Hessian in coordinates $(u,y)$, as stated. \end{proof}

This completes the proof of Proposition \ref{prop:keyprop}.  \end{proof}

Our goal is to show that every metric from Theorem \ref{thm:frobenius2a} is Frobenius in the corresponding coordinate system. It is seen from the construction of such a metric that it can be obtained step by step by iterating the warp product construction (following the rules presctribed by the graph $\mathsf F$).     At each step we have two constant curvature metrics $g$ and $\widehat g$ written in pro-Frobenius coordinate systems and then we `glue' them by using a certain linear function $f$ as described in Proposition  \ref{prop:keyprop}.   This function must satisfy property  (b1).  (The second property (b2) should be understood as a condition on the curvature $\widehat K$ of the metric $\widehat g$.)    In order to iterate the warp product procedure, we need to describe those linear funtions  $f_{\mathsf w} (u,y)$  which satisfy (b1) for the metric $g_{\mathsf w}$.  The next statement provides such a description.

\begin{Proposition}\label{prop:keyprop2}
In the notation of Proposition \ref{prop:keyprop}, consider the warp product metric $g_{\mathsf w} = g + \frac{1}{f} \, \widehat g$.

\begin{itemize}
\item[\rm 1.]  Let  $h(u)$ be another function (independent on  $f$) satisfying condition {\rm(b1)} for $g$, that is $\operatorname{grad}_g h = h(\beta - 4Ku)$ for some constant vector $\beta$. 
Then  $h(u)$ satisfies {\rm(b1)} for $g_{\mathsf w}$, namely
$$
\operatorname{grad}_{g_{\mathsf w}}  h = h \left( \begin{pmatrix} \beta \\ 0  \end{pmatrix}   - 4K  \begin{pmatrix} u \\ y  \end{pmatrix} \right).
$$
\item[\rm 2.]  The function $f(u)$ itself  satisfies {\rm(b1)} for $g_{\mathsf w}$ under the additional condition that
$\langle \alpha, \ddd f \rangle + 4K f(0)=0$, namely
$$
\operatorname{grad}_{g_{\mathsf w}}  f = f \left( \begin{pmatrix} \alpha \\ 0  \end{pmatrix}   - 4K  \begin{pmatrix} u \\ y  \end{pmatrix} \right).
$$
\item[\rm 3.]  Let $\widehat h(v)$ be a function satisfying conditions {\rm(b1)} for the metric $\widehat g$, that is, 
$\operatorname{grad}_{\widehat g} \widehat h = \widehat h ( \widehat \beta  - 4\widehat K v)$.
Then
$f(u)\widehat h(v)=f(u) \widehat h\bigl( \frac{y}{f} )$ satisfies {\rm(b1)} for $g_{\mathsf w}$, namely
$$
\operatorname{grad}_{g_{\mathsf w}} \bigl(  f(u) \widehat h(v)   \bigr) =f(u)  \widehat h(v)  \left( \begin{pmatrix} \alpha \\ \widehat\beta  \end{pmatrix}   - 4K  \begin{pmatrix} u \\ y  \end{pmatrix} \right).
$$    

{\rm(}Notice that $f(u)\widehat h(v)$ is linear in coordinates $(u, y)$. Indeed,   if $f(u)=m_0+\sum m_iu^i$ and $\widehat h(v) =  \widehat m_0 + \sum \widehat m_jv^j$, then
$f(u)\widehat h(v)=f(u)   (\widehat m_0 + \sum \widehat m_jv^j) = \widehat m_0 f(u) +  \sum \widehat m_jy^j$.{\rm)}

\end{itemize}

\end{Proposition}

\begin{proof}
1.   We first notice that $g(\ddd f, \ddd h) = \langle \ddd h , \operatorname{grad}_g f \rangle =  \langle \ddd h , f (\alpha - 4Ku) \rangle $ and, on the other hand,
$g(\ddd f, \ddd h) =  \langle \ddd f, \operatorname{grad}_g h  \rangle =  \langle \ddd f,  h(\beta - 4Ku) \rangle $. 
Since $f$ and $h$ are linear  (perhaps non-homogeneous), we have $\langle \ddd h, u\rangle = h- h_0$ and similarly  $\langle \ddd f, u\rangle = f- f_0$. Hence,
$$
g(\ddd f, \ddd h)=\langle \ddd f,  h(\beta - 4Ku) \rangle = h \langle \ddd f, \beta\rangle - 4Kh (f - f_0) = 
h (\langle \ddd f,\beta\rangle + 4Kf_0) - 4Khf,
$$
and, similarly,  
$$
g(\ddd f, \ddd h) = f (\langle \ddd h, \alpha\rangle + 4Kh_0) - 4Khf.
$$
It follows from this that $h (\langle \ddd f, \beta\rangle + 4K  f_0) =  f (\langle \ddd h, \alpha\rangle + 4Kh_0)$. Notice that the expressions in brackets are some constants.  Since $h$ and $f$ are not proportional, we conclude that these expressions are identically zero implying $g(\ddd f, \ddd h) = - 4Kfh$ and $\langle \ddd h, \alpha\rangle + 4Kh_0 =0$.

We now compute $\operatorname{grad}_{g_{\mathsf w}} h$:
$$
\operatorname{grad}_{g_{\mathsf w}} h = \begin{pmatrix}
g & \!\!\!\!\!\!\!\! (\alpha{-}4K u)\cdot y^\top \\
y \cdot (\alpha{-}4K u)^\top &  *
\end{pmatrix} \begin{pmatrix} \ddd h \\ 0 \end{pmatrix} = \begin{pmatrix}  h (\beta - 4Ku) \\  y \cdot (\alpha - 4K u)^\top \ddd h \end{pmatrix},
$$
where
$$
\begin{aligned}
y \cdot (\alpha - 4K u)^\top \ddd h &=  y\cdot (\alpha\, \ddd h + 4K h(0)) - y\cdot 4K (h(0) + \bigl(u^\top \ddd h)\bigr)  \\
&= y\cdot 0 - y \cdot 4K h = - y \cdot 4K h,
\end{aligned}
$$
which gives  
$
\operatorname{grad}_{g_{\mathsf w}} h  =  h \left(  \begin{pmatrix}  \beta \\ 0 \end{pmatrix}   - 4K \begin{pmatrix}  u \\ y \end{pmatrix}\right)
$,
as required.

2. The proof is just the same,  but we need to use $\alpha \,\ddd f + 4K f(0)=0$ as an additional condition.

3. We now compute the gradient of 
$f(u)\widehat h(v)=f(u) \widehat h\bigl( \frac{y}{f} )=f(u)   (\widehat m_0 + \sum \widehat m_jv^j) = \widehat m_0 f(u) +  \sum \widehat m_jy^j= \widehat h(0) f(u) + \widehat h(y) - \widehat h(0)  $.
$$
\operatorname{grad}_{g_{\mathsf w}} \bigl(\widehat h(0) f(u) + \widehat h(y) - \widehat h(0) \bigr) = \begin{pmatrix}
g &  \operatorname{grad}_g f \cdot v^\top \\ 
v \cdot (\operatorname{grad}_g f)^\top &  g(\ddd f,\ddd f) v \cdot v^\top + f \widehat g
\end{pmatrix}  \begin{pmatrix} \widehat h(0) \ddd f  \\  \ddd \widehat h \end{pmatrix} =
$$
$$
\begin{pmatrix}
 \left(\widehat h(0) + v^\top \ddd \widehat h\right)  \operatorname{grad}_g f   \\
g(\ddd f, \ddd f) \left(\widehat h(0) +  v^\top \ddd\widehat h \right) v   + f \operatorname{grad}_{\widehat g} h(v)
\end{pmatrix} =
$$
$$
\begin{pmatrix}
 h(v) f(u) \bigl(\alpha - 4Ku\bigr)    \\
g(\ddd f, \ddd f) h(v) v   + f(u) h(v) \bigl( \widehat \beta - 4 \widehat Kv\bigr)
\end{pmatrix} =
$$
$$
\begin{pmatrix}
h(v) f(u) \bigl(\alpha - 4Ku\bigr)    \\
(4\widehat K f - 4K f^2)  h(v) v   + f(u) h(v) \bigl( \widehat \beta - 4 \widehat Kv\bigr)
\end{pmatrix} =
 h(v) f(u) \begin{pmatrix} \alpha - 4Ku \\ \widehat \beta - 4Ky \end{pmatrix},
$$
as required. \end{proof}

Propositions \ref{prop:keyprop}  and  \ref{prop:keyprop2}  allow us to construct pro-Frobenius metrics using metrics from extended AFF-pencils as building blocks.   We only need to describe appropriate function $f$ satisfying properties (b1) and (b2) from Proposition \ref{prop:keyprop}  for these metrics.    

\begin{Proposition}\label{prop:keyprop3}
Consider  metrics  $g=P(L)g_{\LC}$ and $\widehat g = \widehat P(\widehat L) \widehat g_{\LC}$ from two different extended AFF-pencils \eqref{eq:AFFextended}  with pro-Frobenius coordinates $u^1,\dots, u^{n_1}$ and $v^1,\dots, v^{n_2}$ respectively.
Let $\lambda$ be a root of $P(\cdot)$,  then

\begin{enumerate}

\item $f=\det (\lambda \Id-L)$ satisfies condition {\rm (b1)} , namely  $\operatorname{grad} f = f (\alpha - 4K u)$, where $K= \frac{1}{4}b_{n_1}$,
$\alpha =  (b_{n_1-1}, \dots , b_1, b_0)$ and $b_i$'s are defined from $P(t)=(t-\lambda)(b_0 + b_1t + \dots + b_{n_1} t^{n_1})$. 

\item If $\widehat P(t) = \widehat a_{n_2 + 1} t^{n_2+1} + \widehat a_{n_2} t^{n_2} + \dots$, then condition  {\rm (b2)}  for $f$ takes the form 
$\widehat a_{n_2 + 1} =  P'(\lambda)$.

\item Moreover,   if $\lambda$ is a double root of $P(\cdot)$, then  $\langle \alpha,\ddd f\rangle + 4 K f(0) = 0$ (cf. item 2 of Proposition  \ref{prop:keyprop2}).

\end{enumerate}
\end{Proposition} 

\begin{proof}
The first part of this statement can be verified by a straightforward computation.  We also notice that $f(u)$ is linear in variables $u^1,\dots, u^{n_1+1}$. 

The second statement  easily follows from Fact \ref{Fact:cas}.  
Indeed, from \eqref{eq:length}  we get
\begin{equation}\label{eq:111}
\frac{1}{f} g(\ddd f, \ddd f) = - P'(\lambda) - 4K f.
\end{equation}
Therefore (b2) can be written as $4 \widehat K  = - P'(\lambda)$.  In view of Fact \ref{Fact:cas},  $\widehat K = -\frac{1}{4} \widehat a_{n_2 + 1}$. Hence, (b2)  amounts to $\widehat a_{n_2 + 1} =  P'(\lambda)$, as stated.

Finally, $\operatorname{grad} f = f (\alpha - 4K u)$ implies 
$$
\begin{aligned}
\frac{1}{f} g(\ddd f, \ddd f) &= \frac{1}{f} \langle \operatorname{grad} f , \ddd f\rangle =  \langle \alpha, \ddd f\rangle - 4 K \langle u, \ddd f \rangle \\
&= 
\langle \alpha, \ddd f\rangle - 4 K (f - f(0))=
\langle \alpha, \ddd f\rangle + 4K f(0)  - 4 K f.
\end{aligned}
$$
Comparing this relation with \eqref{eq:111}, we see that
$\langle \alpha, \ddd f\rangle + 4K f(0)  =  - P'(\lambda)$.
Hence, the l.h.s. of this relation vanishes if and only if $\lambda$ is a double root of $P(t)$.  
\end{proof}

We now `return' to the notation from Section \ref{subsect:4.1}. It follows from Propositions \ref{prop:keyprop} and \ref{prop:keyprop3} that the metric $P_\alpha(L_\alpha) g^{\LC}_\alpha + \frac{1}{\det(\lambda_\beta \Id-L)}
P_\beta(L_\beta) g^{\LC}_\beta$ is pro-Frobenius in coordinates $\sigma_\alpha^1, \dots, \sigma_\alpha^{n_\alpha}, 
\det(\lambda \Id-L)\cdot \sigma_\beta^1, \dots  , \det(\lambda \Id-L)\cdot \sigma_\beta^{n_\beta}$,  if 
$\lambda_\beta$ is a root of $P_\alpha(t)$ and $\overset{\beta}{a}_{n_\beta + 1} = (-1)^{n_\alpha} P_\alpha'(\lambda_\beta)$.  These two conditions exactly coincide with condition (ii) used in the construction of multi-block pencils from  Theorem \ref{thm:frobenius2a} (cf. Section \ref{subsect:4.3} where we discuss the two-block case in slightly different notation).  We can next repeat this construction by `adding' one more similar block  $P_\gamma(L_\gamma) g_\gamma^{\LC}$ to $P_\alpha(L_\alpha) \,g_\alpha^{\LC}$:
$$
P_\alpha(L_\alpha) \,g_\alpha^{\LC} + \frac{1}{\det(\lambda_\beta \Id -L_\alpha)}
P_\beta(L_\beta)\,  g_\beta^{\LC}   +  \frac{1}{\det (\lambda_\gamma \Id -L_\alpha)}   
P_\gamma(L_\gamma) \, g_\gamma^{\LC}.  
$$  
Of course, two conditions should be fulfilled, namely $P_\alpha(\lambda_\gamma)=0$ and $\overset{\gamma}a_{n_\gamma + 1} = (-1)^{n_\alpha} P_\alpha'(\lambda_\gamma)$.   Moreover, in view of Proposition \ref{prop:keyprop2} (item 2) and Proposition \ref{prop:keyprop2} (item 3),  if $\lambda_\beta = \lambda_\gamma$ we need to require that $\lambda_\beta$ is a double root of $P_\alpha$  (which is exactly condition (iii) used in the construction of multi-block pencils from  Theorem \ref{thm:frobenius2a}).

To complete the proof of Theorem \ref{thm:frobenius2a},  we should now iterate the above procedure step-by-step following the  combinatorial data provided by the graph $\mathsf F$ (starting from leaves and moving towards the root).  At each step of this construction we obtain a pencil of pro-Frobenius metrics, leading finally to a Frobenius pencil  (the condition (i) will guarantee the flatness of the final metric).

\subsection*{Statements and Declarations}

The authors have no relevant financial or non-financial interests to disclose.
The authors have no competing interests to declare that are relevant to the content of this article.

\end{document}